\documentclass[reqno]{amsart}

\usepackage{amsmath,amsthm,amsfonts,amssymb,young}
\usepackage{array, boldline, makecell, booktabs}
\usepackage{bm}
\usepackage{ytableau}
\usepackage{euscript, mathrsfs} 
\usepackage{tikz}
\usetikzlibrary{backgrounds}
\usepackage{amsmath}
\usepackage{mathtools}
\usepackage[colorlinks]{hyperref}
\usepackage{cleveref}
\usepackage[margin=1in]{geometry}
\numberwithin{equation}{section}    
\usepackage[all]{xy}
\usepackage{graphicx,import}
\usepackage{amscd}
\usepackage{color}
\usepackage{enumerate}
\usepackage{fourier}
\usepackage{tikz}
\usepackage{cases}
\usepackage{cite}
\usepackage[T1]{fontenc}

\newtheorem{thm}{Theorem}[section]
\newtheorem{lem}[thm]{Lemma}
\newtheorem{proposition}[thm]{Proposition}
\newtheorem{corollary}[thm]{Corollary}
\theoremstyle{definition}
\newtheorem{example}[thm]{Example}
\newtheorem{definition}[thm]{Definition}

\newtheorem{rmk}[thm]{Remark}

\newtheorem{question}[thm]{Question}

\DeclareMathOperator{\mix}{mix}

\DeclareMathOperator{\cPF}{cPF}
\DeclareMathOperator{\pdinv}{pdinv}
\DeclareMathOperator{\ldinv}{ldinv}
\DeclareMathOperator{\cont}{cont}

\DeclareMathOperator{\PF}{PF}

\DeclareMathOperator{\aseq}{aseq}
\DeclareMathOperator{\PTab}{PTab}

\DeclareMathOperator{\h}{\mathfrak{h}}
\DeclareMathOperator{\hhat}{\hat{\mathfrak{h}}}
\DeclareMathOperator{\hb}{\bar{\mathfrak{h}}}

\DeclareMathOperator{\SR}{Reg}
\DeclareMathOperator{\CH}{\mathbf{H}}
\DeclareMathOperator{\WD}{WD}
\DeclareMathOperator{\Reg}{Reg}
\DeclareMathOperator{\up}{up}
\DeclareMathOperator{\touch}{touch}

\DeclareMathOperator{\bo}{bot}
\DeclareMathOperator{\piv}{piv}

\DeclareMathOperator{\area}{area}

\DeclareMathOperator{\adj}{adj}

\DeclareMathOperator{\dinv}{dinv}

\DeclareMathOperator{\stat}{stat}

\DeclareMathOperator{\DR}{DR}

\DeclareMathOperator{\SW}{SW}

\DeclareMathOperator{\pos}{pos}

\title{Shuffle theorem for torus link homology}

\author{Donghyun Kim}
\address{Department of Mathematical Sciences \\ Seoul National University \\Seoul 151-247 \\ Korea}
\email{hyun920310@snu.ac.kr}

\author{Jaeseong Oh}
\address{Department of Mathematics \\ Sungkyunkwan University \\ Suwon \\ Korea}
\email{jaeseongoh@skku.edu}

\begin{document}

\begin{abstract}
We prove that the symmetric function $e_{(1^k)}[-MX^{m,n}] \cdot 1$, arising from the elliptic Hall algebra, equals the generating function for $k$-tuples of cyclic $(m,n)$-parking functions. This result resolves a conjecture of Gorsky--Mazin--Vazirani and Wilson, establishing that the elliptic Hall algebra governs the Khovanov--Rozansky homology of torus links $T(km,kn)$. Consequently, this provides an affirmative answer to a question of Galashin and Lam in the torus link case. As a key step in the proof, we develop a rational analogue of the Shareshian--Wachs involution originally introduced to prove the symmetry property of the chromatic quasisymmetric functions. 
\end{abstract}

\maketitle

\section{Introduction}\label{Sec: intro}
\subsection{Overview}

The \emph{nabla operator} $\nabla$, introduced by Bergeron and Garsia~\cite{BG99, BGHT99}, is a linear operator acting on symmetric functions with coefficients in $\mathbb{Q}(q,t)$. It is defined as the eigen-operator for the modified Macdonald polynomials $\widetilde{H}_\mu[X;q,t]$ \cite{Mac88,GH93},
\(
    \nabla \widetilde{H}_\mu[X;q,t] = T_\mu \widetilde{H}_\mu[X;q,t],
\)
where $T_\mu \coloneqq \prod_{(i,j)\in\mu}t^{i-1}q^{j-1}$. 

The significance of $\nabla$ lies in its deep connection to the \emph{diagonal coinvariant ring}. This ring is the symmetric group $\mathfrak{S}_k$-module defined by
\[
    \DR_k \coloneqq \mathbb{C}[\mathbf{x}_k,\mathbf{y}_k]/\langle\mathbb{C}[\mathbf{x}_k,\mathbf{y}_k]^{\mathfrak{S}_k}_+\rangle,
\]
where $\mathfrak{S}_k$ acts diagonally on the polynomial ring $\mathbb{C}[\mathbf{x}_k,\mathbf{y}_k]=\mathbb{C}[x_1,\dots,x_k,y_1,\dots,y_k]$ in $2k$ variables. Haiman's celebrated proof of the \emph{$(n+1)^{n-1}$ conjecture}~\cite{Hai02} established that the Frobenius character of $\DR_k$ is given by $\nabla e_k$, where $e_k$ is the elementary symmetric function of degree $k$.

The \emph{shuffle theorem}, conjectured by Haglund et al.~\cite{HHLRU05} and proven by Carlsson and Mellit~\cite{CM18}, provides a combinatorial formula for $\nabla e_k$. This formula expresses $\nabla e_k$ as a weighted sum over parking functions:
\[
    \nabla e_k = \sum_{(P,f) \in \PF_{k}} q^{\dinv(P,f)}t^{\area(P)} x^{f},
\]
where $\PF_k$ denotes the set of parking functions of size $k$.

The shuffle theorem exemplifies the deep interplay between algebra, geometry, and combinatorics in modern algebraic combinatorics. Since its proof, the action of $\nabla$ on other symmetric functions has been a subject of intense study. For instance, the Elias--Hogancamp formula~\cite{EH19, GH22} provides a combinatorial formula for $\langle\nabla e_{(1^k)},e_k\rangle$. More recently, Blasiak, Haiman, Morse, Pun, and Seelinger~\cite{BHMPS25LW}, and independently the authors~\cite{KO24}, proved the Loehr--Warrington conjecture~\cite{LW08}, establishing a combinatorial formula for $\nabla s_\lambda$, where $s_\lambda$ is the \emph{Schur function}.

The landscape of the shuffle theorem was further broadened by the introduction of the `rational analogue'. Mellit~\cite{Mel21} generalized the shuffle theorem to the \emph{rational (compositional) shuffle theorem} for $(m,n)$-parking functions with coprime $(m,n)$.
Building on the theory of \emph{elliptic Hall algebra} ${\mathcal{E}_{q,t}}$ (also known as the Schiffmann algebra) of Burban and Schiffmann~\cite{BS12}, Blasiak et al.~\cite{BHMPS23, BHMPS25LW} established a unified framework for `rationalizing' shuffle theorems. Within this framework, they proved the rational version of the Loehr--Warrington conjecture involving the elliptic Hall algebra element $s_\lambda[-MX^{m,n}]\in\mathcal{E}_{q,t}$ (notation explained in Section~\ref{subsec: EHA}). In this context, the algebraic side of the rational shuffle theorem corresponds to the action of the operator $e_k[-MX^{m,n}]\in\mathcal{E}_{q,t}$.

While these foundational results provide combinatorial formulas for a broad class of symmetric functions, they do not address the distinct combinatorial structure arising from the product of elementary symmetric functions $e_{(1^k)}[-MX^{m,n}] \cdot 1$, namely, the rational analogue of the Elias--Hogancamp formula.

The main theorem of this paper resolves this issue. Building on partial conjectures of Gorsky, Mazin, and Vazirani~\cite{GMV17}, Wilson~\cite{Wil18,Wil23} conjectured that the symmetric function $e_{(1^k)}[-MX^{m,n}] \cdot 1$ is the combinatorial generating function for $k$-tuples of cyclic $(m,n)$-parking functions. We refer the reader to Section~\ref{Sec: rcs} for the relevant combinatorial definitions.

\begin{thm}(\cite[Conjecture 5.1]{Wil23})\label{thm: main}
    For coprime positive integers $(m,n)$ and a positive integer $k$, we have
    \begin{equation*}
        e_{(1^k)}[-MX^{m,n}]\cdot1 = (1-t)^k\sum_{\pi^\bullet \in\cPF^k_{m,n}} q^{\stat(\pi^\bullet)} t^{\area(P_{\pi^\bullet})} x^{f_{\pi^{\bullet}}}.  
    \end{equation*}
\end{thm}

\subsection{Connection to Knot Homology}

Notably, Wilson and Gorsky--Mazin--Vazirani conjectured Theorem~\ref{thm: main}, motivated by their study of knot homology. This connection deserves further elaboration. Given a link $L$, recording the graded dimension of its Khovanov--Rozansky (KR) homology~\cite{KR08a,KR08b,Kho07} yields a Laurent polynomial $\mathcal{P}^{\operatorname{KR}}_L(a,q,t)$ in $a, q^{1/2},t^{1/2}$, known as the \emph{knot superpolynomial}~\cite{DGR06}. If the odd cohomology of the KR homology vanishes, the superpolynomial becomes a Laurent polynomial in $a,q,t$. When specializing at $t=q^{-1}$, it recovers the HOMFLY polynomial of $L$.

The Khovanov--Rozansky homology and the knot superpolynomial provide a central interface between diverse mathematical disciplines. These invariants encode the geometry of braid varieties and cluster algebras, emerge from Hilbert schemes and compactified Jacobians through the Oblomkov--Rasmussen--Shende conjectures, and interact deeply with the elliptic Hall algebra, rational Cherednik algebras, and generalized Catalan combinatorics~\cite{Che13, GM13, GORS14, GN15, CD16, CD17, OR17, ORS18, CGGS21, GNR21, Mel21, GLSBS22, GL23, CGHM24, GL24, CGGLLS25}.

Recently, Galashin and Lam~\cite{GL23} posed a question relating the elliptic Hall algebra to knot superpolynomials. A \emph{curve} is defined as the graph of a strictly increasing continuous function $f : [0, m] \rightarrow [0, n]$ satisfying $f(0) = 0$ and $f(m) = n$. To a curve $C$, one associates a link $L_C$ in the torus $\mathbb{T}=\mathbb{R}^2/\mathbb{Z}^2$. They also constructed an elliptic Hall algebra element $D_C$ associated with the curve $C$. Using this element, one defines a symmetric function $F_C$ and the \emph{EHA superpolynomial} $\mathcal{P}^{\mathcal{E}}_C(a,q,t)$ using plethystic substitution by
\[
    F_C:=D_C\cdot1\left[\dfrac{X}{1-t}\right],\qquad 
    \mathcal{P}^{\mathcal{E}}_{C}(a,q,t):= F_C[a-a^{-1}].
\]

\begin{question}[\cite{GL23}]\label{conj: Galashin--Lam} 
For which curves $C$ do we have
\begin{equation}\label{eq: GL conjecture}
     \mathcal{P}^{\mathcal{E}}_{C}(a,q,t)= (1-t)^{k(C)-1}\mathcal{P}^{\operatorname{KR}}_{L_C}(a,q,t),
\end{equation}
where $k(C)$ denotes the number of components of the link $L_C$?
\end{question}

The correspondence proposed by Galashin and Lam is consistent with established results. For torus knots, the relationship follows from the earlier work of Mellit~\cite{Mel21} combined with Gorsky and Neguț~\cite{GN15}. More recently, Caprau, Gonzalez, Hogancamp, and Mazin~\cite{CGHM24} extended this verification to the case of \emph{Coxeter knots} by utilizing the results of \cite{BHMPS23} on the shuffle theorem under any line.

Despite the verification for torus and Coxeter knots, an explicit formula for the elliptic Hall algebra action corresponding to torus links has not been established. On the homological side, Gorsky, Mazin, and Vazirani~\cite{GMV17} and Wilson~\cite{Wil23} showed that the KR homology of the torus link $T(km, kn)$ is described by the combinatorics of $k$-tuples of cyclic $(m,n)$-parking functions. They further conjectured that this combinatorial structure arises algebraically from the action of $e_{(1^k)}[-MX^{m,n}]$.

Theorem~\ref{thm: main} confirms this conjecture, answering the question of Galashin--Lam for the torus link case. In light of this correspondence, Theorem~\ref{thm: main} can be regarded as the \emph{shuffle theorem for torus link homology}.

\begin{corollary}\label{Cor: torus link superpolynomial = EHA superpolynomial}
    For the straight line segment $C(km,kn)$ connecting $(0,0)$ to $(km,kn)$, the associated link is the torus link $T(km,kn)$, and the identity \eqref{eq: GL conjecture} holds:
    \[
         \mathcal{P}^{\mathcal{E}}_{C(km,kn)}(a,q,t)= (1-t)^{k-1}\mathcal{P}^{\operatorname{KR}}_{T(km,kn)}(a,q,t).
    \]
\end{corollary}

For a certain class of curves called $\mathbb{Z}$-convex, Galashin and Lam~\cite{GL23} conjectured a generalization of the previous conjecture by BHMPS~\cite{BHMPS23} regarding the Schur positivity of the symmetric function $F_C/(1-t)^{k(C)-1}$. Since the right-hand side of Theorem~\ref{thm: main} (also \eqref{eq: Wilson compact}) can be written as a positive sum of LLT polynomials (see~\cite{Wil23}), this implies the Schur positivity of $F_{C(km,kn)}/(1-t)^{k-1}$, confirming the Schur positivity conjecture of~\cite{GL23,BHMPS23} for the case of $C=C(km,kn)$.

\subsection{Main Results: Jacobi--Trudi Identities and Combinatorial Operators}

Our proof of Theorem~\ref{thm: main} relies on a fundamental reformulation of the rational shuffle theorem. We introduce a new determinantal approach to the algebraic operators. Specifically, this approach hinges on defining a family of operators, denoted by $\mathfrak{h}$, which generate `strictly increasing' cyclic $(m,n)$-parking functions (see Sections~\ref{Sec: rcs} and \ref{Sec: wil} for precise definitions). We then establish a new formulation for the rational compositional shuffle theorem, expressing the algebraic term $C_{\alpha}[-MX^{m,n}] \cdot 1$ as the determinant of a matrix whose entries are given by these $\mathfrak{h}$-operators (Lemma~\ref{lem: C alpha operator expression}).

This determinantal structure proves to be highly versatile; it allows us to derive a corresponding Jacobi--Trudi identity for the general Schur function action $s_\lambda[-MX^{m,n}] \cdot 1$. As a consequence, we provide a new, elementary proof of the Loehr--Warrington conjecture for $\nabla^m s_\lambda$ (corresponding to the case $n=1$). In particular, we give a concrete combinatorial tie between the compositional shuffle formula for $\nabla^m C_{\alpha}$ and the Loehr--Warrington formula for $\nabla^m s_\lambda$. We also successfully express $e_{(1^k)}[-MX^{m,n}] \cdot 1$ in terms of $k$-tuples of cyclic $(m,n)$-parking functions.

The combinatorial heart of our argument, presented in Section~\ref{Sec: sw}, lies in establishing the necessary relations between the $\mathfrak{h}$-operators. To achieve this, we introduce an \emph{$(m,n)$-generalization of the Shareshian--Wachs involution}. Originally developed to prove the symmetry of chromatic quasisymmetric functions \cite{SW16}, this involution is adapted here to address the combinatorics of cyclic parking functions, thereby proving the key relations required for our determinantal formulas.

\section*{acknowledgement}
The authors are grateful to Eugene Gorsky, Thomas Lam, Anton Mellit, Andrei Negu\c{t}, Brendon Rhoades, George Seelinger, Joshua Wen, and Andy Wilson for helpful conversations. J. Oh was supported by NRF grant RS-2025-16067413.

\section{Preliminaries}\label{Sec: pre}
\subsection{Symmetric functions}
A \emph{composition} is a sequence 
\(
    \alpha = (\alpha_1, \alpha_2, \dots, \alpha_k)
\)
of positive integers (called \emph{parts}) and we say that $\alpha$ is a composition of $n$, denoted by $\alpha\models n$, if $|\alpha| = \alpha_1 + \alpha_2 + \cdots + \alpha_k$ equals $n$.
A \emph{partition} is a composition whose parts are weakly decreasing.
We write $\lambda \vdash n$ if $|\lambda|=n$.
We often identify a partition $\lambda$ with its \emph{Young diagram}, defined as the set of cells
\(
    \lambda = \{(i,j) \in \mathbb{Z}_+ \times \mathbb{Z}_+ : 1 \leq i \leq k, \, 1 \leq j \leq \lambda_i\}.
\)
The \emph{conjugate partition} $\lambda^{t} = (\lambda^{t}_1, \lambda^{t}_2, \dots)$ is obtained by reflecting the Young diagram of $\lambda$ across the main diagonal $y = x$. For partitions $\lambda$ and $\mu$ such that $\mu \subseteq \lambda$ (as sets of cells), the \emph{skew partition} $\lambda/\mu$ is defined as the set difference $\lambda \setminus \mu \subset \mathbb{Z}_+ \times \mathbb{Z}_+$.

Let $\Lambda$ be the algebra of symmetric functions in an infinite set of variables $X = \{x_1, x_2, \dots\}$ with coefficients in the field $\mathbb{Q}(q,t)$. We follow standard notation for the various bases of $\Lambda$: $e_\lambda$, $h_\lambda$, $p_\lambda$, $m_\lambda$, and $s_\lambda$ denote the elementary, complete homogeneous, power sum, monomial, and Schur symmetric functions, respectively. The involution $\omega: \Lambda \to \Lambda$ is the algebra automorphism defined on the Schur basis by $\omega(s_{\lambda}) = s_{\lambda^{t}}$.

We equip $\Lambda$ with the \emph{Hall inner product} $\langle \cdot, \cdot \rangle$ by requiring that the Schur functions form an orthonormal basis; that is, $\langle s_\lambda, s_\mu \rangle = \delta_{\lambda,\mu}$. For a symmetric function $f$, we let $f^\bullet$ denote the operator of multiplication by $f$. We denote by $f^\perp$ the adjoint of $f^\bullet$ with respect to the Hall inner product.

We introduce some additional notation for \emph{plethystic substitutions}. If $A$ is an expression in terms of indeterminates (including $q, t$), we define $p_k[A]$ to be the result of substituting $a^k$ for every indeterminate $a$ occurring in $A$. Furthermore, for any $f \in \Lambda$, the plethysm $f[A]$ is defined by substituting $p_k[A]$ for $p_k$ in the power sum expansion of $f$. By convention, the alphabet $X$ is identified with the sum $x_1 + x_2 + \dots$, so that $f[X]$ is the standard symmetric function $f(x_1, x_2, \dots)$.

Haglund, Morse, and Zabrocki \cite{HMZ12} defined an operator $\mathbf{C}_a$ for a positive integer $a$ acting on $\Lambda$ by:
\begin{equation*}
    \mathbf{C}_a f[X]=(-q)^{1-a}f[X-(q-1)/(qz)]\sum_{m\geq 0}z^{m}h_m[X]\vert_{z^a}.
\end{equation*}
For a composition $\alpha=(\alpha_1,\dots,\alpha_{k})$, we define
\begin{equation*}
C_{\alpha}=\mathbf{C}_{\alpha_k}\cdots\mathbf{C}_{\alpha_2}\cdot\mathbf{C}_{\alpha_1}\cdot 1.
\end{equation*}
The symmetric function $C_{\alpha}$ played a key role in refining the shuffle conjecture as the compositional shuffle conjecture \cite{HMZ12}, which eventually led to the rational compositional shuffle theorem \cite{Mel21}.

\subsection{Elliptic Hall algebra}\label{subsec: EHA}

Let us collect some definitions and results concerning the \emph{elliptic Hall algebra} $\mathcal{E}_{q,t}$ as introduced in \cite{BS12}. We follow the notation of \cite{BHMPS21Delta}. 

The elliptic Hall algebra $\mathcal{E}_{q,t}$ is generated by subalgebras $\Lambda(X^{m,n})$ isomorphic to the algebra of symmetric functions $\Lambda$, indexed by pairs of coprime integers $(m,n)$, along with a central Laurent polynomial subalgebra $\mathbb{Q}(q,t)[c_1^{\pm1}, c_2^{\pm1}]$, subject to the relations described in \cite[Section 3.3]{BHMPS21Delta}. 

We now recall a natural action of the elliptic Hall algebra on the ring of symmetric functions constructed in \cite{FT11,SV13}. Let $\widetilde{H}_\mu[X;q,t]$ denote the \emph{modified Macdonald polynomials} \cite{Mac88}. The \emph{diagram generator} $B_\lambda(q,t)$ is defined by
 $\sum_{(i,j)\in \lambda} q^{i-1}t^{j-1}$. For a symmetric function $f \in \Lambda$, let $f[B]$ and $f[\bar{B}]$ denote the operators defined by their action on the basis $\{\widetilde{H}_\mu\}$ as follows:
\[
    f[B]\widetilde{H}_\mu = f[B_\mu(q,t)]\widetilde{H}_\mu, \qquad f[\bar{B}]\widetilde{H}_\mu = f[\overline{B_\mu(q,t)}]\widetilde{H}_\mu,
\]
where $\overline{B_\mu(q,t)} = B_\mu(q^{-1},t^{-1})$. The following proposition summarizes the action of $\mathcal{E}_{q,t}$ on $\Lambda$.
\begin{proposition}\cite[Proposition 3.1]{BHMPS23}\label{Prop: E action on Lambda} 
There is an action of $\mathcal{E}_{q,t}$ on $\Lambda$ defined on generators as follows:
\begin{enumerate}
    \item $c_1 \mapsto 1$, $c_2 \mapsto (qt)^{-1}$.
    \item $f(X^{1,0}) \mapsto (\omega f)\left[B-\frac{1}{M}\right]$, \quad $f(X^{-1,0}) \mapsto (\omega f)\left[\overline{\frac{1}{M}-B}\right]$.
    \item $f(X^{0,1}) \mapsto f\left[-\frac{X}{M}\right]^\bullet$, \quad $f(X^{0,-1}) \mapsto f[X]^\perp$.
\end{enumerate}
\end{proposition}

The \emph{nabla operator} $\nabla$, defined in \cite{BGHT99}, is the diagonal operator with the modified Macdonald polynomials as eigenfunctions:
\[
    \nabla \widetilde{H}_\mu = t^{n(\mu)}q^{n(\mu')}\widetilde{H}_\mu,
\]
where $n(\mu)=\sum_i (i-1)\mu_i$. While the operator $\nabla$ does not directly correspond to an action of an element of  $\mathcal{E}_{q,t}$ on $\Lambda$, it satisfies the following commutation relation with the generators:
\begin{lem}\label{Lem: nabla conjugation}\cite[Lemma 3.4.1]{BHMPS23} 
Conjugation by $\nabla$ acts as a symmetry of the $\mathcal{E}_{q,t}$-action on $\Lambda$. Specifically, we have
    \[
    \nabla f(X^{m,n})\nabla^{-1} = f(X^{m+n,n}).
    \]
\end{lem}

Combining Proposition~\ref{Prop: E action on Lambda} and Lemma~\ref{Lem: nabla conjugation} yields the following identity: for a symmetric function $f$,
\begin{equation}\label{eq: m,1 = nabla^m}
    f[-MX^{m,1}]\cdot 1 = \nabla^m f.
\end{equation}
Consequently, the algebraic sides of the shuffle theorem (concerning $\nabla^m e_n$) and the Loehr--Warrington conjecture (concerning $\nabla^m s_\lambda$) can be seen as special cases of certain elliptic Hall algebra action on $1$.

Furthermore, symmetric functions appearing in the rational compositional shuffle theorem \cite{Mel21}, rational version of Loehr--Warrington formula \cite{BHMPS25LW}, and the conjecture of Wilson, Gorsky--Mazin--Vazirani (Theorem~\ref{thm: main}) can also be expressed via the elliptic Hall algebra action on $1$, as $C_{\alpha}[-MX^{m,n}]$, $s_\lambda[-MX^{m,n}]$, and $e_{(1^k)}[-MX^{m,n}]$, respectively. This relies on the identification in \cite[Proposition 6.7]{Neg14}. 

\subsection{Labeled lattice paths}
In this subsection, we define a common generalization that unifies the tuples of cyclic parking functions introduced in \cite{Wil23} (appearing in Theorem~\ref{thm: main}) with the $(km, kn)$-parking functions arising in the rational (compositional) shuffle theorem \cite{Mel21}.

 Let $m$ and $n$ be coprime positive integers.
An \emph{$(m,n)$-lattice path} is a path of unit north $(0,1)$ and east $(1,0)$ steps, starting at $(-a, 0)$ for some integer $a \ge 0$.
The path must begin with a north step, contain exactly $n$ north steps, and stay weakly above the line $my=nx$.

For a positive integer $k \ge 1$, we view $(\mathbb{Z}^2)^k$ as a collection of $k$ copies (or \emph{sheets}) of the lattice $\mathbb{Z}^2$, indexed by $i \in \{1, \dots, k\}$.
A \emph{cell} is a unit square in one of these sheets.
The \emph{content} of a cell $c$ in the $i$-th sheet is defined as
\[
\operatorname{cont}(c) = k(my - nx) + i-1,
\]
where $(x,y)$ denotes the coordinates of the south-east corner of $c$.

A \emph{labeled $(m,n)$-lattice paths tuple of order $k$} is a sequence $\pi^\bullet=(\pi^{(1)}, \dots, \pi^{(k)})$, where each component $\pi^{(\ell)}=(P^{(\ell)}, f^{(\ell)})$ consists of an $(m,n)$-lattice path $P^{(\ell)}$ and a labeling function $f^{(\ell)}$ that assigns a positive integer to each north step of $P^{(\ell)}$. The tuple is equipped with a \emph{connecting permutation} $\phi \in \mathfrak{S}_k$ (not necessarily unique) and must satisfy the following conditions:

\begin{enumerate}
    \item (Boundary Condition) The end points match cyclically according to $\phi$, such that
    \[
        \operatorname{end}(\pi^{(\ell)}) = \operatorname{start}(\pi^{(\phi(\ell))}) + (m,n) \quad \text{for all } 1 \le \ell \le k,
    \]
    where $\operatorname{start}(\pi)$ and $\operatorname{end}(\pi)$ denote the starting point and end point of the path $\pi$, respectively, and addition is coordinate-wise.
    
    \item (Labeling Condition) Let $N_i^{(\ell)}$ denote the $i$-th north step of the component $\pi^{(\ell)}$. If $N_i^{(\ell)}$ and $N_{i+1}^{(\ell)}$ are consecutive steps in the path (i.e., they are geometrically adjacent vertical steps), then their labels must strictly increase:
    \[
        f^{(\ell)}(N_i^{(\ell)}) < f^{(\ell)}(N_{i+1}^{(\ell)}).
    \]
    This condition extends to the boundary: we define the virtual step $N_{n+1}^{(\ell)}$ as the $(m,n)$-translation of $N_1^{(\phi(\ell))}$, carrying the label $f^{(\phi(\ell))}(N_1^{(\phi(\ell))})$.
\end{enumerate}

Each component $\pi^{(\ell)}$ is typically embedded in the $\ell$-th sheet within the strip $0 \le y \le n$. The tuple $\pi^\bullet$ is commonly visualized in one of two ways:
\begin{itemize}
    \item By drawing all components $\pi^{(\ell)}$ together in a single fundamental strip $0 \le y \le n$ (while keeping track of the component indices), or
    \item By vertically stacking the components such that $\pi^{(\ell)}$ lies in the strip $(\ell-1)n \le y \le \ell n$.
\end{itemize}

A set of labeled $(m,n)$-lattice paths tuples of order $k$ is denoted by $\PF^k_{m,n}$. For a tuple $\pi^\bullet \in \PF^k_{m,n}$, we denote the underlying path tuple by $P_{\pi^\bullet}=(P^{(1)}, \dots, P^{(k)})$ and the collective labeling by $f_{\pi^\bullet}$. Specifically, for a north step $N$ belonging to the component $P^{(\ell)}$, we write $f_{\pi^\bullet}(N)$ to denote $f^{(\ell)}(N)$.

We now define statistics $\area$, $\pdinv$ and $\ldinv$, along with the skeleton of labeled $(m,n)$-lattice paths. 
For a single $(m,n)$-lattice path $P$, the \emph{area}, denoted $\operatorname{area}(P)$, is the number of cells with non-negative content lying below $P$.
Equivalently, this counts the cells between $P$ and the line $my=nx$. 
By counting these cells row by row, from the $1$st to the $n$-th row, we obtain the \emph{area sequence} $\operatorname{aseq}(P) = (a_1, \dots, a_n)$.
For a $k$-tuple $P^\bullet=(P^{(1)}, \dots, P^{(k)})$, the total area is $\operatorname{area}(P^\bullet) = \sum_{\ell=1}^{k} \operatorname{area}(P^{(\ell)})$.

Let $\pi$ be a labeled $(m,n)$-lattice path, and denote its $i$-th north step by $N_i$.
Consider a collection of indeterminates $\{z_{i,j,k}\}$ indexed by $1\le i \le n$, $j \in \mathbb{Z}_{\ge 0}$, and $k \in \mathbb{Z}_{\ge 1}$.
We define the \emph{skeleton} of $\pi$, denoted $z(\pi)$, as the monomial
\[
    z(\pi) = \prod_{i=1}^n z_{i, \operatorname{aseq}(P_\pi)_i, f_\pi(N_i)}.
\]
For $\pi^\bullet=(\pi^{(1)}, \dots, \pi^{(k)}) \in \PF^k_{m,n}$, we define its skeleton as the product of the component skeletons:
\[
    z(\pi^\bullet) = \prod_{\ell=1}^k z(\pi^{(\ell)}).
\]

The \emph{content} of a north step is defined as the content of the cell immediately to its left, and the content of an east step is that of the cell immediately below it.
The \emph{path diagonal inversion}, denoted $\operatorname{pdinv}(P^{\bullet})$, is the number of pairs $(N, E)$ consisting of a north step $N$ and an east step $E$ of $P^{\bullet}$ such that $\operatorname{cont}(N) < \operatorname{cont}(E)$.

For a tuple $\pi^\bullet \in \PF^k_{m,n}$, a pair $(N, N')$ of north steps in $\pi^\bullet$ forms a \emph{labeled diagonal inversion} if:
\begin{enumerate}
    \item $\operatorname{cont}(N) < \operatorname{cont}(N') < \operatorname{cont}(N) + km$, and
    \item $f_{\pi^\bullet}(N) \ge f_{\pi^\bullet}(N')$.
\end{enumerate}
The statistic $\operatorname{ldinv}(\pi^\bullet)$ is the total number of such inversions.

\begin{rmk}
We compare our definitions of $\operatorname{pdinv}$ and $\operatorname{ldinv}$ with those in the existing literature.
Our path diagonal inversion corresponds to the definition of the dimension of the invariant set in \cite{GMV17} (or the equivalent statistic on lattice paths).
This definition differs from Wilson's path diagonal inversion in \cite{Wil23}, which we denote by $\operatorname{pdinv}'$, by the relation
\[
    \operatorname{pdinv}'(P^{\bullet}) = \operatorname{pdinv}(P^{\bullet}) + \operatorname{maxldinv}(P^{\bullet}),
\]
where $\operatorname{maxldinv}$ denotes the maximum value of the $\operatorname{ldinv}$ statistic over all valid labels of $P^{\bullet}$.

Furthermore, our labeled diagonal inversion $\operatorname{ldinv}$ is distinct from statistics appearing in previous references, where several variants of this concept exist.
Wilson \cite{Wil23} also used the term \emph{labeled diagonal inversion}, which we denote by $\operatorname{ldinv}'$ to avoid confusion.
Bergeron, Garsia, Leven, and Xin \cite{BGLX16} defined \emph{temporary dinv}, denoted by $\operatorname{tdinv}$, while Mellit \cite{Mel21} used the same terminology with a slight modification, denoted here by $\operatorname{tdinv}'$.
These statistics are related by the following identities:
\[
    \operatorname{ldinv} = \operatorname{maxldinv} - \operatorname{ldinv}' = \operatorname{maxldinv} - \operatorname{tdinv} = \operatorname{tdinv}'.
\]
\end{rmk}

\begin{example} 
Let $m=4, n=3$, and $k=2$. Figure~\ref{fig:path_tuple} illustrates an example of $(m,n)$-lattice paths. If we label their north steps by $3,2,4$ and $1,2,1$, respectively, we obtain a tuple of labeled $(m,n)$-lattice paths of order $k$ with a connecting permutation $21$. We denote this by $\pi^\bullet=(\pi, \pi')$. 

In the figure, the content value of each cell is displayed, with nonnegative integers shown in bold. Furthermore, a value is colored green if it is a content of an east step or a north step of $\pi^\bullet$.

First, the skeletons are given by
\[
    z_\pi=z_{1,3,3}z_{2,2,2}z_{3,3,4}, \quad z_{\pi'}=z_{1,2,1}z_{2,3,2}z_{3,1,1}, \quad \text{and} \quad z_{\pi^\bullet}=z_\pi z_{\pi'}.
\]
Restricting our attention to the area sequences, we have
\[
    \aseq(P_\pi)=(3,2,3) \quad \text{and} \quad \aseq(P_{\pi'})=(2,3,1).
\]
In particular, the total area is $\area(P_{\pi^\bullet})=\area(P_\pi)+\area(P_{\pi'})=8+6=14$.
Specializing $z_{i,j,k}\mapsto t^{j}x_k$ in $z_{\pi^\bullet}$ yields the monomial $x^{f_{\pi^\bullet}}$ corresponding to $\pi^\bullet$ together with the area information:
\[
    x^{f_{\pi^\bullet}}:=\prod_{\text{$N$: north step of $P_{\pi^{\bullet}}$}}f_{\pi^\bullet}(N)  ,\quad \text{and }\quad t^{\area(P_{\pi^\bullet)}}x^{f_{\pi^\bullet}}=t^{14}x_1^2x_2^2x_3x_4.
\]

To compute $\pdinv$ and $\ldinv$, we first denote the north and east steps of $\pi$ by $N_1,N_2,N_3$ and $E_1,\dots,E_5$, and those of $\pi'$ by $N_1',N_2',N_3'$ and $E_1',E_2',E_3'$. Then,
\begin{align*}
    \pdinv(\pi^\bullet)&=\#\{(N_2,E_3),(N_2,E_1'),(N_1',E_3),(N_1',E_1'),(N_3',E_1),(N_3',E_3),(N_3',E_1')\}=7, \\
    \ldinv(\pi^\bullet)&=\#\{(N_1,N_2'), (N_2,N_2'),(N_3',N_1')\}=3.
\end{align*}
    
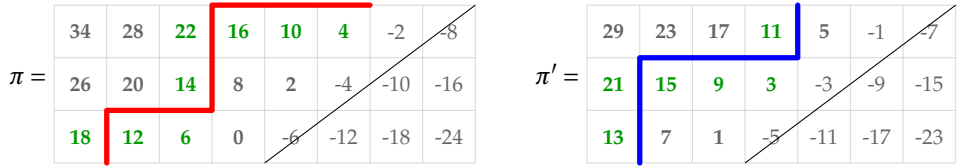
\begin{figure}[ht]
\centering
    $\pi=$
\begin{tikzpicture}[scale=0.7, baseline=(current bounding box.center)]
    % Grid definitions
    \def\xmin{-4} \def\xmax{3}
    \def\ymin{0} \def\ymax{2}

    % LOOP: Fill Numbers
    \foreach \x in {\xmin,...,\xmax} {
        \foreach \y in {\ymin,...,\ymax} {
            
            % 1. Calculate Value
            \pgfmathtruncatemacro{\val}{8*(\y) - 6*(\x+1)}
            
            % 2. Check if Bold (val >= 0)
            \ifnum\val<0
                \def\txtstyle{}
            \else
                \def\txtstyle{\bfseries}
            \fi
            
            % 3. Check if Green
            \def\txtcolor{black!60}
            
            % Logic for Pi (Red Path):
            % Below East Steps
            \ifnum\y=0 \ifnum\x=-3 \def\txtcolor{green!60!black} \fi \fi
            \ifnum\y=0 \ifnum\x=-2 \def\txtcolor{green!60!black} \fi \fi
            \ifnum\y=2 \ifnum\x=-1 \def\txtcolor{green!60!black} \fi \fi
            \ifnum\y=2 \ifnum\x=0  \def\txtcolor{green!60!black} \fi \fi
            \ifnum\y=2 \ifnum\x=1  \def\txtcolor{green!60!black} \fi \fi
            
            % Left of North Steps
            % ADDED THIS LINE FOR CELL 18 (Left of start step)
            \ifnum\x=-4 \ifnum\y=0 \def\txtcolor{green!60!black} \fi \fi
            
            \ifnum\x=-2 \ifnum\y=1 \def\txtcolor{green!60!black} \fi \fi
            \ifnum\x=-2 \ifnum\y=2 \def\txtcolor{green!60!black} \fi \fi

            % Draw Node
            \node[\txtcolor, font=\txtstyle, scale=0.8] at (\x+0.5, \y+0.5) {\val};
        }
    }

    % Grid and Lines
    \draw[help lines, gray!30, thin] (-4,0) grid (4,3);
    \draw[black, thin] (0,0) -- (4,3);

    % Path
    \draw[red, line width=2pt, line cap=round, line join=round] 
        (-3,0) -- (-3,1) -- (-2,1) -- (-1,1) -- (-1,2) -- (-1,3) -- (0,3) -- (1,3) -- (2,3);
\end{tikzpicture}
\qquad
    $\pi'=$
    \begin{tikzpicture}[scale=0.7, baseline=(current bounding box.center)]
        % Grid definitions
        \def\xmin{-3} \def\xmax{3}
        \def\ymin{0} \def\ymax{2}

        % LOOP: Fill Numbers
        \foreach \x in {\xmin,...,\xmax} {
            \foreach \y in {\ymin,...,\ymax} {
                
                % 1. Calculate Value (Base is 1)
                \pgfmathtruncatemacro{\val}{1 + 8*(\y) - 6*(\x+1)}
                
                % 2. Check if Bold
                \ifnum\val<0
                    \def\txtstyle{}
                \else
                    \def\txtstyle{\bfseries}
                \fi
                
                % 3. Check if Green
                \def\txtcolor{black!60}
                
                % Logic for Pi' (Blue Path):
                % Below East Steps:
                \ifnum\y=1 \ifnum\x=-2 \def\txtcolor{green!60!black} \fi \fi
                \ifnum\y=1 \ifnum\x=-1 \def\txtcolor{green!60!black} \fi \fi
                \ifnum\y=1 \ifnum\x=0  \def\txtcolor{green!60!black} \fi \fi

                % Left of North Steps:
                \ifnum\x=-3 \ifnum\y=0 \def\txtcolor{green!60!black} \fi \fi
                \ifnum\x=-3 \ifnum\y=1 \def\txtcolor{green!60!black} \fi \fi
                \ifnum\x=0  \ifnum\y=2 \def\txtcolor{green!60!black} \fi \fi

                % Draw Node
                \node[\txtcolor, font=\txtstyle, scale=0.8] at (\x+0.5, \y+0.5) {\val};
            }
        }

        % Grid and Lines
        \draw[help lines, gray!30, thin] (-3,0) grid (4,3);
        \draw[black, thin] (0,0) -- (4,3);

        % Path
        \draw[blue, line width=2pt, line cap=round, line join=round] 
            (-2,0) -- (-2,1) -- (-2,2) -- (-1,2) -- (0,2) -- (1,2) -- (1,3);
    \end{tikzpicture}
\caption{An example of an element of $\PF^{2}_{4,3}$ (labels are specified in the context).}
    \label{fig:path_tuple}
\end{figure}
\end{example}

\subsection{Cyclic parking functions}
 A \emph{cyclic $(m,n)$-parking function} is a labeled $(m,n)$-lattice path $\pi$ that starts with a north step $N$, contains exactly $m$ east steps, and satisfies the condition that if it ends with a north step $N'$, then $f_\pi(N) > f_\pi(N')$.
Equivalently, a cyclic $(m,n)$-parking function is an element of $\PF^1_{m,n}$.

Let $\cPF_{m,n}^k$ denote the set of $k$-tuples of cyclic $(m,n)$-parking functions.
This set can be viewed as the subset of $\PF^k_{m,n}$ where the connecting permutation $\phi$ can be chosen to be the identity map. In Theorem~\ref{thm: main}, the combinatorial formula on the right-hand side is given by
\[
\sum_{\pi^\bullet\in\cPF^k_{m,n}} q^{\stat(\pi^\bullet)}t^{\area(P_{\pi^\bullet})} x^{f_{\pi^\bullet}},
\]
where we conveniently denote
\begin{equation*}
    \stat(\pi^{\bullet}) := C(m,n,k) - \pdinv(P_{\pi^\bullet}) - \ldinv(\pi^\bullet),
\end{equation*}
with the constant term $C(m,n,k) = \frac{(mk-1)(nk-1)+k-1}{2}$.

We define a partial order $\prec_{m,n}$ on $\cPF_{m,n}$.
For $\pi, \pi' \in \cPF_{m,n}$, we denote $\pi \prec_{m,n} \pi'$ if $P_{\pi'}$ lies weakly to the left of $P_{\pi}$, the paths share no east steps, and at every vertex $v$ where they intersect, the north step $N \in P_{\pi}$ ending at $v$ and the north step $N' \in P_{\pi'}$ starting at $v$ satisfy $f_{\pi}(N) \ge f_{\pi'}(N')$. If $v$ is on the line $y=n$, then we take $N'$ to be the $(m,n)$-translation of the first step of $\pi'$. When $m$ and $n$ are clear from the context, we simply write $\pi \prec \pi'$.
In the subsequent section, we reformulate the rational compositional shuffle theorem in terms of $P$-tableaux on this partial order (\Cref{prop: reformulation of rational shuffle}).

\section{Reformulation of rational Compositional shuffle theorem}\label{Sec: rcs}
\subsection{Crossings and $q$-statistics}\label{subsec: Crossings and q-stat}

\begin{definition}
Let $(N, a)$ and $(N', a')$ be two labeled north steps.
Let $(x, y)$ be the end point of $N$ and $(x', y)$ be the starting point of $N'$.
We say that the labeled north step $(N, a)$ can be \emph{extended} with $(N', a')$ if $x < x'$, or if $x = x'$ and $a < a'$.

Let $\pi$ and $\pi'$ be labeled $(m,n)$-lattice paths. For a technical reason, we consider extended labeled paths $\tilde{\pi}$ and $\tilde{\pi}'$, obtained by appending labeled north steps to the ends of $\pi$ and $\pi'$, respectively. For $1 \le y \le n+1$, let $N_y$ (respectively $N'_y$) denote the $y$-th north step of $\tilde{\pi}$ (respectively $\tilde{\pi}'$). We also denote by $f$ and $f'$ the labeling functions of $\tilde{\pi}$ and $\tilde{\pi}'$, respectively.

We say that $\tilde{\pi}$ and $\tilde{\pi}'$ have a \emph{crossing} at $(x,y)$ for some $1 \le y \le n$ if:
\begin{enumerate}
    \item $(N_y, f(N_y))$ can be extended with $(N'_{y+1}, f'(N'_{y+1}))$,
    \item $(N'_y, f'(N'_y))$ can be extended with $(N_{y+1}, f(N_{y+1}))$, and
    \item the crossing occurs at $x$-coordinate $x = \min\{x_{N_{y+1}}, x_{N'_{y+1}}\}$, where $x_S$ denotes the $x$-coordinate of a north step $S$.
\end{enumerate}

Given sheet information indicating that $\pi$ lies in an earlier sheet than $\pi'$, we say that a crossing is \emph{positive} if $N_y$ lies to the left of $N'_y$; if they share the same geometric position, we further require that $f(N_y) < f'(N'_y)$. We denote by $\operatorname{pos}(\tilde{\pi}, \tilde{\pi}')$ the total number of positive crossings between $(\tilde{\pi}, \tilde{\pi}')$.

Finally, for $\pi^\bullet=(\pi^{(1)}, \dots, \pi^{(k)}) \in \PF^k_{m,n}$ with a connecting permutation $\phi$, we define the statistic $\operatorname{pos}(\pi^\bullet)$ by
\[
    \operatorname{pos}(\pi^\bullet) = \sum_{1 \le \ell < \ell' \le k} \operatorname{pos}(\tilde{\pi}^{(\ell)}, \tilde{\pi}^{(\ell')}).
\]
Here, the extended labeled path $\tilde{\pi}^{(\ell)}$ is obtained from $\pi^{(\ell)}$ by appending a north step with the label on the first north step of $\pi^{(\phi(\ell))}$. We leave it as an exercise to the reader to check that $\pos(\pi^{\bullet})$ does not depend on the choice of the connecting permutation $\phi$.
\end{definition}

\begin{lem}\label{lem:stat+pos depend only on skeleton}
    Let $\pi^\bullet\in\PF^k_{m,n}$ be a labeled $(m,n)$-lattice paths tuple of order $k$. The sum
    \[
        \pdinv(\pi^\bullet) + \ldinv(\pi^\bullet) + \pos(\pi^\bullet)
    \]
    depends only on the skeleton $z(\pi^\bullet)$.
\end{lem}

\begin{proof}
We analyze cases where statistics $\pdinv$ and $\ldinv$ vary according to the sheet assignment.

    \noindent $\bullet$ Path Diagonal Inversions ($\pdinv$):
     For two cells $c$ and $c'$ (which may belong to different sheets), the condition $\cont(c) < \cont(c')$ is invariant under changing the sheet ordering, except in the case where $c$ and $c'$ share the same geometric position and their relative sheet ordering is swapped.
    Thus, for a pair $(N, E)$ consisting of a north step and an east step, the contribution to the path diagonal inversion count is independent of the specific sheet ordering unless the end point of $N$ coincides with the end point of $E$. We refer to such a configuration as a \emph{corner} at height $y$ if $N$ is the $y$-th north step. Figure~\ref{fig:configurations} illustrates the local configuration of a corner.

    \noindent $\bullet$ Labeled Diagonal Inversions ($\ldinv$):
    A similar argument applies to $\ldinv$. For a pair of labeled north steps $(N, N')$, the contribution to the statistic is invariant under changing the sheet ordering, except in two specific configurations:
    \begin{enumerate}
        \item \emph{Overlapped pair}: $N$ and $N'$ correspond to the same geometric step with different labels.
        \item \emph{Adjacent pair}: $N$ and $N'$ are vertically adjacent (i.e., the end point of $N$ coincides with the starting point of $N'$). Moreover, the label of $N$ is larger or equal to the label of $N'$.
    \end{enumerate}
    See Figure~\ref{fig:configurations} for the local configurations of an overlapped pair and an adjacent pair.

    It is enough to show the following local claim:

    \textbf{Claim.} Let $(N^{(1)},a_1),\dots,(N^{(k)},a_k)$ be north steps with labels at height $y$ and $(M^{(1)},b_1),\dots,(M^{(k)},b_k)$ be north steps with labels at height $y+1$. We also require that there exists a permutation $\phi\in \mathfrak{S}_{k}$ such that $(N^{(i)},a_i)$ is extended with $(M^{(\phi(i))},b_{\phi(i)})$. Now pick two permutations $\psi$ and $\psi'$ such that $(N^{(\psi(i))},a_{\psi(i)})$ is extended with $(M^{(\psi'(i))},b_{\psi'(i)})$ and we make a (partial) labeled path with these north steps and assign it to the $i$-th sheet. Then the total contribution of the following four items is invariant under the choice of $\psi$ and $\psi'$:
    \begin{enumerate}
        \item The number of corners at height $y$ such that the north step is in an earlier sheet;
        \item The number of overlapped pairs at height $y$ such that a step with a bigger label is in an earlier sheet;
        \item The number of adjacent pairs involving heights $y$ and $y+1$ contributing such that the upper north step is in an earlier sheet;
        \item The number of positive crossings at height $y$.
    \end{enumerate}

    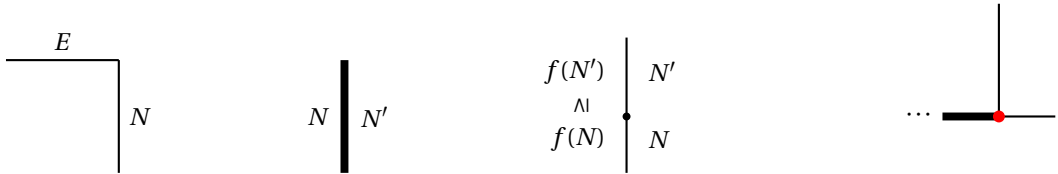
\begin{figure}[ht]
        \centering
        \begin{tikzpicture}[scale=1.5, thick, >=stealth]
        
            % --- Figure 1: East step and North step forming a corner ---
            \begin{scope}[shift={(0,0)}]
                \draw (0,1) -- (1,1) node[midway, above] {$E$};
                \draw (1,1) -- (1,0) node[midway, right] {$N$};
            \end{scope}
        
            % --- Figure 2: Overlapped North steps ---
            \begin{scope}[shift={(3,0)}]
                \draw[line width=3pt] (0,0) -- (0,1);
                \node[left] at (-0.05, 0.5) {$N$};
                \node[right] at (0.05, 0.5) {$N'$};
            \end{scope}
        
            % --- Figure 3: Stacked labels with inequality ---
            \begin{scope}[shift={(5.5,0)}]
                \draw (0,0) -- (0,1.2);
                
                % Labels
                \node[anchor=east] at (-0.1, 0.3) {$f(N)$};
                \node[anchor=west] at (0.1, 0.3) {$N$};
                
                \node[anchor=east] at (-0.1, 0.9) {$f(N')$};
                \node[anchor=west] at (0.1, 0.9) {$N'$};
                
                % Center Dot
                \fill[black] (0,0.5) circle (1pt);
                
                % Inequality symbol
                \node at (-0.4, 0.6) {\rotatebox{90}{$\geq$}};
            \end{scope}
        
            % --- Figure 4: Path intersection ---
            \begin{scope}[shift={(7.8,0)}] % Adjusted shift slightly
                \node at (0.3, 0.5) {$\cdots$};
                
                % Incoming Horizontal line
                \draw[line width=3pt] (0.5,0.5) -- (1,0.5);
        
                % Outgoing lines
                \draw (1,0.5) -- (1,1.5);   % North
                \draw (1,0.5) -- (1.5,0.5); % East
                
                % Intersection dot
                \fill[red] (1,0.5) circle (1.5pt);
            \end{scope}
        
        \end{tikzpicture}
        \caption{The corner, overlapped pairs, adjacent pairs, and the crossing.}
        \label{fig:configurations}
    \end{figure}

It suffices to show the claim for $k=2$. For simplicity, let $(N,a)$ and $(N',a')$ denote labeled north steps at height $y$, and let $(M,b)$ and $(M',b')$ denote labeled north steps at height $y+1$. Without loss of generality, we assume that $N$ lies weakly to the left of $N'$ (if they are geometrically identical, a label of $N'$ is bigger) and that $M$ lies weakly to the left of $M'$ (following the same convention). Note that $(N,a)$ can be extended with $(M,b)$ and $(N',a')$ can be extended with $(M',b')$. We conduct a case-by-case analysis.

    \noindent $\bullet$ Case 1): the start point of $M$ is strictly to the left of the end point of $N'$. 
    In this scenario, there is no geometric interaction between the relevant steps that produces a contribution to items (1) through (4). A representative configuration is illustrated below:
    
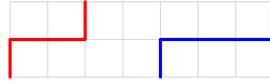
\begin{center}
        \begin{tikzpicture}[scale=.5]
            % Grid covering the bounding box of both paths
            \draw[help lines, gray!30, thin] (-4,0) grid (3,2);
            
            % --- Path 1 (Black): NEEEN starting at (0,0) ---
            \draw[blue, very thick, line cap=round, line join=round] 
                (0,0) -- (0,1) -- (1,1) -- (2,1) -- (3,1) -- (3,2);
            
            % --- Path 2 (Black): NEEN starting at (-4,0) ---
            \draw[red, very thick, line cap=round, line join=round] 
                (-4,0) -- (-4,1) -- (-3,1) -- (-2,1) -- (-2,2);
        \end{tikzpicture}
    \end{center}
        
    \noindent $\bullet$ Case 2): $(N',a')$ forms an adjacent pair with $(M,b)$.
    
    \begin{itemize}
        \item Case 2(a): $N$ lies strictly to the left of $N'$. 
        Depending on the sheet assignment, this configuration may contribute to items (1) or (3). A typical configuration is illustrated below:
        
        
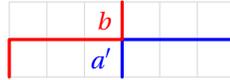
\begin{center}
            \begin{tikzpicture}[scale=.5]
                % Grid covering the bounding box
                \draw[help lines, gray!30, thin] (-3,0) grid (3,2);
                
                % --- Path 1 (Blue): NEEEN starting at (0,0) ---
                \draw[blue, very thick, line cap=round, line join=round]
                    (0,0) -- (0,1) node[midway, left] {$a'$} 
                          -- (1,1) -- (2,1) -- (3,1) -- (3,2);
                
                % --- Path 2 (Red): NEEEN starting at (-3,0) ---
                \draw[red, very thick, line cap=round, line join=round]
                    (-3,0) -- (-3,1) -- (-2,1) -- (-1,1) -- (0,1) 
                           -- (0,2) node[midway, left] {$b$};
            \end{tikzpicture}
        \end{center}
         Labels satisfy $b \le a'$.  If the red path is in an earlier sheet, there is exactly one contribution to item (3). Otherwise, if the red path is in a later sheet, there is exactly one contribution to item (1).
        
        \item Case 2(b): $N$ and $N'$ are identical. For labels we must have $a<b\le a'$. A typical configuration is illustrated below:
        
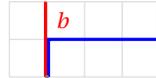
\begin{center}
            \begin{tikzpicture}[scale=.5]
                % Define the shift amount for visual separation
                \def\sep{0.04}
        
                % Grid covering the bounding box
                \draw[help lines, gray!30, thin] (-1,0) grid (3,2);
        
                % Define style for thickened paths
                \tikzset{
                    thick path/.style={line width=2pt, line cap=round, line join=round}
                }
        
                % --- Path 1 (Blue): NEEEN starting at (0,0) ---
                \draw[blue,very thick]
                    (\sep,0) -- (\sep,1) node[midway, left, font=\small] {}
                              -- (1,1) -- (2,1) -- (3,1) -- (3,2);
        
                % --- Path 2 (Red): NN starting at (0,0) ---
                \draw[red,very thick]
                    (-\sep,0) -- (-\sep,1)
                              -- (-\sep,2) node[midway, right, font=\small] {$b$};
        
                % Optional: Mark the nominal starting point (0,0)
                \fill[gray] (0,0) circle (1.5pt);
            \end{tikzpicture}
        \end{center}
     If the red path is in an earlier sheet, there is exactly one contribution to item (3). Otherwise, if the red path is in a later sheet, there is exactly one contribution to item (2).
    \end{itemize}

    \noindent $\bullet$ Case 3): $(N',a')$ can be extended with $(M,b)$.

    \begin{itemize}
        \item Case 3(a):  $N$ lies strictly to the left of $N'$. A typical configuration is illustrated below:
        \begin{center}
        % Define common styles for consistency
        \tikzset{
            grid style/.style={help lines, gray!30, thin},
            path style/.style={very thick, line cap=round, line join=round}
        }
        
        % --- Figure 1 ---
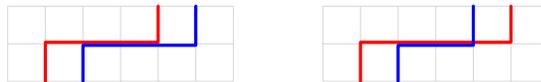
        \begin{tikzpicture}[scale=.5]
            % Grid covering the area
            \draw[grid style] (-3,0) grid (3,2); 
            \def\sep{0.04}
        
            % Red path: Starts at (-2,0) using NEEEN
            \draw[red, path style] 
                (-2,0) -- (-2,1+\sep) -- (-1,1+\sep) -- (0,1+\sep) -- (1,1+\sep) -- (1,2);
        
            % Blue path: Starts at (-1,0) using NEEEN
            \draw[blue, path style] 
                (-1,0) -- (-1,1-\sep) -- (0,1-\sep) -- (1,1-\sep) -- (2,1-\sep) -- (2,2);
        \end{tikzpicture}
        \hspace{1cm} % Space between figures
        %
        % --- Figure 2 ---
        \begin{tikzpicture}[scale=.5]
            % Grid covering the area
            \draw[grid style] (-3,0) grid (3,2);
            \def\sep{0.04}
        
            % Red path: Starts at (-2,0) using NEEEEN
            \draw[red, path style] 
                (-2,0) -- (-2,1+\sep) -- (-1,1+\sep) -- (0,1+\sep) -- (1,1+\sep) -- (2,1+\sep) -- (2,2);
        
            % Blue path: Starts at (-1,0) using NEEN
            \draw[blue, path style] 
                (-1,0) -- (-1,1-\sep) -- (0,1-\sep) -- (1,1-\sep) -- (1,2);
        \end{tikzpicture}
        \end{center}
        If $(N,a)$ is in an earlier sheet, there is exactly one contribution to item (4). Otherwise, if $(N,a)$ is in a later sheet, there is exactly one contribution to item (1).
        
        \item Case 3(b): $N$ and $N'$ are identical. A typical configuration is illustrated below:
        \begin{center}
        % Define common styles for consistency
        \tikzset{
            grid style/.style={help lines, gray!30, thin},
            path style/.style={very thick, line cap=round, line join=round}
        }
        
        % --- Figure 1 ---
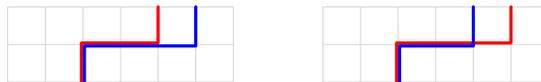
        \begin{tikzpicture}[scale=.5]
            % Grid covering the area
            \draw[grid style] (-3,0) grid (3,2); 
            \def\sep{0.04}
        
            % Red path: Starts at (-2,0) using NEEEN
            \draw[red, path style] 
                (-1-\sep,0) -- (-1-\sep,1+\sep) -- (-1,1+\sep) -- (0,1+\sep) -- (1,1+\sep) -- (1,2);
        
            % Blue path: Starts at (-1,0) using NEEEN
            \draw[blue, path style] 
                (-1+\sep,0) -- (-1+\sep,1-\sep) -- (0,1-\sep) -- (1,1-\sep) -- (2,1-\sep) -- (2,2);
        \end{tikzpicture}
        \hspace{1cm} % Space between figures
        %
        % --- Figure 2 ---
        \begin{tikzpicture}[scale=.5]
            % Grid covering the area
            \draw[grid style] (-3,0) grid (3,2);
            \def\sep{0.04}
        
            % Red path: Starts at (-2,0) using NEEEEN
            \draw[red, path style] 
                (-1-\sep,0) -- (-1-\sep,1+\sep) -- (-1,1+\sep) -- (0,1+\sep) -- (1,1+\sep) -- (2,1+\sep) -- (2,2);
        
            % Blue path: Starts at (-1,0) using NEEN
            \draw[blue, path style] 
                (-1+\sep,0) -- (-1+\sep,1-\sep) -- (0,1-\sep) -- (1,1-\sep) -- (1,2);
        \end{tikzpicture}
        \end{center}
        When $a=a'$, there is no contribution to any of items (1) to (4). Assume $a<a'$. If $(N,a)$ is in an earlier sheet, there is exactly one contribution to item (4). Otherwise, if $(N,a)$ is in a later sheet, there is exactly one contribution to item (2). 
    \end{itemize}
\end{proof}

\subsection{Original formulation of rational shuffle theorem}
A \emph{$(km, kn)$-Dyck path} is a lattice path from $(0,0)$ to $(km, kn)$ consisting of north and east unit steps that stays weakly above the line $my=nx$. A \emph{$(km, kn)$-parking function} is a pair $\pi=(P_\pi,f_\pi)$, where $P_\pi$ is a $(km, kn)$-Dyck path and $f_\pi$ is a labeling function of the north steps of $P_\pi$ with positive integers such that the labels are strictly increasing along consecutive north steps. We denote the set of $(km, kn)$-parking functions by $\PF_{km,kn}$. For $\pi\in \PF_{km,kn}$, let $(a_1m,a_1 n),\dots,(a_{r}m,a_{r}n)$ be a sequence of points that $P_{\pi}$ intersects the line $my=nx$, enumerated from the bottom. Note that we must have $a_1=0$ and $a_{r}=k$. We define $\touch(\pi)=\touch(P_{\pi})$ to be a composition of $k$ given by $(a_2-a_1,\dots,a_{r}-a_{r-1})$.

By decomposing a $(km,kn)$-parking function into horizontal strips defined by $(\ell-1)n < y \le \ell n$ for $\ell=1, \dots, k$, we identify a $(km, kn)$-parking function with a labeled $(m,n)$-lattice path of order $k$, denoted by $\pi^\bullet=(\pi^{(1)}, \dots, \pi^{(k)})$.
In this framework, each component $\pi^{(\ell)}$ is viewed as lying in the $\ell$-th sheet, where the point $((\ell-1)m, (\ell-1)n)$ in the global lattice is identified with the local origin $(0,0)$ of that sheet.
This $k$-tuple satisfies the following conditions:
\begin{enumerate}
    \item $\pi^{(1)}$ starts at the origin $(0,0)$.
    \item For $1 \le \ell \le k-1$, the path $\pi^{(\ell+1)}$ starts at the shifted end point of the previous path, i.e., $ \operatorname{end}(\pi^{(\ell)}) = \operatorname{start}(\pi^{(\ell+1)}) + (m,n)$. This ensures that the global path is continuous across strip boundaries.
    \item For $1 \le \ell \le k-1$, if $\pi^{(\ell)}$ ends with a north step $N$ and $\pi^{(\ell+1)}$ starts with a north step $N'$, then $f_{\pi^\bullet}(N) < f_{\pi^\bullet}(N')$.
\end{enumerate}
Note that these conditions guarantee that a $(km,kn)$-parking function corresponds to an element of $\PF^k_{m,n}$ whose connecting permutation $\phi$ is the cyclic shift $\ell \mapsto \ell+1 \pmod k$.

This perspective allows us to naturally define $\area$, $\ldinv$, and $\pdinv$, and consequently the statistic $\stat$, for $(km, kn)$-parking functions. The rational compositional shuffle theorem, conjectured in \cite{BGLX16} and proven by Mellit \cite{Mel21}, states:
\begin{thm}\label{thm: Rational shuffle} 
    For coprime positive integers $m,n$ and a composition $\alpha\models k$, we have 
    \[
        C_{\alpha}[-MX^{m,n}] \cdot 1 = \sum_{\substack{\pi^\bullet \in \mathrm{PF}_{km,kn}\\ \touch(\pi^{\bullet})=\alpha}} q^{\stat(\pi^\bullet)} t^{\mathrm{area}(P_{\pi^\bullet})} x^{f_{\pi^\bullet}}.
    \]
\end{thm}

\begin{rmk}
The original formulation of the rational shuffle theorem utilizes the $\operatorname{dinv}$ statistic. In \cite[Lemma 3.3]{GMV17}, it was shown that $C(m,n,k)-\pdinv(P) = \dinv(P)$.
\end{rmk}

\subsection{The reformulation}

We now reformulate Theorem~\ref{thm: Rational shuffle} in terms of $P$-tableaux of cyclic $(m,n)$-parking functions. Let $\PTab_{m,n}(k)$ be the subset of $\cPF_{m,n}^k$ consisting of tuples $\pi^\bullet=(\pi^{(1)},\dots,\pi^{(k)})$ satisfying:
\begin{itemize}
    \item $\aseq(P_{\pi^{(1)}})_1=0$, and
    \item $\pi^{(\ell)} \nprec \pi^{(\ell+1)}$ for all $1 \le \ell \le k-1$.
\end{itemize}
In other words, $\PTab_{m,n}(k)$ is the set of \emph{$P$-tableaux} of shape $(k)$ in the sense of \cite{Gas96}, subject to an additional condition on the area sequence.

\begin{proposition}\label{prop: reformulation of rational shuffle}
There exists a bijection $\Gamma: \PTab_{m,n}(k) \to \PF_{km,kn}$ such that
\begin{equation*}
\pdinv(P_{\pi^\bullet})+\ldinv(\pi^{\bullet})=\pdinv(P_{\Gamma(\pi^\bullet)})+\ldinv(\Gamma(\pi^{\bullet}))
\end{equation*}
and $z(\pi^{\bullet})=z(\Gamma(\pi^{\bullet}))$ for $\pi^{\bullet
}\in \PTab_{m,n}(k)$.
\end{proposition}

To the best of our knowledge, even the identity $|\PTab_{m,n}(k)|= |\PF_{km,kn}|
$ was previously unknown. To construct a desired bijection $\Gamma: \PTab_{m,n}(k) \to \PF_{km,kn}$, we work recursively using a statistic-preserving $\mix$ operation.

\begin{definition}[The Mix Operation and the bijection $\Gamma$]

     Let $\pi^\bullet=(\pi^{(1)},\dots,\pi^{(k)})$ be a $(km, kn)$-parking function and let $\tau$ be a cyclic $(m, n)$-parking function. Let $\ell_{\max}$ be the maximal integer $\ell$ such that $\tau$ and $\pi^{(\ell)}$ have a crossing, if such an integer exists. Let $v$ be the last crossing (in the standard path order) between $\tau$ and $\pi^{(\ell_{\max})}$. Then the \emph{mix} of $\tau$ and $\pi^\bullet$, denoted by $\mix(\tau,\pi^\bullet)$, is the $((k+1)m, (k+1)n)$-parking function $(\pi_{\mix}^{(1)},\dots,\pi_{\mix}^{(k+1)})$ defined by:
    \begin{itemize}
        \item $\pi_{\mix}^{(\ell)} = \pi^{(\ell)}$ for $\ell < \ell_{\max}$.
        \item $\pi_{\mix}^{(\ell_{\max})}$ is formed by the portion of $\pi^{(\ell_{\max})}$ before $v$ concatenated with the portion of $\tau$ after $v$.
        \item $\pi_{\mix}^{(\ell_{\max}+1)}$ is formed by the portion of $\tau$ before $v$ concatenated with the portion of $\pi^{(\ell_{\max})}$ after $v$.
        \item $\pi_{\mix}^{(\ell)} = \pi^{(\ell-1)}$ for $\ell > \ell_{\max}+1$.
    \end{itemize}
    Note that $\mix(\tau,\pi^\bullet)$ is only defined when $\tau$ and a component of $\pi^\bullet$ has a crossing.

    Now, given a $P$-tableau $\pi^\bullet= (\pi^{(1)},\dots, \pi^{(k)}) \in \PTab_{m,n}(k)$, we recursively define the map $\Gamma$ as follows:
    \begin{enumerate}
        \item $\Gamma^{(1)}(\pi^\bullet) = \pi^{(1)}$.
        \item $\Gamma^{(\ell+1)}(\pi^\bullet) = \mix\left( \pi^{(\ell+1)},\Gamma^{(\ell)}(\pi^\bullet) \right)$ for $1 \le \ell \le k-1$.
    \end{enumerate}
   Finally, we set $\Gamma(\pi^\bullet) = \Gamma^{(k)}(\pi^\bullet)$.
\end{definition}

\begin{example}Consider a tuple of $(8,5)$-lattice paths $\pi^\bullet=(\pi,\pi',\pi'')$, defined as follows:

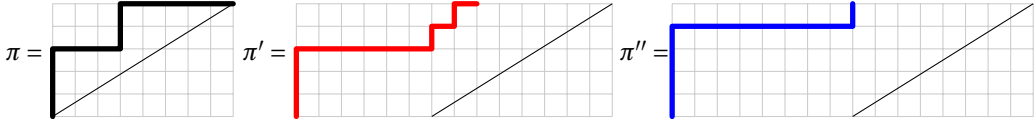
\begin{center}{$\pi=$
\begin{tikzpicture}[scale=0.3, >=stealth, baseline=(current bounding box.center)]

    % 1. The 5x8 Grid
    % (0,0) is bottom-left, (8,5) is top-right
    \draw[step=1, gray!40, thin] (0,0) grid (8,5);

    % 2. The diagonal line from (0,0) to (8,5)
    \draw[thin, black] (0,0) -- (8,5);

    % 3. The Path: NNNEEENNEEEEE
    \draw[line width=2pt, line cap=round, line join=round] 
        (0,0) -- (0,1)  % N
              -- (0,2)  % N
              -- (0,3)  % N
              -- (1,3)  % E
              -- (2,3)  % E
              -- (3,3)  % E
              -- (3,4)  % N
              -- (3,5)  % N
              -- (4,5)  % E
              -- (5,5)  % E
              -- (6,5)  % E
              -- (7,5)  % E
              -- (8,5); % E

\end{tikzpicture}
$\pi'=$
\begin{tikzpicture}[scale=0.3, >=stealth, baseline=(current bounding box.center)]

    % 1. The Grid (spanning -6 to 8)
    \draw[step=1, gray!40, thin] (-6,0) grid (8,5);

    % 2. The diagonal line (main region)
    \draw[thin, black] (0,0) -- (8,5);

    % 3. The Path: NNNEEEEEENENE starting from (-6,0)
    \draw[red, line width=2pt, line cap=round, line join=round] 
        (-6,0) -- (-6,1) % N
               -- (-6,2) % N
               -- (-6,3) % N
               -- (-5,3) % E
               -- (-4,3) % E
               -- (-3,3) % E
               -- (-2,3) % E
               -- (-1,3) % E
               -- (0,3)  % E
               -- (0,4)  % N
               -- (1,4)  % E
               -- (1,5)  % N
               -- (2,5); % E

\end{tikzpicture}
$\pi''=$\begin{tikzpicture}[scale=0.3, >=stealth, baseline=(current bounding box.center)]

    % 1. The Grid (extended to -8)
    \draw[step=1, gray!40, thin] (-8,0) grid (8,5);

    % 2. The diagonal line (main region reference)
    \draw[thin, black] (0,0) -- (8,5);

    % 3. The Path: NNNNEEEEEEEEN starting from (-8,0)
    \draw[blue, line width=2pt, line cap=round, line join=round] 
        (-8,0) -- (-8,1) % N
               -- (-8,2) % N
               -- (-8,3) % N
               -- (-8,4) % N
               -- (-7,4) % E
               -- (-6,4) % E
               -- (-5,4) % E
               -- (-4,4) % E
               -- (-3,4) % E
               -- (-2,4) % E
               -- (-1,4) % E
               -- (0,4)  % E
               -- (0,5); % N

\end{tikzpicture}}
\end{center}

For simplicity, we disregard labels and interpret a crossing as a geometric intersection of two lattice paths. In other words, labels are adjusted so that geometric intersection is a crossing. We then compute $\Gamma(\pi^\bullet)$, which is obtained by recursive applications of the mix operation. We begin with $\Gamma^{(1)}(\pi^\bullet)=\pi$.

By displaying $\pi$ and $\pi'$ together in the same sheet, we observe a crossing at $(0,3)$. This is the last crossing point, highlighted by a green dot in the figure below.
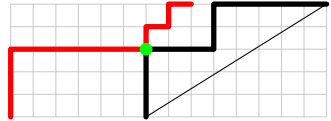
\begin{center}
\begin{tikzpicture}[scale=0.3, >=stealth, baseline=(current bounding box.center)]

    % 1. The Grid (spanning -6 to 8)
    \draw[step=1, gray!40, thin] (-6,0) grid (8,5);

    % 2. The diagonal line (main region)
    \draw[thin, black] (0,0) -- (8,5);

    % 3. The Path: NNNEEEEEENENE starting from (-6,0)
    \draw[red, line width=2pt, line cap=round, line join=round] 
        (-6,0) -- (-6,1) % N
               -- (-6,2) % N
               -- (-6,3) % N
               -- (-5,3) % E
               -- (-4,3) % E
               -- (-3,3) % E
               -- (-2,3) % E
               -- (-1,3) % E
               -- (0,3)  % E
               -- (0,4)  % N
               -- (1,4)  % E
               -- (1,5)  % N
               -- (2,5); % E
    \draw[line width=2pt, line cap=round, line join=round] 
        (0,0) -- (0,1)  % N
              -- (0,2)  % N
              -- (0,3)  % N
              -- (1,3)  % E
              -- (2,3)  % E
              -- (3,3)  % E
              -- (3,4)  % N
              -- (3,5)  % N
              -- (4,5)  % E
              -- (5,5)  % E
              -- (6,5)  % E
              -- (7,5)  % E
              -- (8,5); % E
    \fill[green] (0,3) circle (8pt);

\end{tikzpicture}
\end{center}
Applying the mix operation yields
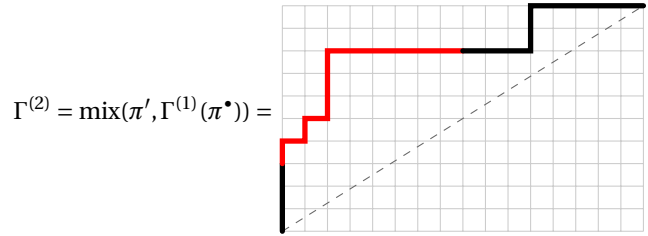
\begin{center}
$\Gamma^{(2)}=\operatorname{mix}(\pi',\Gamma^{(1)}(\pi^\bullet))=$
\begin{tikzpicture}[scale=0.3, >=stealth, baseline=(current bounding box.center)]

    % 1. Grid covering the area (extended to fit the path)
    \draw[step=1, gray!40, thin] (0,0) grid (16,10);

    % 2. Reference Diagonal for 16x10 (dashed to show the target frame)
    \draw[dashed, black!60] (0,0) -- (16,10);

    % 3. The Segments
    
    % Segment 1: Black NNN
    % Starts at (0,0), ends at (0,3)
    \draw[black, line width=2pt, line cap=round] 
        (0,0) -- (0,1) -- (0,2) -- (0,3);

    % Segment 2: Red NENENNNEEEEEE (added one E)
    % Starts at (0,3), now ends at (8,8)
    \draw[red, line width=2pt, line cap=round] 
        (0,3) -- (0,4)  % N
              -- (1,4)  % E
              -- (1,5)  % N
              -- (2,5)  % E
              -- (2,6)  % N
              -- (2,7)  % N
              -- (2,8)  % N
              -- (3,8)  % E
              -- (4,8)  % E
              -- (5,8)  % E
              -- (6,8)  % E
              -- (7,8)  % E
              -- (8,8); % E (Added step)

    % Segment 3: Black EEENNNEEEEE
    % Now starts at (8,8)
    \draw[black, line width=2pt, line cap=round] 
        (8,8) -- (9,8)   % E
              -- (10,8)  % E
              -- (11,8)  % E
              -- (11,9)  % N
              -- (11,10) % N
              -- (12,10) % E
              -- (13,10) % E
              -- (14,10) % E
              -- (15,10) % E
              -- (16,10);% E

\end{tikzpicture}
\end{center}

Next, we visualize the cyclic copies of $\pi''$ and $\Gamma^{(2)}(\pi^\bullet)$ together as follows:
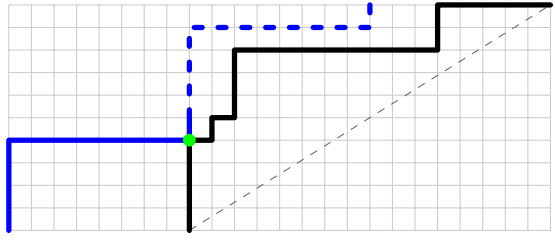
\begin{center}
    \begin{tikzpicture}[scale=0.3, >=stealth, baseline=(current bounding box.center)]

    % 1. Grid covering the extended area (-8 to 16 in x, 0 to 11 in y)
    \draw[step=1, gray!40, thin] (-8,0) grid (16,10);

    % 2. Reference Diagonal (optional, kept for context)
    \draw[dashed, black!60] (0,0) -- (16,10);

    % --- The Original Path (Now all BLACK) ---
    % Combined into one continuous draw command for simplicity
    \draw[black, line width=2pt, line cap=round, line join=round] 
        % Segment 1: NNN from (0,0) to (0,3)
        (0,0) -- (0,1) -- (0,2) -- (0,3)
        % Segment 2: NENENNNEEEEEE from (0,3) to (8,8)
              -- (0,4) -- (1,4) -- (1,5) -- (2,5) -- (2,6) -- (2,7) -- (2,8)
              -- (3,8) -- (4,8) -- (5,8) -- (6,8) -- (7,8) -- (8,8)
        % Segment 3: EEENNNEEEEE from (8,8) to (16,11)
              -- (9,8) -- (10,8) -- (11,8) -- (11,9) -- (11,10) -- (11,10)
              -- (12,10) -- (13,10) -- (14,10) -- (15,10) -- (16,10);

    % --- The New BLUE Paths ---
    
    % Part 1: Solid Blue NNNNEEEEEEEEN starting from (-8,0)
    % Ends at (0,5)
    \draw[blue, line width=2pt, line cap=round, line join=round] 
        (-8,0) -- (-8,1) -- (-8,2) -- (-8,3) -- (-8,4) % NNNN
               -- (-7,4) -- (-6,4) -- (-5,4) -- (-4,4) -- (-3,4) -- (-2,4) -- (-1,4) -- (0,4) % EEEEEEEE
               -- (0,5); % N

    % Part 2: Dotted Blue NNNNEEEEEEEEN continuing from (0,5)
    % Ends at (8,10)
    % Added 'dotted' style here
    \draw[blue, line width=2pt, loosely dashed, line cap=round, line join=round] 
        (0,5) -- (0,6) -- (0,7) -- (0,8) -- (0,9) % NNNN
              -- (1,9) -- (2,9) -- (3,9) -- (4,9) -- (5,9) -- (6,9) -- (7,9) -- (8,9) % EEEEEEEE
              -- (8,10); % N

    % Optional: Mark endpoint of blue path
    \fill[green] (0,4) circle (8pt);

\end{tikzpicture}
\end{center}
Note that there is a crossing at $(0,4)$, which is the last crossing, highlighted by a green dot. We apply the mix operation again to obtain:
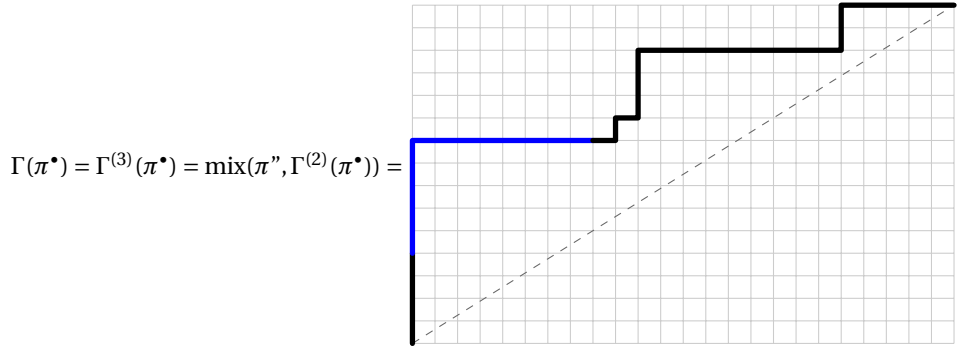
\begin{center}
    $\Gamma(\pi^\bullet)=\Gamma^{(3)}(\pi^\bullet)=\operatorname{mix(\pi'',\Gamma^{(2)}(\pi^\bullet))=}$
    \begin{tikzpicture}[scale=0.3, >=stealth, baseline=(current bounding box.center)]

    % 1. Grid covering the area 24x15
    \draw[step=1, gray!40, thin] (0,0) grid (24,15);

    % 2. Reference Diagonal (from start to end)
    \draw[dashed, black!60] (0,0) -- (24,15);

    % --- Segment 1: Black NNNN ---
    % (0,0) -> (0,4)
    \draw[black, line width=2pt, line cap=round, line join=round] 
        (0,0) -- (0,4);

    % --- Segment 2: Blue NNNNN EEEEEEEE ---
    % Starts at (0,4)
    % NNNNN -> (0,9)
    % EEEEEEEE -> (8,9)
    \draw[blue, line width=2pt, line cap=round, line join=round] 
        (0,4) -- (0,9) -- (8,9);

    % --- Segment 3: Black ENENNN EEEEEEEEE NNEEEEE ---
    % Starts at (8,9)
    \draw[black, line width=2pt, line cap=round, line join=round] 
        (8,9)   -- (9,9)    % E
                -- (9,10)   % N
                -- (10,10)  % E
                -- (10,13)  % NNN
                -- (19,13)  % EEEEEEEEE
                -- (19,15)  % NN
                -- (24,15); % EEEEE

\end{tikzpicture}
\end{center}
\end{example}

\begin{proof}[Proof of \Cref{prop: reformulation of rational shuffle}]
We proceed by verifying the three necessary properties of the map $\Gamma$: well-definedness, bijectivity, and statistic preservation.

$\bullet$ Well-definedness:
We must show that, at each step of the recursive construction, the operation $\mix$ is well defined.
Let $\Gamma^{(\ell-1)}(\pi^\bullet) = (\tau^{(1)},\dots,\tau^{(\ell-1)})$
denote the $((\ell-1)m,(\ell-1)n)$-parking function constructed thus far, and let
$\pi^{(\ell)}$ be the next cyclic $(m,n)$-parking function to be added.
It suffices to show that $\pi^{(\ell)}$ has a crossing with some $\tau^{(i)}$.

Suppose, for the sake of contradiction, that $\pi^{(\ell)}$ has no crossings with any of the $\tau^{(i)}$.
Since $P_{\tau^{(1)}}$ passes through the origin, the path $\pi^{(\ell)}$ must remain weakly to the left of each $\tau^{(i)}$, share no east steps with them, and satisfy the labeling constraint at every geometric intersection.
Moreover, since $(\tau^{(1)},\dots,\tau^{(\ell-1)})$ contains a shifted copy of $\pi^{(\ell-1)}$, we obtain $\pi^{(\ell-1)} \prec \pi^{(\ell)}$, which is a contradiction.

$\bullet$ Bijectivity:
We construct the inverse map $\Psi: \PF_{km,kn} \to \PTab_{m,n}(k)$ by inverting the recursive step.
Suppose we have an $(\ell m, \ell n)$-parking function $\Pi^{(\ell)} = \Gamma^{(\ell)}(\pi^\bullet)$ and seek to reconstruct the cyclic $(m,n)$-parking function $\pi^{(\ell)}$ and the predecessor path $\Pi^{(\ell-1)}$.

The procedure is as follows:
\begin{enumerate}
    \item Identify the \emph{maximal} index $i \in \{1, \dots, \ell-1\}$ such that the $i$-th and $(i+1)$-th components of $\Pi^{(\ell)}$ have a crossing.
    \item Let $v$ be the \emph{last} geometric point (in the standard path ordering) where such a crossing occurs.
    \item We perform the `un-mix' operation at $v$:
    \begin{itemize}
        \item The cyclic $(m,n)$-parking function $\pi^{(\ell)}$ is recovered by concatenating the portion of the $(i+1)$-th component ending at $v$ with the portion of the $i$-th component starting at $v$.
        \item The $i$-th component of the predecessor path $\Pi^{(\ell-1)}$ is restored by concatenating the portion of the $i$-th component ending at $v$ with the portion of the $(i+1)$-th component starting at $v$.
        \item The components at indices $j < i$ in $\Pi^{(\ell-1)}$ remain identical to the corresponding components in $\Pi^{(\ell)}$.
        \item All components at indices $j > i+1$ in $\Pi^{(\ell)}$ are shifted down by one index to form the components $j-1$ in $\Pi^{(\ell-1)}$.
    \end{itemize}
\end{enumerate}
This operation is the exact inverse of $\mix$.
By recursively applying this procedure from $\ell=k$ down to $1$, we uniquely recover the tuple in $\PTab_{m,n}(k)$.

$\bullet$ Preservation of statistics:
Note that the $\mix$ operation obviously preserves the skeleton of the paths. Furthermore, the maximality of the intersection point $v$ selected during the $\mix$ operation ensures that the number of positive crossings remains the same.
Consequently, by \Cref{lem:stat+pos depend only on skeleton}, we conclude that 
\begin{equation*}
\pdinv(P_{\pi^\bullet})+\ldinv(\pi^{\bullet})=\pdinv(P_{\Gamma(\pi^\bullet)})+\ldinv(\Gamma(\pi^{\bullet})).
\end{equation*}
\end{proof}

Given \(\pi^{\bullet}=(\pi^{(1)},\dots,\pi^{(\ell)})\in \PF_{\ell m,\ell n}\) and \(\tau\in \cPF_{m,n}\), if the last labeled north step of \(\pi^{(\ell)}\) can be extended with the first labeled north step of \(\tau\), then
\[
\mix(\tau,\pi^{\bullet})=(\pi^{(1)},\dots,\pi^{(\ell)},\tau).
\]
In particular, this always holds if \(P_{\tau}\) passes through the origin. We therefore conclude the following: for
\((\pi^{(1)},\dots,\pi^{(k)})\in \PTab_{m,n}(k)\), the path
\(P_{\Gamma(\pi^{\bullet})}\) passes through \((im,in)\) if and only if
\(P_{\pi^{(i+1)}}\) passes through the origin.

For a composition \(\alpha \models k\), let \(\PTab_{m,n}(k;\alpha)\) denote the
collection of \((\pi^{(1)},\dots,\pi^{(k)})\in \PTab_{m,n}(k)\) such that
\[
    \aseq(P_{\pi^{(i)}})_1 \begin{cases}
        = 0 & \text{if } i=1+\sum_{j=1}^{\ell}\alpha_j \text{ for some } \ell \ge 0, \\
        > 0 & \text{otherwise}.
    \end{cases}
\]
Then the bijection \(\Gamma\) restricts to a bijection
\[
\PTab_{m,n}(k;\alpha)\;\longrightarrow\;
\bigl\{\pi^{\bullet}\in \PF_{km,kn} : \touch(\pi^{\bullet})=\alpha \bigr\}.
\]

Combined with Theorem~\ref{thm: Rational shuffle} and
Proposition~\ref{prop: reformulation of rational shuffle}, we finally deduce:

\begin{corollary}\label{cor: compositional shuffle final}
For a pair of coprime positive integers \((m,n)\) and a composition
\(\alpha\models k\) of a positive integer \(k\), we have
\[
C_{\alpha}[-MX^{m,n}] \cdot 1
=
\sum_{\pi^\bullet\in\PTab_{m,n}(k;\alpha)}
q^{\stat(\pi^\bullet)}
\,t^{\mathrm{area}(P_{\pi^\bullet})}
\,x^{f_{\pi^\bullet}}.
\]
\end{corollary}

\section{Rational Shareshian--Wachs involution}\label{Sec: sw}

The main goal of this section is to prove the following identities:
\begin{align}
    \h_a \h_b = \h_b \h_a, 
    \qquad 
    \hhat_a \hhat_b &= \hhat_b \hhat_a, 
    \qquad 
    \hb_a \hb_b = \hb_b \hb_a, 
    \label{eq: sw} \\
    q(\hhat_b \hb_a - \hhat_{a-1} \hb_{b+1})
    &= \hb_a \hhat_b - \hb_{b+1} \hhat_{a-1},
    \label{eq: qsw}
\end{align}
where the operators $\h_i$, $\hhat_i$, and $\hb_i$ will be defined shortly. In the next section, we will work exclusively with these identities to derive our main results.

In particular, \eqref{eq: sw} may be viewed as a rational generalization of the symmetry property of chromatic quasisymmetric functions \cite[Theorem 4.5]{SW16}. We prove \eqref{eq: sw} by constructing a rational analogue of the Shareshian--Wachs involution, which we denote by $\SW$. To this end, we introduce the notion of a \emph{wiggle diagram}, which is purely combinatorial and elementary. We first define the map $\SW$ on wiggle diagrams and then apply it to tuples of cyclic parking functions by identifying them with wiggle diagrams via a rotation. Equation~\eqref{eq: qsw} is proved by a slight modification of the map $\SW$.

\begin{subsection}{Operators}\label{subsec: operators}
We introduce the operators $\h_i$, $\hhat_i$, and $\hb_i$, which will play a central role in the next section. Although these operators depend on the integers $m$ and $n$, we omit them from the notation for simplicity.

\begin{definition}
For a positive integer $a$, we define a \emph{chain} of length $a$ to be a tuple $(\pi^{(1)},\dots,\pi^{(a)})$ such that each $\pi^{(i)} \in \cPF_{m,n}$ and
\[
    \pi^{(i)} \prec \pi^{(i+1)} \quad \text{for } 1 \le i \le a-1.
\]
Let $\CH_a$ denote the set of all such chains. We further decompose $\CH_a$ as a disjoint union $\CH_a = \bar{\CH}_a \sqcup \hat{\CH}_a$, where a chain belongs to $\bar{\CH}_a$ if and only if the path $P_{\pi^{(1)}}$ passes through the origin.
\end{definition}

For a sequence of tuples $v=(\pi_1^{\bullet},\dots,\pi_{\ell}^{\bullet})$ where each $\pi_i^{\bullet}\in \CH_{a_i}$ for some $a_i$, we regard it as a large tuple in $\cPF^{\sum a_i}_{m,n}$ given by a concatenation. Under this identification, we may abuse the notation to write $\stat(v)$ and $z(v)$. By Lemma~\ref{lem:stat+pos depend only on skeleton}, permuting entries of $v$ within the same $\pi_i^{\bullet}$ does not change the value $\stat$.

Let $\mathbb{C}(q,t)[\mathbf{z}]$ be the polynomial ring in indeterminates $\{z_{ijk}\}$, where $1 \le i \le n$, $j \in \mathbb{Z}_{\ge 0}$, and $k \in \mathbb{Z}_{\ge 1}$. We define $\mathcal{D} \subset \mathbb{C}(q,t)[\mathbf{z}]$ to be the sub-algebra generated by the elements $z(\pi)$ for $\pi \in \cPF_{m,n}$.

\begin{definition}
We define an operator $\h_a$ acting on $\mathcal{D}$ as follows. Note that $\mathcal{D}$ is linearly generated by elements of the form $z(\tau^{\bullet})$, where $\tau^{\bullet}$ is a tuple in $\cPF_{m,n}^k$. We set
\begin{equation*}
    \h_a \cdot z(\tau^{\bullet})
    =
    \sum_{\pi^\bullet \in \CH_a}
    q^{\stat(\tau^{\bullet},\pi^\bullet)-\stat(\tau^{\bullet})} \,
    z(\tau^{\bullet},\pi^\bullet).
\end{equation*}
For another tuple $\nu^{\bullet}$ satisfying $z(\nu^{\bullet}) = z(\tau^{\bullet})$, by Lemma~\ref{lem:stat+pos depend only on skeleton} it is straightforward to check that
\[
\stat(\tau^{\bullet},\pi^\bullet)-\stat(\tau^{\bullet})
=
\stat(\nu^{\bullet},\pi^\bullet)-\stat(\nu^{\bullet}),
\]
which guarantees that the action of $\h_a$ is well-defined.  The operators $\hhat_a$ and $\hb_a$ are defined in a similar manner by replacing $\CH_a$ with $\hat{\CH}_a$ or $\bar{\CH}_a$, respectively. Note that we trivially have $\h_a=\hhat_a+\hb_a$. For degenerate cases, we set
$\h_0 = \hhat_0 = 1,
\hb_0 = 0$,
and
$\h_a = \hhat_a = \hb_a = 0$, for $a < 0.
$\end{definition}

For example, we have
\begin{equation*}
     \hat{\mathfrak{h}}_2 \bar{\mathfrak{h}}_3 \mathfrak{h}_4 \cdot 1=\sum_{(\pi^{\bullet},\tau^{\bullet},\nu^{\bullet})\in \CH_4\times \bar{\CH}_3\times \hat{\CH}_2} q^{\stat(\pi^{\bullet},\tau^{\bullet},\nu^{\bullet})}z(\pi^{\bullet},\tau^{\bullet},\nu^{\bullet}).
\end{equation*}

The operators $\h_i$, $\hhat_i$, and $\hb_i$ do not commute in general; however, operators of the same type do commute, as stated in \eqref{eq: sw}. To show that $\h_a \h_b = \h_b \h_a$, it suffices to construct a bijection between $\CH_a \times \CH_b$ and $\CH_b \times \CH_a$ that preserves both the statistic $\stat$ and the skeleton $z$. Such a bijection will also imply
\[
\hhat_a \hhat_b = \hhat_b \hhat_a
\quad \text{and} \quad
\hb_a \hb_b = \hb_b \hb_a.
\]
To this end, we delve into the world of wiggle diagrams.

\begin{subsection}{Combinatorics of wiggle diagrams}

We introduce wiggle diagrams and construct a map $\SW$ on them. Our goal is to prove Proposition~\ref{prop: sw wiggle diagram}. Later, we will prove \eqref{eq: sw} by simply regarding an element of $\CH_a \times \CH_b$ as a wiggle diagram and applying the map $\SW$.
\end{subsection}

\begin{definition}
    For positive integers $a$ and $b$, an \emph{$(a,b)$-wiggle diagram} consists of a collection of $a$ black strands and $b$ red strands in $\mathbb{R}^{2}$. Each strand is defined by a continuous curve $\gamma: \mathbb{R} \to \mathbb{R}^2$. We construct these strands by taking a fundamental curve from $(r,0)$ to $(r,1)$ for some $r \in \mathbb{R}$ and extending it periodically in the vertical direction such that the resulting curve is continuous and $y$-monotonic. Specifically, we require:
    \begin{enumerate}
        \item  Each strand intersects every horizontal line $y=c$ at exactly one point.
        \item  Strands of the same color do not intersect.
        \item  Strands of different colors intersect at finitely many points within the band $0 \leq y < 1$.
    \end{enumerate}
    Due to the periodicity, the diagram is fully determined by its restriction to the horizontal band $0 \leq y < 1$, which we call the \emph{principal band}.
\end{definition}

Later, each strand in a wiggle diagram corresponds to an element of $\cPF_{m,n}$. For a wiggle diagram $\Pi$, we establish the following terminology. An intersection point between a red strand and a black strand is called a \emph{crossing}. We classify the sign of a crossing based on the relative orientation of the strands:
    a crossing is \emph{positive} (respectively \emph{negative}) if the strand entering the crossing from the south-west direction is a black (respectively red) strand.

The wiggle diagram naturally partitions $\mathbb{R}^{2}$ into connected open components called \emph{regions}. By definition, no strand passes through the interior of a region. Around any crossing $v$, there are locally four distinct adjacent regions. Since every strand is $y$-monotonic (strictly increasing in the vertical direction), we designate these as the \emph{north}, \emph{south}, \emph{east}, and \emph{west} regions of $v$. Specifically, the region strictly above $v$ is the north region, the region strictly below $v$ is the south region, and the east and west regions are defined accordingly. Now we associate an infinite planar graph $G(\Pi)$, defined as follows:
\begin{enumerate}
    \item Vertices: The vertices of $G(\Pi)$ are the regions of the wiggle diagram $\Pi$.
     \item Edges: Two vertices in $G(\Pi)$ are connected by an edge if the corresponding regions are the west and east regions of the same crossing. There might be multiple edges between two vertices.
\end{enumerate}
By construction, $G(\Pi)$ admits a natural embedding into $\mathbb{R}^2$ as an infinite planar graph whose edges pass through the crossings. Consider a connected component of $G(\Pi)$. If this component has infinitely many regions or contains an unbounded region, then we call the regions in this component \emph{regular regions}. The set of regular regions of $\Pi$ is denoted by $\SR(\Pi)$.

We define $\SW(\Pi)$ to be the configuration obtained by changing the colors of the segments (i.e., parts of strands) that are surrounded by regular regions. The following lemma is useful for determining whether a region in $\Pi$ is regular. For a region $R$, we denote by $R^{+}$ the region obtained from $R$ by a periodic shift upward by one period. Note that for an unbounded region $R$, we have $R^+=R$.
\begin{lem}\label{lem: infinite region classification}
    For a wiggle diagram $\Pi$, a bounded region $R$ is regular if and only if $R$ and $R^{+}$ are connected in $G(\Pi)$.
\end{lem}
\begin{proof}
    The converse direction is trivial. For the forward direction, assume that $R \in \SR(\Pi)$. If $R$ is connected to some unbounded region $R'$, then, by a periodic shift, $R^{+}$ is also connected to $(R')^{+} = R'$. Therefore, we conclude that $R$ and $R^{+}$ are connected. Now assume that the connected component containing $R$ has only bounded regions. Then this component must contain infinitely many regions. Thus, there exists a region $Q$ such that $Q$ and $Q^{k+}$ lie in the same component for some $k \ge 1$. Let $P$ be a path from $Q$ to $Q^{k+}$, and consider its periodic shift $P^{+}$, which is a path from $Q^{+}$ to $Q^{(k+1)+}$. Since the graph $G(\Pi)$ is planar, the two paths $P$ and $P^{+}$ must share a vertex. Hence, $Q$ and $Q^{+}$ are connected. Consequently, we deduce that $R$ and $R^{+}$ are connected. 
\end{proof}

\begin{proposition}\label{prop: sw wiggle diagram}
For an $(a,b)$-wiggle diagram $\Pi$, the configuration $\SW(\Pi)$ is a $(b,a)$-wiggle diagram. Moreover, the number of positive crossings in the principal band of $\Pi$ is equal to the number of positive crossings in the principal band of $\SW(\Pi)$.
\end{proposition}

We prove Proposition~\ref{prop: sw wiggle diagram} by applying a sequence of operations (see Definition~\ref{def: operation}) to reduce to the base cases described in Proposition~\ref{prop: wiggle classification}. We begin by introducing the necessary terminology.

A connected portion of a strand joining two points on it is called a \emph{segment} of the strand. A shape bounded by two segments—one necessarily from a red strand and the other from a black strand—is called a \emph{bigon}. The segments defining a bigon are referred to as the \emph{boundary red segment} and the \emph{boundary black segment} of the bigon. A region that is a bigon is called \emph{simple}. Note that for any bigon, there exists a simple region contained inside it. For example, in Figure~\ref{fig: bigon}, a shape $R_1 \cup R_2$ is a bigon as it is bounded by two segments connecting $v_1$ and $v_2$. However, it is not a region and we can find a simple region $R_1$ inside this bigon.

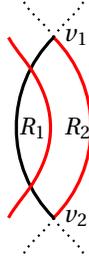
\begin{figure}
 \begin{tikzpicture}[scale=0.4]

    \coordinate (v1) at (0,3);

    \coordinate (v2) at (0,-3);

    \node[right] at (v1) {$v_1$};

    \node[right] at (v2) {$v_2$};

    \draw[red, very thick] (v1) to[out=-45, in=45] (v2);

    % Dotted extension for the red curve (Stays on the RIGHT)

    % It leaves v1 going up-right (angle 60)

    \draw[dotted, thick, black] (v1) to[out=60, in=-120] (1.0, 4.2);

    % Bottom extension (Stays on the right)

    \draw[dotted, thick, black] (v2) to[out=-60, in=120] (1.0, -4.2);

    % --- Region R1 (The left side) ---

    % Main black curve (Left boundary of the big lens)

    \draw[black, very thick] (v1) to[out=-135, in=135] (v2);

    % Dotted extension for the black curve (Stays on the LEFT)

    % It leaves v1 going up-left (angle 120)

    \draw[dotted, thick, black] (v1) to[out=120, in=-60] (-1.0, 4.2);

    % Bottom extension (Stays on the left)

    \draw[dotted, thick, black] (v2) to[out=-120, in=60] (-1.0, -4.2);

    % --- The Middle Red Curve ---

    % This creates the small region R1 on the left side

    % We define intersection points v3 and v4 for the labels if needed

    \coordinate (v3) at (-0.6, 1.7);

    \coordinate (v4) at (-0.6, -1.7);

    % The curve itself

    \draw[red, very thick] (-1.5, 3) to[out=-60, in=120] (v3) 

                           to[out=-60, in=60] (v4) 

                           to[out=-120, in=60] (-1.5, -3);

    % --- Labels for the Regions ---

    \node at (-0.7, 0) { $R_1$};

    \node at (0.8, 0) { $R_2$};

\end{tikzpicture}

\caption{A bigon and a simple region inside it.}\label{fig: bigon}
\end{figure}

A simple region is called \emph{full} if the difference between the $y$-coordinates of its highest and lowest points is exactly 1. Otherwise, the simple region is called \emph{proper}.

\begin{definition}\label{def: operation}
We define the following operations on $\Pi$.
\begin{itemize}
    \item Operation (I): If there exists a proper simple region such that no other strands touch its boundary, we contract the region.
    \item Operation (II): If there exists a simple region such that exactly one other strand touches its boundary only at a point, we detach the touching point.
\end{itemize}
Operations (I) and (II) must be performed periodically: if we perform operation (I) or (II) in a band $r \leq y < r+1$, then the same operation must be performed in every shifted band $r+z \leq y < r+z+1$ for all $z \in \mathbb{Z}$.    
\end{definition}

\begin{example}
   On the left of Figure~\ref{fig: wiggle example}, the $(4,2)$-wiggle diagram $\Pi$ is shown. The region $R_{1}$ is a full simple region, while $R_2$ is a proper simple region. Negative
crossings in a principal band are indicated by blue dots. In the graph $G(\Pi)$, there is an edge connecting $R_{3}$ and $R_{4}$,
and also an edge connecting $R_{3}$ and $R_{4}^{+}$. Therefore, $R_{4}$
and $R_{4}^{+}$ lie in the same connected component in the graph $G(\Pi)$, which implies that
$R_{3}, R_{4} \in \Reg(\Pi)$. Note that regions $R_{1}$ and $R_{2}$ together with their periodic shifts are not regular. Accordingly, $\SW(\Pi)$ is shown on the right. We can check that the number of positive crossings remains the same in $\Pi$ and $\SW(\Pi)$.

\newcommand{\Mone}{
    % Vertical lines
     \foreach \x in {2, 4, 6, 8} {
        \draw[very thick, black] (\x, 0) -- (\x, 8);
    }

    % Horizontal dashed lines and labels
    \draw[dashed, very thick, blue] (1.5, 0) -- (8.5, 0);
    \node[blue, left, font=\Large] at (1.5, 0) {$y=0$};

    \draw[dashed, very thick, blue] (1.5, 4) -- (8.5, 4);
    \node[blue, left, font=\Large] at (1.5, 4) {$y=1$};

%leftmost red strand
      \draw[red, very thick, line cap=round] (3, 0) .. controls (2.5, 1) .. (2,1.5);
      \draw[red, very thick, line cap=round] (2,1.5) .. controls (3, 3) .. (3,4);
       \draw[red, very thick, line cap=round] (3, 4) .. controls (2.5, 5) .. (2,5.5);
      \draw[red, very thick, line cap=round] (2,5.5) .. controls (3, 7) .. (3,8);

%right red strand
      %\draw[red, very thick, line cap=round] (7, 0) .. controls (1.5, 1) .. (8,3.5);
      \draw[red, very thick, line cap=round] (7, 0) .. controls (5.5, 0.5) .. (4,0.7);
       \draw[red, very thick, line cap=round] (4, 0.7) .. controls (2.5, 1) .. (4,1.9);
        \draw[red, very thick, line cap=round] (4,1.9) .. controls (6,2.5) .. (8,3.5);
      \draw[red, very thick, line cap=round] (8, 3.5) .. controls (7.5, 3.7) .. (7,4);
      \begin{scope}[shift={(0,4)}]
         \draw[red, very thick, line cap=round] (7, 0) .. controls (5.5, 0.5) .. (4,0.7);
       \draw[red, very thick, line cap=round] (4, 0.7) .. controls (2.5, 1) .. (4,1.9);
        \draw[red, very thick, line cap=round] (4,1.9) .. controls (6,2.5) .. (8,3.5);
      \draw[red, very thick, line cap=round] (8, 3.5) .. controls (7.5, 3.7) .. (7,4);
    \end{scope}

    % Blue dots at intersections
    \fill[blue] (4, 1.9) circle (2.5pt);
    \fill[blue] (6, 2.55) circle (2.5pt);
    \fill[blue] (8, 3.5) circle (2.5pt);

    % Region labels
    \node[font=\huge\bfseries] at (2.5, 3.3) {$R_1$};
    \node[font=\huge\bfseries] at (3.6, 1.3) {$R_2$};
    \node[font=\huge\bfseries] at (5, 3.5) {$R_3$};
     \node[font=\huge\bfseries] at (7, 1.5) {$R_4$};
    \node[font=\huge\bfseries] at (7, 5.5) {$R_4^{+}$};

    \node[font=\huge\bfseries] at (5, -1) {$\Pi$};
}

\newcommand{\Mtwo}{
  \foreach \x in {2} {
        \draw[very thick, black] (\x, 0) -- (\x, 8);
    }

    \foreach \x in {  6, 8} {
        \draw[very thick, red] (\x, 0) -- (\x, 8);
    }
    \foreach \x in {4} {
        \draw[very thick, red] (\x, 0) -- (\x, 0.7);
    }
     \foreach \x in {4} {
        \draw[very thick, black] (\x, 0.7) -- (\x, 1.9);
    }
    \foreach \x in {4} {
        \draw[very thick, red] (\x, 1.9) -- (\x, 4.7);
    }
     \foreach \x in {4} {
        \draw[very thick, black] (\x, 4.7) -- (\x, 5.9);
    }
    \foreach \x in {4} {
        \draw[very thick, red] (\x, 5.9) -- (\x, 8);
    }

    % Horizontal dashed lines and labels
    \draw[dashed, very thick, blue] (0.5, 0) -- (8.5, 0);
    \node[blue, left, font=\Large] at (0.5, 0) {$y=0$};

    \draw[dashed, very thick, blue] (0.5, 4) -- (8.5, 4);
    \node[blue, left, font=\Large] at (0.5, 4) {$y=1$};

%leftmost red strand
      \draw[red, very thick, line cap=round] (3, 0) .. controls (2.5, 1) .. (2,1.5);
      \draw[red, very thick, line cap=round] (2,1.5) .. controls (3, 3) .. (3,4);
       \draw[red, very thick, line cap=round] (3, 4) .. controls (2.5, 5) .. (2,5.5);
      \draw[red, very thick, line cap=round] (2,5.5) .. controls (3, 7) .. (3,8);

%right red strand
      %\draw[red, very thick, line cap=round] (7, 0) .. controls (1.5, 1) .. (8,3.5);
      \draw[black, very thick, line cap=round] (7, 0) .. controls (5.5, 0.5) .. (4,0.7);
       \draw[red, very thick, line cap=round] (4, 0.7) .. controls (2.5, 1) .. (4,1.9);
        \draw[black, very thick, line cap=round] (4,1.9) .. controls (6,2.5) .. (8,3.5);
      \draw[black, very thick, line cap=round] (8, 3.5) .. controls (7.5, 3.7) .. (7,4);
      \begin{scope}[shift={(0,4)}]
         \draw[black, very thick, line cap=round] (7, 0) .. controls (5.5, 0.5) .. (4,0.7);
       \draw[red, very thick, line cap=round] (4, 0.7) .. controls (2.5, 1) .. (4,1.9);
        \draw[black, very thick, line cap=round] (4,1.9) .. controls (6,2.5) .. (8,3.5);
      \draw[black, very thick, line cap=round] (8, 3.5) .. controls (7.5, 3.7) .. (7,4);
    \end{scope}

    % Blue dots at intersections
    \fill[blue] (4, 1.9) circle (2.5pt);
    \fill[blue] (4, 0.7) circle (2.5pt);
    \fill[blue] (6, 0.35) circle (2.5pt);

 \node[font=\huge\bfseries] at (5, -1) {$\SW(\Pi)$};

}

\newcommand{\Mthree}{ \foreach \x in {2, 4, 6, 8} {
        \draw[very thick, black] (\x, 0) -- (\x, 4);
    }

    % Horizontal dashed lines and labels
    \draw[dashed, very thick, blue] (1.5, 0) -- (8.5, 0);
    \node[blue, left, font=\Large] at (1.5, 0) {$y=0$};

    \draw[dashed, very thick, blue] (1.5, 4) -- (8.5, 4);
    \node[blue, left, font=\Large] at (1.5, 4) {$y=1$};

%leftmost red strand
      \draw[red, very thick, line cap=round] (3, 0) .. controls (2.5, 1) .. (2,1.5);
      \draw[red, very thick, line cap=round] (2,1.5) .. controls (3, 3) .. (3,4);

%right red strand
      %\draw[red, very thick, line cap=round] (7, 0) .. controls (1.5, 1) .. (8,3.5);
      \draw[red, very thick, line cap=round] (7, 0) .. controls (5.5, 0.5) .. (4,0.7);
       \draw[red, very thick, line cap=round] (4, 0.7) .. controls (2.5, 1) .. (4,1.9);
        \draw[red, very thick, line cap=round] (4,1.9) .. controls (6,2.5) .. (8,3.5);
      \draw[red, very thick, line cap=round] (8, 3.5) .. controls (7.5, 3.7) .. (7,4);

    \node[font=\huge\bfseries] at (3.6, 1.3) {$R$};

    \node[font=\huge\bfseries] at (5, -1) {$\Pi$};
}

\newcommand{\Mfour}{ \foreach \x in {2, 4, 6, 8} {
        \draw[very thick, black] (\x, 0) -- (\x, 4);
    }

    % Horizontal dashed lines and labels
    \draw[dashed, very thick, blue] (1.5, 0) -- (8.5, 0);
    \node[blue, left, font=\Large] at (1.5, 0) {$y=0$};

    \draw[dashed, very thick, blue] (1.5, 4) -- (8.5, 4);
    \node[blue, left, font=\Large] at (1.5, 4) {$y=1$};

%leftmost red strand
      \draw[red, very thick, line cap=round] (3, 0) .. controls (2.5, 1) .. (2,1.5);
      \draw[red, very thick, line cap=round] (2,1.5) .. controls (3, 3) .. (3,4);

%right red strand
      %\draw[red, very thick, line cap=round] (7, 0) .. controls (1.5, 1) .. (8,3.5);
      \draw[red, very thick, line cap=round] (7, 0) .. controls (5.5, 0.5) .. (4,1.3);
        \draw[red, very thick, line cap=round] (4,1.3) .. controls (6,2.5) .. (8,3.5);
      \draw[red, very thick, line cap=round] (8, 3.5) .. controls (7.5, 3.7) .. (7,4);

    \node[font=\huge\bfseries] at (5.3, 1.3) {$R'$};

    \node[font=\huge\bfseries] at (5, -1) {$\Pi'$};
}

\newcommand{\Mfive}{ \foreach \x in {2, 4, 6, 8} {
        \draw[very thick, black] (\x, 0) -- (\x, 4);
    }

    % Horizontal dashed lines and labels
    \draw[dashed, very thick, blue] (1.5, 0) -- (8.5, 0);
    \node[blue, left, font=\Large] at (1.5, 0) {$y=0$};

    \draw[dashed, very thick, blue] (1.5, 4) -- (8.5, 4);
    \node[blue, left, font=\Large] at (1.5, 4) {$y=1$};

%leftmost red strand
      \draw[red, very thick, line cap=round] (3, 0) .. controls (2.5, 1) .. (2,1.5);
      \draw[red, very thick, line cap=round] (2,1.5) .. controls (3, 3) .. (3,4);

%right red strand
      %\draw[red, very thick, line cap=round] (7, 0) .. controls (1.5, 1) .. (8,3.5);
      \draw[red, very thick, line cap=round] (7, 0) .. controls (3.3, 1) .. (8,3.5);
      \draw[red, very thick, line cap=round] (8, 3.5) .. controls (7.5, 3.7) .. (7,4);

    \node[font=\huge\bfseries] at (5, -1) {$\Pi''$};
}

\begin{figure}[htbp]
\centering
\resizebox{\textwidth}{!}{
\begin{tikzpicture}[xscale=1.5, yscale=1]

    \begin{scope}[shift={(0,0)}]
        \Mone
    \end{scope}

    \begin{scope}[shift={(10,0)}]
        \Mtwo
    \end{scope}

\end{tikzpicture}
}
\caption{Wiggle diagram $\Pi$ and its $\SW(\Pi)$.}
\label{fig: wiggle example}
\end{figure}
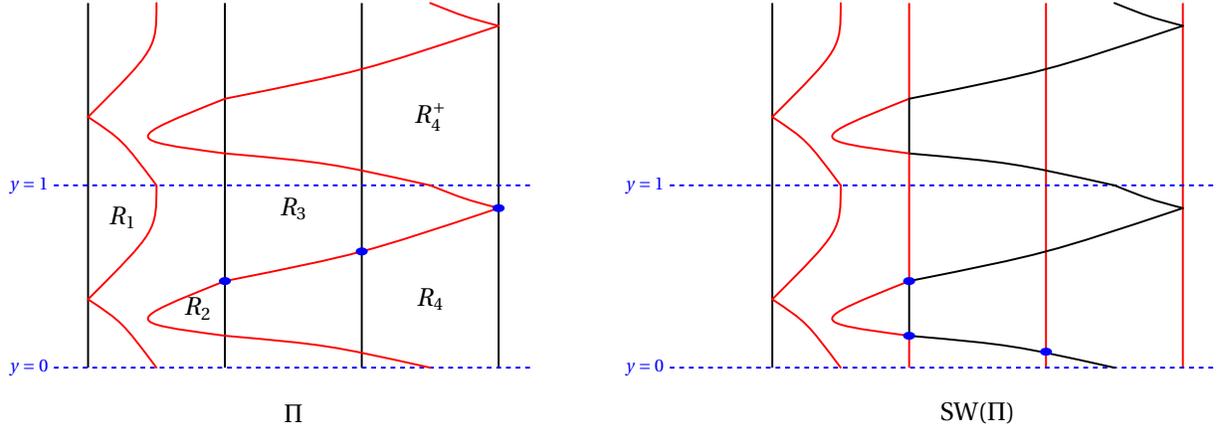

See Figure~\ref{fig: operation example}. We can perform operation (I) on the simple region $R$ in $\Pi$, shown on the
left. The resulting wiggle diagram
$\Pi'$ is shown in the middle. We can then perform operation (II) on the
simple region $R'$ in $\Pi'$, obtaining $\Pi''$ shown on the right.

\begin{figure}[htbp]
\centering
\resizebox{\textwidth}{!}{
\begin{tikzpicture}[xscale=1.5, yscale=1]

    \begin{scope}[shift={(0,0)}]
        \Mthree
    \end{scope}

    \begin{scope}[shift={(10,0)}]
        \Mfour
    \end{scope}

    \begin{scope}[shift={(20,0)}]
        \Mfive
    \end{scope}

\end{tikzpicture}
}
\caption{Applications of operations (I) and (II).}
\label{fig: operation example}
\end{figure}
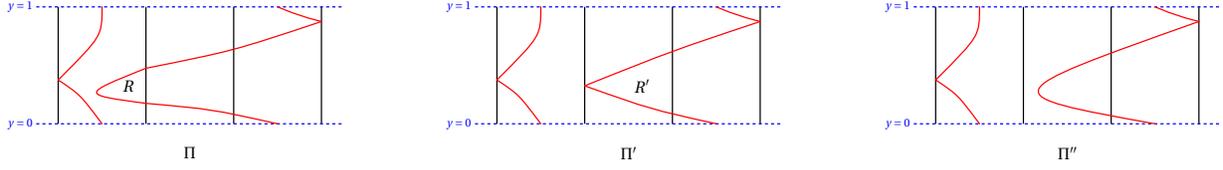
    
\end{example}
    \begin{lem}\label{lem: operation valid}
    Given an $(a,b)$-wiggle diagram $\Pi$, let $\Pi'$ be the diagram obtained by applying operation (I) or (II). If Proposition~\ref{prop: sw wiggle diagram} holds for $\Pi'$, then it also holds for $\Pi$.
\end{lem}

\begin{proof}
    Case 1):
    Suppose $\Pi'$ is obtained by applying operation (I) to a simple region $R$ in $\Pi$. Since $R$ is a simple region with no other strands touching its boundary, it is an isolated vertex in $G(\Pi)$. Clearly, $R$ is not regular. We deduce that $\SW(\Pi)$ is identical to $\SW(\Pi')$ except for the re-insertion of the original region $R$ (with its original colors of boundary segments). We conclude that if $\SW(\Pi')$ is a $(b,a)$-wiggle diagram then so is $\SW(\Pi)$.

    Now we analyze positive crossings. Without loss of generality, let the black boundary segment be located to the left of $R$. Then the number of positive crossings increases by exactly 1 when going from $\Pi'$ to $\Pi$, and similarly from $\SW(\Pi')$ to $\SW(\Pi)$.
    
    Case 2):
    Now assume $\Pi'$ is obtained by applying operation (II) to a simple region $B$ where the other strand touches the boundary only once, as illustrated on the left of Figure~\ref{fig: operation 2 example}. Let $\ell_1$ and $\ell_2$ be boundary segments of $B$ from the left and without loss generality assume that the segment $\ell_3$ (with a black color) touches the segment $\ell_2$ at a touching point $v$. Denote by $D$, $C$, and $E$ the regions north, east, and south of $v$, and let $A$ be a region to the left of $B$ across the segment $\ell_1$.  In the graph $G(\Pi)$, there exist edges connecting $A$ to $D$ and $A$ to $E$. 
    
    In $\Pi'$, the touching segment $\ell_3$ is detached, causing regions $D$ and $E$ to merge into a single region $D'$. In the new graph $G(\Pi')$, $D'$ remains connected to $A$. We conclude that  $A, D, E \in \Reg(\Pi)$ if and only if $A, D' \in \Reg(\Pi')$. Note that graph structures for $G(\Pi')$ and $G(\Pi)$ other than those local configurations are identical. 

 We analyze the color changes of $\ell_1,\ell_2,\ell_3$ in $\SW(\Pi)$ versus $\SW(\Pi')$:
    \begin{itemize}
        \item If $A, D, E \notin \Reg(\Pi)$ (and thus $A, D' \notin \Reg(\Pi')$), then the segments $\ell_1, \ell_2, \ell_3$ do not change color in either $\SW(\Pi)$ or $\SW(\Pi')$. The local structure is preserved.
        \item If $A, D, E \in \Reg(\Pi)$ and $C\notin  \Reg(\Pi')$, then we have $B,C\notin \Reg(\Pi')$ and also $B,C\notin \Reg(\Pi)$. Segments $\ell_1, \ell_2, \ell_3$ do not change color in either $\SW(\Pi)$ or $\SW(\Pi')$.
         \item If $A, D, E \in \Reg(\Pi)$ and $C\in  \Reg(\Pi')$, then we have $B\notin \Reg(\Pi')$ and  $B,C\in \Reg(\Pi)$. Segments $\ell_1, \ell_2, \ell_3$ all change color in $\SW(\Pi)$ while only $\ell_3$ changes color in $\SW(\Pi')$.
    \end{itemize}
    In all sub-cases, the amount that the number of positive crossings varies from $\Pi$ to $\Pi'$ is identical to that of $\SW(\Pi)$ to $\SW(\Pi')$.
\end{proof}

    \begin{center}
\begin{figure}
\begin{tikzpicture}[>=stealth, line width=1pt, scale=1]
    
    \tikzset{
        black curve/.style={black, thick, line cap=round},
        red curve/.style={red, very thick, line cap=round},
        dashed line/.style={black, dashed, thin}
    }

    % --- LEFT DIAGRAM (PI) ---
    % MODIFIED: l3 is now tangent to l2
    \begin{scope}[xshift=0cm]
        % Region labels
        \node at (-1.2, 1) {\Large $A$};
        \node at (0, 1) {\Large $B$};
        \node at (2.0, 1.0) {\Large $C$};
        \node at (0.8, 2.3) {$D$}; % Moved label D
        \node at (0.7, -0.3) {$E$}; % Moved label E

        % The Bigon (Lens shape)
        % Left side (black) l1
        \draw[black curve] (0, 0) to[bend left=60] node[pos=0.8, left] {$\ell_1$} (0, 2);
        
        % Right side (red) l2
        % We define a coordinate 'tangent_point' on this curve for l3 to touch.
        \draw[red curve] (0, 0) to[bend right=60] coordinate[pos=0.5] (tangent_point)  (0, 2);
        \node at (0.5, 1.9) {\textcolor{red}{$\ell_2$}};
         %\node[pos=0.8, right, black] {$\ell_2$}
        % The Tangent curve l3
        % It comes from the bottom right, touches 'tangent_point', and goes to top right.
        \draw[black curve] (1.2, -0.5) 
            to[out=110, in=-85] (tangent_point) 
            to[out=95, in=-110] node[pos=0.6, right] {$\ell_3$} (1.4, 2.5);

        % Dashed whiskers
        \draw[dashed line] (0, 2) -- (-0.5, 2.6);
        \draw[dashed line] (0, 2) -- (0.5, 2.6);
        \draw[dashed line] (0, 0) -- (-0.5, -0.6);
        \draw[dashed line] (0, 0) -- (0.5, -0.6);
        \node at (0.7, 1) {$v$};
        % Pi Label
        \node at (0, -1.5) {$\Pi$};
    \end{scope}

    % --- RIGHT DIAGRAM (PI PRIME) ---
    \begin{scope}[xshift=6cm]
        % Region labels
        \node at (-1, 1) {\Large $A$};
        \node at (0, 1) {\Large $B$};
        \node at (2.2, 1.2) {\Large $C$};
        \node at (1, 1.2) {\Large $D'$};

        % The Bigon (Lens shape)
        % Left side (black) l1
        \draw[black curve] (0, 0) to[bend left=60] node[pos=0.8, left] {$\ell_1$} (0, 2);
        \node at (0.5, 1.9) {\textcolor{red}{$\ell_2$}};
        % Right side (red) l2
        \draw[red curve] (0, 0) to[bend right=60]  (0, 2);

        % The separated line l3
        \draw[black curve] (1.8, -0.2) to[bend left=30] node[pos=0.8, right] {$\ell_3$} (1.6, 2.3);

        % Dashed whiskers
        \draw[dashed line] (0, 2) -- (-0.5, 2.6);
        \draw[dashed line] (0, 2) -- (0.5, 2.6);
        \draw[dashed line] (0, 0) -- (-0.5, -0.6);
        \draw[dashed line] (0, 0) -- (0.5, -0.6);
        
        % Pi Prime Label
        \node at (0, -1.5) {$\Pi'$};
    \end{scope}
\end{tikzpicture}\caption{Local configurations for operation (II).}\label{fig: operation 2 example}
\end{figure}
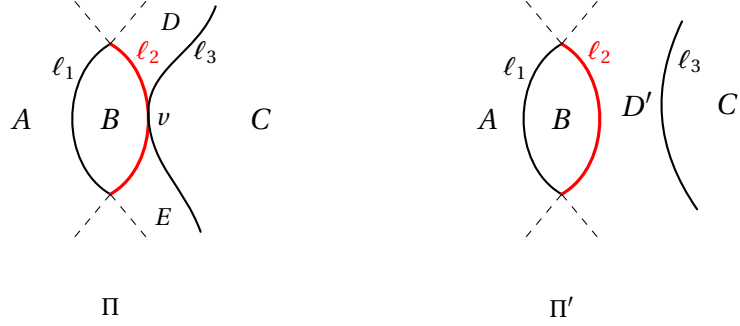
\end{center}

  Since operations (I) and (II) strictly decrease the number of crossings, they cannot be applied infinitely many times to a given wiggle diagram $\Pi$. Thus, any sequence of these operations must eventually terminate. We provide a classification of the resulting \emph{reduced wiggle diagrams}: those admitting no further applications of operation (I) or (II).

\begin{proposition}\label{prop: wiggle classification}
    Let $\Pi$ be a reduced $(a,b)$-wiggle diagram. Then there exist a subset $A$ of the black strands and a subset $B$ of the red strands, together with a bijection $\phi: A \to B$, such that:
    \begin{itemize}
        \item Every black strand $\ell \in A$ intersects the corresponding red strand $\phi(\ell)$ exactly once in the principal band, and intersects no other red strands.
        \item Black strands not in $A$ and red strands not in $B$ do not participate in any crossings.
    \end{itemize}
\end{proposition}

For a reduced wiggle diagram $\Pi$ as in Proposition~\ref{prop: wiggle classification}, each pair $\ell \in A$ and $\phi(\ell)\in B$ forms a full simple region (see, for example, how the region $R_1$ is formed in Figure~\ref{fig: wiggle example}). Such a full simple region is isolated in $G(\Pi)$ and hence does not belong to $\Reg(\Pi)$. All other regions are unbounded, and therefore belong to $\Reg(\Pi)$. Consequently, $\SW(\Pi)$ is obtained by preserving the colors of the strands corresponding to elements of $A$ or $B$, and changing the colors of all other strands. It is immediate that $\SW(\Pi)$ is a $(b,a)$-wiggle diagram, and in this case positive crossings (respectively, negative crossings) are sent to positive crossings (respectively, negative crossings). We conclude that Proposition~\ref{prop: sw wiggle diagram} holds for reduced wiggle diagrams.
 
Although the statement of Proposition~\ref{prop: wiggle classification} may appear intuitively plausible at first glance, its rigorous proof requires some technical care (see Lemmas~\ref{y} and~\ref{z}). While Lemma~\ref{z} is apparently stronger than Lemma~\ref{y}, it is necessary to establish Lemma~\ref{y} first in order to prove Lemma~\ref{z}.

\begin{lem}\label{y}
    For a wiggle diagram $\Pi$, if there exists a simple region such that exactly one strand (other than those forming the boundary) touches its boundary, then operation (I) or (II) can be applied somewhere in $\Pi$.
\end{lem}

\begin{proof}
     Assume there exists such a simple region. Without loss of generality, we take a simple region $R$ which is the leftmost (relative to the ordering of black strands) among simple regions with the property that the black boundary segment is on the left and a red strand is touching the black boundary segment. Let $\ell$ be the red strand touching the black boundary segment.

    Case 1): If $\ell$ touches the boundary black segment of $R$ only once, then we can apply operation (II) to $R$.

    Case 2): Assume $\ell$ touches the boundary black segment of $R$ at least twice. Let $v$ and $v'$ be two distinct touching points. Let $\ell'_1$ be a black strand that forms the boundary black segment of $R$ and we enumerate the black strands to the left of $\ell'_1$ as $\ell'_2, \ell'_3, \dots$ (from right to left). Trace the path of $\ell$ from $v$ to $v'$. Let $w_1, w_2, \dots, w_p$ be the sequence of points where $\ell$ intersects any black strand $\{\ell'_k\}$. Let $a_j$ be the index of the black strand containing $w_j$, i.e., $w_j \in \ell \cap \ell'_{a_j}$.
    Note that $w_1 = v \in \ell'_1$ (so $a_1=1$) and $w_p = v' \in \ell'_1$ (so $a_p=1$). Since strands cannot jump over adjacent regions without crossing the boundaries, we have $|a_j - a_{j+1}| \le 1$. Since the sequence $a_1,\dots,a_p$ starts and ends at $1$, there must exist a turnaround index $j$. This leads to two subcases:

    Subcase 2a): There exists an index $j$ such that $a_{j+1} = a_j + 1$ and $a_{j+2} = a_j$.
    Let $k = a_j$. In this scenario, the strand $\ell$ leaves $\ell'_k$ at $w_j$, touches $\ell'_{k+1}$ at $w_{j+1}$, and returns to $\ell'_k$ at $w_{j+2}$. Consider the bigon bounded by two segments connecting $w_j$ and $w_{j+2}$. Then it must be a region and we can apply operation (II). See the left of Figure~\ref{fig: regions classify} for example.

    Subcase 2b): There exists an index $j$ such that $a_j = a_{j+1}$. Let $k = a_j$. Here, $\ell$ intersects $\ell'_k$ at two consecutive points $w_j$ and $w_{j+1}$ without intersecting any other black strand in between. Consider the bigon $R'$ bounded by two segments connecting $w_{j}$ and $w_{j+1}$.
    \begin{itemize}
        \item If the red boundary segment of $R'$ is on the left of $R'$, then $R'$ must be a region and we can apply operation (I). See the middle of Figure~\ref{fig: regions classify} for example.
        \item  If the red boundary segment of $R'$ is on the right of $R'$ then $R'$ itself may not be a region. In that case we take a simple region $R''$ inside $R'$. If $R'$ itself is a simple region, then simply let $R''=R'$. If there exists a red strand (other than those forming boundary of $R''$) that touches the boundary of $R''$ then it violates the fact that we initially chose $R$ to be the leftmost such simple region. Therefore we can apply operation (I) to $R''$. See the right of Figure~\ref{fig: regions classify} for example.
    \end{itemize}

    \begin{figure}
\begin{tikzpicture}[>=stealth, line width=1pt, scale=1]

    % --- STYLE DEFINITIONS ---
     \tikzset{
        black curve/.style={black, thick, line cap=round},
        red curve/.style={red, very thick, line cap=round},
        dashed line/.style={black, dashed, thin}
    }
    \tikzset{
        black line/.style={black, thick}, 
        black curve/.style={black, thick, line cap=round},
        red curve/.style={red, thick, line cap=round},
        dot/.style={fill=red, circle, inner sep=1.5pt}
    }

    % --- MACRO FOR BACKGROUND (STRAIGHT LINES) ---
    \newcommand{\drawbackground}[1]{
        \begin{scope}[xshift=#1cm]
            % Straight Vertical Lines l3' and l2'
            % l3' (Leftmost)
            \draw[black line] (-1.5, -0.5) -- (-1.5, 2.5);
            \node[below] at (-1.5, -0.5) {$\ell_3'$};
            
            % l2' (Middle)
            \draw[black line] (-0.6, -0.5) -- (-0.6, 2.5);
            \node[below] at (-0.6, -0.5) {$\ell_2'$};
            
            % l1' (Boundary of R) - Slightly curved
            \draw[black curve] (0.55, 0) to[bend left=20] (0.55, 2);
            \node[below] at (0.2, 1.2) {$\ell_1'$};
            
            % Region R (Closed loop implication)
            \draw[red curve] (0.55, 0) to[bend right=60] (0.55, 2);
            \node at (0.8, 1) {\Large $R$};
            \draw[dashed line] (0.55, 2) -- (0.55-0.4, 2+0.4);
            \draw[dashed line] (0.55, 2) -- (0.55+0.4, 2+0.4);
            \draw[dashed line] (0.55, 0) -- (0.55-0.4, 0-0.4);
            \draw[dashed line] (0.55, 0) -- (0.55+0.4, 0-0.4);
        \end{scope}
    }

    % --- 1. LEFT PANEL: Loop touches l3' tangentially ---
    \drawbackground{0}
    \begin{scope}[xshift=0cm]
        % Tangency points on R boundary
        \coordinate (v) at (0.42, 1.6);
        \coordinate (v_p) at (0.42, 0.4);
        
        % Incoming Red Line (Grazes v from top)
        \draw[red curve] (-0.3, 2.8) node[above, red] {$\ell$} 
            to[out=-90, in=90] (v);
            
        % The Loop: 
        % 1. Leaves v tangentially (downwards/leftwards)
        % 2. Touches l3' (x=-1.5) tangentially at the apex
        % 3. Returns to v' tangentially
        \draw[red curve] (v) 
            .. controls (0.3, 1.4) and (-1.5, 1.6) .. (-1.5, 1.0) % Out to touch l3'
            .. controls (-1.5, 0.4) and (0.3, 0.6) .. (v_p);      % Back to v'

        % Outgoing Red Line (Grazes v' going down)
        \draw[red curve] (v_p) 
            to[out=-90, in=90] (-0.3, -0.8);
            
        % Dots
        \node[dot] at (v) {}; \node[right, red, font=\footnotesize] at (v) {$v$};
        \node[dot] at (v_p) {}; \node[right, red, font=\footnotesize] at (v_p) {$v'$};
    \end{scope}

    % --- 2. MIDDLE PANEL: Loop transverses l2' ---
    \drawbackground{5}
    \begin{scope}[xshift=5cm]
        \coordinate (v) at (0.42, 1.6);
        \coordinate (v_p) at (0.42, 0.4);

        % Incoming
        \draw[red curve] (-0.3, 2.8) node[above, red] {$\ell$} to[out=-90, in=90] (v);

        % The Loop:
        % Crosses l2' (x=-0.6). Goes slightly deeper (e.g. -1.0) then turns back.
        \draw[red curve] (v) 
            .. controls (0.3, 1.4) and (-1.0, 1.5) .. (-1.0, 1.0) % Crosses l2'
            .. controls (-1.0, 0.5) and (0.3, 0.6) .. (v_p);      % Returns

        % Outgoing
        \draw[red curve] (v_p) to[out=-90, in=90] (-0.3, -0.8);

        % Dots
        \node[dot] at (v) {}; \node[right, red, font=\footnotesize] at (v) {$v$};
        \node[dot] at (v_p) {}; \node[right, red, font=\footnotesize] at (v_p) {$v'$};
    \end{scope}

    % --- 3. RIGHT PANEL: Loop touches l2' twice ---
    \drawbackground{10}
    \begin{scope}[xshift=10cm]
        \coordinate (v) at (0.42, 1.6);
        \coordinate (v_p) at (0.42, 0.4);

        % Incoming
        \draw[red curve] (-0.3, 2.8) node[above, red] {$\ell$} to[out=-90, in=90] (v);

        % The Loop: Double Tangency on l2' (x=-0.6)
        % Leaves v, touches l2' (top), retreats right, touches l2' (bottom), returns to v'
        \draw[red curve] (v) 
            to[out=-100, in=90] (-0.6, 1.3)   % First touch on l2'
            to[out=-90, in=90] (-0.4, 1.0)    % Retreat towards R
            to[out=-90, in=90] (-0.6, 0.7)    % Second touch on l2'
            to[out=-90, in=80] (v_p);         % Return to v'

        % Outgoing
        \draw[red curve] (v_p) to[out=-100, in=90] (-0.3, -0.8);

        % Dots
        \node[dot] at (v) {}; \node[right, red, font=\footnotesize] at (v) {$v$};
        \node[dot] at (v_p) {}; \node[right, red, font=\footnotesize] at (v_p) {$v'$};
    \end{scope}

\end{tikzpicture}\caption{Subcases for Case 2.}\label{fig: regions classify}
\end{figure}
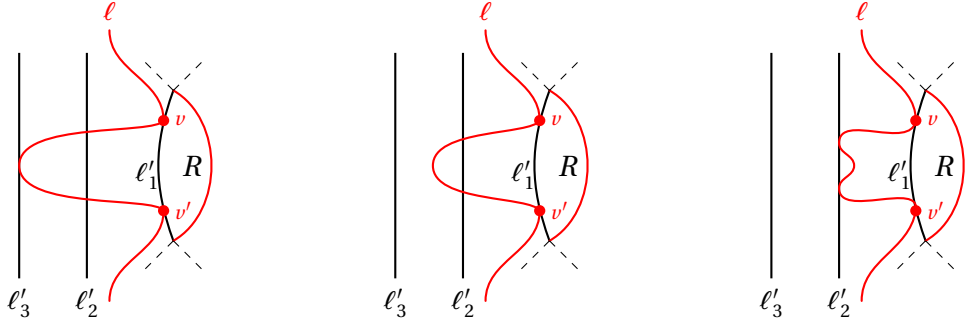
\end{proof}

\begin{lem}\label{z}
    For a wiggle diagram $\Pi$, if there exists a simple region such that at least one strand (other than those forming the boundary) intersects its boundary, then operation (I) or (II) can be applied somewhere in $\Pi$.
\end{lem}

\begin{proof}
    Let $R$ be the leftmost simple region with this property. Without loss of generality, assume the left boundary of $R$ is a black segment. Let $\ell$ be a red strand (other than the one forming the boundary) that intersects the boundary black segment of $R$. If such a red strand does not exist, then we are done with Lemma~\ref{y}.

    Let $\ell'_1$ be the black strand forming the left boundary of $R$, and let $\ell'_2$ be the black strand immediately to the left of $\ell'_1$. Let $v$ be a point where $\ell$ intersects boundary black segment of $R$.
    Starting from $v$, we traverse $\ell$ upwards (respectively downwards) to find the first point $w$ (respectively $w'$) where $\ell$ intersects either $\ell'_1$ or $\ell'_2$. Such points must exist because the points $v+(0,1)$ and $v-(0,1)$ lie on $\ell$ and $\ell'_1$ due to the periodicity of the diagram.

    Case 1): Both $w$ and $w'$ lie on $\ell'_2$.
    Consider the bigon bounded by two segments connecting $w$ and $w'$ and let $R'$ be a simple region inside it. Then strands other than those forming boundary of $R'$ cannot intersect the boundary of $R$, otherwise, it will violate that we chose $R$ to be the leftmost among such simple regions. We can apply operation (I) to $R'$.

    Case 2): At least one point (say $w$) lies on $\ell'_1$.
    Consider the bigon bounded by two segments connecting $v$ and $w$ and let $R''$ be a simple region inside it. Note that a black strand (other than that forming the boundary of $R''$) cannot intersect the boundary of $R''$. By Lemma~\ref{y}, we can find somewhere to apply operation (I) or (II).
\end{proof}

\begin{proof}[Proof of Proposition~\ref{prop: wiggle classification}]
    Pick any red strand $\ell$. If $\ell$ intersects some black strand $\ell'$ twice (in a principal band), then there must be a simple region where we can apply operation (I) otherwise, by Lemma~\ref{z} it is a contradiction. If $\ell$ intersects two distinct black strands then due to the periodicity, $\ell$ intersects some black strand twice (in a principal band). We conclude that each red strand may intersect only one black strand only once in a principal band. Moreover it is not possible that two red strands intersect the same black strand due to Lemma~\ref{z}.
\end{proof}

\begin{rmk}\label{remark: relax second}
    We may relax the third condition for a wiggle diagram by allowing two strands to intersect at finitely many intervals in the principal band. We simply regard those intervals as points, therefore crossings, and apply the map $\SW$ in the same way. For example, on the left of Figure~\ref{fig: relax wiggle}, $\Pi$ is a (2,1)-wiggle diagram with a relaxed third condition; the red strand meets the left black strand at two intervals $I_1$ and $I_2$. As we regard those intervals as points, $I_2$ is a positive crossing and $I_1$ is a negative crossing. Also in the graph $G(\Pi)$, $I_1$ and $I_2$ induces edges connecting $R$ and $R'$. Region $R''$ is the only region that is not regular in the principal band. Corresponding $\SW(\Pi)$ is shown on the right.
\end{rmk}    
\begin{figure}
\begin{tikzpicture}[scale=0.7, >=stealth]

    % --- Definitions for styles ---
    % dashed blue lines for y=0, y=1
    \tikzstyle{guideline} = [dash pattern=on 3pt off 3pt, blue!80, thick]
    % Main vertical walls
    \tikzstyle{wall} = [thick, black, line cap=round]
    % The highlighted 'flow' curves
    \tikzstyle{highlight} = [thick, red!90!black, line cap=round]
    % The secondary 'flow' curves
    \tikzstyle{regular} = [thick, black, line cap=round]
    
    % --- Diagram 1: Pi (Left) ---
    \begin{scope}[local bounding box=LeftDiagram]
        
        % 1. Horizontal Guidelines (y=0 and y=1)
        % Extending slightly beyond the walls
        \draw[guideline] (-1, 0) -- (3.5, 0) node[left] at (-1, 0) {\large $\color{blue} y=0$};
        \draw[guideline] (-1, 4) -- (3.5, 4) node[left] at (-1, 4) {\large $\color{blue} y=1$};

        % 2. Vertical Walls (Base Structure)
        \draw[wall] (0, 0) -- (0, 4); % Left wall
        \draw[wall] (2.5, 0) -- (2.5, 4); % Right wall

        % 3. The Red Path (The main feature of Pi)
        % Logic: Bottom curve -> Wall segment -> Bump -> Wall segment -> Cross curve -> Top return
        
        % Curve from bottom (approx x=2) to left wall
        \draw[highlight] (2.0, 0) to[out=150, in=-80] (-0.03, 1.2);
        
        % Segment I2 on the wall
        \draw[highlight] (-0.03, 1.2) -- (-0.03, 1.8);
        
        % The Bump (Semicircle to the left)
        \draw[highlight] (-0.03, 1.8) arc (270:90:0.5); 
        % Note: Center is roughly (-0.5, 2.3), radius 0.5. 
        % Starts at (0,1.8) (angle -90 relative to center) goes to (0,2.8)
        
        % Segment I1 on the wall
        \draw[highlight] (-0.03, 2.8) -- (-0.03, 3.2);
        
        % Curve crossing to the right wall
        \draw[highlight] (-0.03, 3.2) to[out=10, in=180] (2.5, 3.5);
        
        % Top curve returning from y=1 (approx x=2) to meet the cross
        \draw[highlight] (2.5, 3.5) to[out=90, in=-90] (2.0, 4);

        % 4. Labels
        \node at (-1.2, 2.3) { $R$};
        \node at (1.25, 2.3) { $R'$};
        \node at (-0.2, 2.3) { $R''$};
        \node[left, xshift=-2pt] at (0, 3.0) {\small $I_1$};
        \node[left, xshift=-2pt] at (0, 1.5) {\small $I_2$};
        \node at (1.25, -0.8) { $\Pi$};
        
    \end{scope}

    % --- Diagram 2: SW(Pi) (Right) ---
    % Shifted to the right by 6 units
    \begin{scope}[shift={(6,0)}]
        
        % 1. Horizontal Guidelines
        \draw[guideline] (-1, 0) -- (3.5, 0) node[left] at (-1, 0) {};
        \draw[guideline] (-1, 4) -- (3.5, 4) node[left] at (-1, 4) {};

        % 2. Vertical Walls (Base Structure)
        % Left wall (Top and Bottom parts are black)
        \draw[wall,red] (0, 0) -- (0, 1.2);
        \draw[wall,red] (0, 3.2) -- (0, 4);
        \draw[wall] (0.03,1.2) -- (0.03, 3.2);
        % Right wall is fully RED in SW(Pi)
        \draw[highlight] (2.5, 0) -- (2.5, 4);

        % 3. The Highlighted (Red) Parts on Left Wall
        % (Inverse of the flow lines in the first diagram)
        \draw[highlight] (0, 1.2) -- (0, 1.8); % I2 position
        \draw[highlight] (0, 1.8) arc (270:90:0.5); % The Bump
        \draw[highlight] (0, 2.8) -- (0, 3.2); % I1 position

        % 4. The Flow Curves (Now Black)
        % Curve from bottom
        \draw[regular] (2.0, 0) to[out=150, in=-80] (0.03, 1.2);
        
        % Curve crossing to the right wall
        \draw[regular] (0.03, 3.2) to[out=10, in=180] (2.5, 3.5);
        
        % Top curve
        \draw[regular] (2.5, 3.5) to[out=90, in=-90] (2.0, 4);

        % 5. Label
        \node at (1.25, -0.8) { $\SW(\Pi)$};

    \end{scope}

\end{tikzpicture}
\caption{Wiggle diagram with a relaxed condition.} \label{fig: relax wiggle} % Label must come AFTER caption
\end{figure}
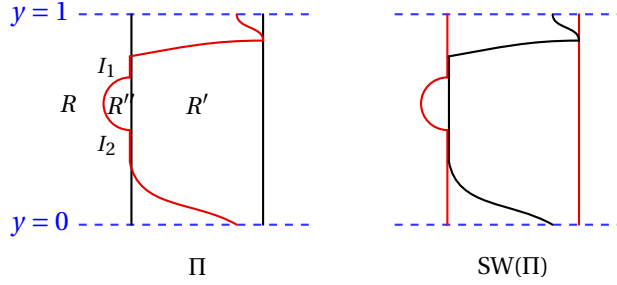

\end{subsection}

\begin{subsection}{From $\CH_{a}\times \CH_b$ to $(a,b)$-wiggle diagrams: proof of \eqref{eq: sw}}
\begin{definition}
Given $\pi\in \cPF_{m,n}$, we associate an infinite curve $\operatorname{IC}(\pi)$ as follows.
\begin{itemize}
    \item Step 1) Draw an infinite curve from $P_{\pi}$ by periodic extension.
    \item Step 2) Perturb every north step. Fix a small number $\epsilon>0$ and consider a north step starting from $(x,y)$ with label $\ell$. Shift this north step horizontally to the right by $\left(nx+m(\ell-y)\right)\epsilon$. East steps are modified accordingly. See Figure~\ref{fig: perturb}.
    \item Step 3) Finally, rotate and scale the whole picture so that the point $(m,n)$ is sent to $(0,1)$.
\end{itemize}
Given two vertically adjacent north steps $N$ and $N'$ of $\pi$, recall that the label of $N'$ is strictly larger than that of $N$. Therefore, after Step~2), the starting point of $N'$ lies weakly to the right of the end point of $N$. Hence, the resulting object is an infinite curve with period $(m,n)$ that moves strictly to the right or upward. After completing Step~3), $\operatorname{IC}(\pi)$ becomes a $y$-monotonic infinite curve with period $(0,1)$.

It is easy to check that $\operatorname{IC}(\pi')$ stays strictly to the left of $\operatorname{IC}(\pi)$ if $\pi \prec \pi'$. Finally, for $(\pi^{\bullet},\tau^{\bullet})\in \CH_a\times \CH_b$, we associate an $(a,b)$-wiggle diagram (with a relaxed third condition; see Remark~\ref{remark: relax second}) $\WD(\pi^{\bullet},\tau^{\bullet})$ by drawing $a$ black strands from $\operatorname{IC}(\pi^{(i)})$ for $1\le i\le a$ and $b$ red strands from $\operatorname{IC}(\tau^{(i)})$ for $1\le i\le b$.
\end{definition}

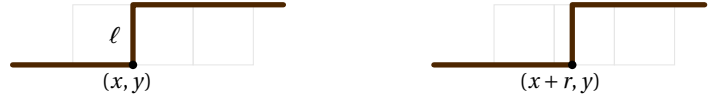
\begin{figure}[ht]
    \centering
    \begin{tikzpicture}[scale=0.8]
        % --- Definitions ---
        \def\m{3} \def\n{2} \def\k{2}
        \def\M{6} \def\N{4}
    
        \colorlet{path0}{black!70!orange}
        \colorlet{path1}{blue!70!cyan}
        \colorlet{celltext}{gray!70!black}
    
        \tikzset{
            grid line/.style={gray!20, thin},
            axis line/.style={->, >=stealth, thick},
            diag line/.style={dashed, thick, orange!80},
            path line/.style={line width=2pt, cap=round, join=round},
            content node/.style={font=\sffamily\tiny, text=celltext},
            label node/.style={font=\small\bfseries}
        }
    
        % SHEET 0
        \begin{scope}[local bounding box=sheet0]
            \foreach \x in {-1,...,1} {
                \foreach \y in {0} {
                    \draw[grid line] (\x,\y) rectangle (\x+1,\y+1);
                    %\pgfmathsetmacro{\cont}{int(0 + \M*\y - \N*(\x+1))}
                    %\node[content node] at (\x+0.5, \y+0.5) {$\cont$};
                }
            }
            %\draw[axis line] (-1.2,0) -- (4.2,0) node[right] {};
            %\draw[axis line] (0,-0.5) -- (0,2.5) node[above] {};
            %\draw[diag line] (-1, -0.66) -- (3.5, 2.33);
            \draw[path line, path0] (-2,0)--(0,0) -- (0,1) -- (2.5,1);
            \fill (0,0) circle (2pt);
            \node[label node] at (-0.3,0.5) {$\ell$};
             \node[label node] at (-0.1,-0.3) {$(x,y)$};
            %\node[label node] at (0.7,1.5) {$r_2$};
            %\node[label node] at (0.7-3,1.5-2) {$r_2$};
        \end{scope}
    
        % SHEET 1
        \begin{scope}[xshift=7cm, local bounding box=sheet1]
            \foreach \x in {-1,...,1} {
                \foreach \y in {0} {
                    \draw[grid line] (\x,\y) rectangle (\x+1,\y+1);
                    %\pgfmathsetmacro{\cont}{int(0 + \M*\y - \N*(\x+1))}
                    %\node[content node] at (\x+0.5, \y+0.5) {$\cont$};
                }
            }
            %\draw[axis line] (-1.2,0) -- (4.2,0) node[right] {};
            %\draw[axis line] (0,-0.5) -- (0,2.5) node[above] {};
            %\draw[diag line] (-1, -0.66) -- (3.5, 2.33);
            \draw[path line, path0] (-2,0)--(0.3,0) -- (0.3,1) -- (2.5,1);
            %\fill (2,1) circle (2pt);
            %\node[label node] at (-0.3,0.5) {$\ell$};
            \fill (0.3,0) circle (2pt);
             \node[label node] at (0.1,-0.3) {$(x+r,y)$};
            %\node[label node] at (0.7,1.5) {$r_2$};
            %\node[label node] at (0.7-3,1.5-2) {$r_2$};
        \end{scope}
    \end{tikzpicture}
    \caption{On the left, a local picture around a labeled north step is shown. The right figure illustrates the result after perturbing the north step to the right by a distance $r=\left(nx+m(\ell-y)\right)\epsilon$.}
    \label{fig: perturb}
\end{figure}

\begin{proof}[Proof of \eqref{eq: sw}]
For $\pi^{\bullet}\in \CH_{a}$ and $\tau^{\bullet}\in \CH_b$, positive crossings in $(\pi^{\bullet},\tau^{\bullet})$ defined in Section~\ref{subsec: Crossings and q-stat} correspond to positive crossings in $\WD(\pi^{\bullet},\tau^{\bullet})$ defined in Section~\ref{subsec: operators}. The map $\WD^{-1}\circ \SW \circ \WD: \CH_{a}\times \CH_{b} \rightarrow \CH_{b}\times \CH_{a}$ trivially preserves the skeleton and by Lemma~\ref{lem:stat+pos depend only on skeleton} and Proposition~\ref{prop: sw wiggle diagram} preserves $\stat$. Since this map is an involution (bijection), we complete the proof of \eqref{eq: sw}. 
\end{proof}

\end{subsection}

\begin{subsection}{Proof of \eqref{eq: qsw} }
We prove \eqref{eq: qsw} by introducing a map $\SW'$ which is a slight modification of the map $\SW$. By establishing its property, Proposition~\ref{prop: qsw wiggle}, we complete the proof. We begin by defining several terminologies.

We define $B(a,b)$ to be the set of $(a,b)$-wiggle diagrams $\Pi$ that satisfy the following conditions. Let $\ell$ denote the rightmost black strand. The conditions are:
\begin{itemize}
    \item  On the horizontal line $y=0$, $\ell$ has a bigger $x$-coordinate than any other red strand.
    \item  Let $\Pi'$ be the $(a-1,b)$-wiggle diagram obtained by deleting $\ell$ from $\Pi$. Then there exists a black strand in $\SW(\Pi')$ that intersects $\ell$.
\end{itemize}
We define $R(a,b)$ to be the set of $(b,a)$-wiggle diagrams satisfying the analogous conditions with the roles of black and red strands swapped.

For a diagram $\Pi \in B(a,b)$, we construct $\SW'(\Pi)$ as follows.
Let $\ell$ be the rightmost black strand and $\ell'$ be the rightmost red strand. Starting from the point where $\ell'$ intersects $y=0$, traverse $\ell'$ in the increasing $y$ direction. Let $v$ be the first point where $\ell'$ intersects $\ell$. The existence of such an intersection is guaranteed by the second condition in the definition of $B(a,b)$. Let $R$ denote the region immediately to the south of the crossing $v$ (it turns out that $R\notin \Reg(\Pi)$). Let $U$ be the set of regions belonging to the same connected component as $R$ in the graph $G(\Pi)$. We define the configuration $\SW'(\Pi)$ by changing the colors of segments that are surrounded by regions in $\Reg(\Pi) \cup U$.

\begin{proposition}\label{prop: qsw wiggle}
    The map $\SW'$ is a bijection from $B(a,b)$ to $R(a,b)$ such that the number of positive crossings (in the principal band) increases by 1.
\end{proposition}

\begin{proof}
    Let $\Pi \in B(a,b)$. Let $\ell_{1},\ell_{2},\dots$ (respectively $\ell'_{1},\ell'_{2},\dots$) be black strands (respectively red strands) enumerated from the right. Let $\Pi'$ denote the wiggle diagram obtained by deleting $\ell_1$ from $\Pi$. By the definition of $B(a,b)$, in $\Pi'$ there exists a region $R' \in \Reg(\Pi')$ whose boundary contains a segment of $\ell'_1$ and which intersects the strand $\ell_1$. By Lemma~\ref{lem: infinite region classification}, there exists a path from $R'$ to $(R')^{+}$
\begin{equation*}
    R' = Y_1 \rightarrow Y_2 \rightarrow \dots \rightarrow Y_{c} = (R')^{+}
\end{equation*}
in the graph $G(\Pi')$. Since each edge in $G(\Pi')$ corresponds to a crossing of $\Pi'$, we may identify the above path as a sequence of crossings $(w_1,w_2,\dots,w_{c-1})$ such that the edge between $Y_i$ and $Y_{i+1}$ arises from the crossing $w_i$.

    Returning to $\Pi$, let $w$ be the point where $\ell_1$ intersects the line $y=0$. Traversing $\ell'_1$ upwards from the line $y=0$, let $v$ be the first point where $\ell'_1$ intersects $\ell_1$.
    The strand $\ell_1$ passes through the region $R'$ (from $\Pi'$), dividing it. Let $R''$ be the region located to the left of $\ell_1$. Note that $R'' \in \Reg(\Pi)$ as the path induced by sequence of crossings $(w_1,w_2,\dots,w_{c-1})$ still connects $R''$ and $(R'')^{+}$. See the middle of Figure~\ref{fig:q_sw_picture} for the description of $\Pi$.

    We construct an augmented diagram $\bar{\Pi}$ by adding an auxiliary red strand $\ell''$. This strand is placed to the right of all strands in $\Pi$ and is constructed so that it intersects $\ell_1$ at exactly one point, $v'$, located on the segment of $\ell_1$ between $w$ and $v$.  See the right of Figure~\ref{fig:q_sw_picture} for the description of $\bar{\Pi}$.

    Let $R_1, R_2, R_3$, and $R_4$ denote the west, south, east, and north regions of the crossing $v'$ in $\bar{\Pi}$, respectively.
    In the graph $G(\bar{\Pi})$, the regions $R_1$ and $R_3$ are connected by an edge. Since $R_3$ is an unbounded region (being on the far right), the connected component containing $R_1$ and $R_3$ is contained in $\Reg(\bar{\Pi})$.
    Furthermore, $R_3$ is connected to $R''$. Since $R'' \in \Reg(\Pi)$, it follows that $R_3 \in \Reg(\bar{\Pi})$. We also have $R_2 \in \Reg(\bar{\Pi})$ as $R_3=R_2^{+}$.

    Consequently, when we form $\SW(\bar{\Pi})$, the auxiliary strand $\ell''$ ensures that the regions corresponding to the component $U$ (defined in the construction of $\SW'$; connected component of $G(\Pi)$ having $R_1$) are regular regions. In $\SW(\bar{\Pi})$, the strand $\ell''$ changes to a black strand, while the configuration of the remaining strands corresponds exactly to $\SW'(\Pi)$.

    We now track the change in positive crossings. In $\bar{\Pi}$, the intersection $v'$ is a positive crossing and becomes a negative crossing in $\SW(\bar{\Pi})$. By Proposition~\ref{prop: sw wiggle diagram}, we deduce that the number of positive crossings increases by 1 from $\Pi$ to $\SW'(\Pi)$.

     Now it remains to show $\SW'(\Pi)\in R(a,b)$. Let $\hat{\ell}$ be the rightmost red strand in $\SW'(\Pi)$. Then $\hat{\ell}$ does not pass through any crossings in the sequence $(w_1,w_2,\dots,w_{c-1})$, therefore the path induced by this sequence connects $\hat{R}$ and $(\hat{R})^{+}$ where $\hat{R}$ is the region to the east of $w_1$ in $\SW'(\Pi)-\hat{\ell}$. Since $\hat{R}$ intersects $\hat{\ell}$, we conclude $\SW'(\Pi)\in R(a,b)$.

\begin{figure}
\begin{tikzpicture}[>=stealth, line width=1pt, scale=1]

    % --- STYLE DEFINITIONS ---
    \tikzset{
        blue dashed/.style={blue!70, dashed, thin},
        red curve/.style={red, thick, line cap=round},
        black curve/.style={black, thick, line cap=round},
        dot/.style={fill=red, circle, inner sep=1.5pt}
    }

    % --- MACRO FOR THE TONGUE SHAPE (Region R) ---
    \newcommand{\drawTongue}[2]{
        \begin{scope}[xshift=#2]
            % The Tongue Shape
            \draw[red curve] (-0.5, 3.5) ..controls (2.5,3).. (-0.5, 2.5);
            % Label
            \node[red] at (0.5, 3.0) {\Large #1};
        \end{scope}
    }
    
    % ==========================================
    % DIAGRAM 1: Pi' (Left)
    % ==========================================
    \begin{scope}[xshift=0cm]
        \draw[blue dashed] (-1, 0) -- (2.5, 0); 
        \node[ blue] at (2.5,-0.3) {$y=0$};

        % Red curve ell' ending in a dot
        \draw[red curve] (-0.5, 0) node[below, black] {\textcolor{red}{$\ell_1'$}}
            to[out=75, in=200] (0.8, 1.8) coordinate (end_lp);
        \node[dot] at (end_lp) {};

        % Top Region R'
        \drawTongue{}{0cm}
        \node at (0, 3) {$R'$};

        % Vertical Ellipsis
        \node[red, scale=1.5] at (0.2, 2.3) {$\vdots$};

        % Label
        \node at (0.5, -1) { $\Pi'$};
    \end{scope}

    % ==========================================
    % DIAGRAM 2: Pi (Middle)
    % ==========================================
    \begin{scope}[xshift=5cm]
        \coordinate (w) at (1.2, 0);
        \coordinate (v) at (0.8, 1.8);

        \draw[blue dashed] (-1, 0) -- (3, 0);

        % Red curve ell' (touches ell at v)
        \draw[red curve] (-0.5, 0) node[below, black] {\textcolor{red}{$\ell_1'$}}
            to[out=75, in=200] (v);

        % Top Region R''
        \drawTongue{}{0cm}
         \node at (0, 3) {$R''$};

        % Black curve ell
        \draw[black curve] (w) node[below] {$\ell_1$}
            to[out=100, in=-80] (v)
            to[out=100, in=-90] (0.5, 2.7) 
            to[out=90, in=-100] (0.7, 3.5) 
            to[out=80, in=-90] (0.8, 3.8);

        % Intersection dot v
        \node[dot] at (v) {};
        \node[right] at (v) {$v$};
        \node at (0.9,0.15) {$w$};

        % Vertical Ellipsis
        \node[red, scale=1.5] at (0.2, 2.3) {$\vdots$};

        % Region R1 (Bounded by ell' and ell, below v)
        \node at (0.4, 0.8) {$R_1$};

        % Label
        \node at (1, -1) { $\Pi$};
    \end{scope}

    % ==========================================
    % DIAGRAM 3: overline{Pi} (Right)
    % ==========================================
    \begin{scope}[xshift=10.5cm]
        % Coordinates
        \coordinate (w) at (1.2, 0);
        \coordinate (v) at (0.8, 1.8);       % Upper intersection (same as Pi)
        \coordinate (v_prime) at (1.0, 1.0); % Lower intersection

        \draw[blue dashed] (-1, 0) -- (4.5, 0);

        % Top Region R''
        \drawTongue{}{0cm}
        \node at (0, 3) {$R''$};

        % Black curve ell (Passes through v' then v)
        \draw[black curve] (w) node[below] {$\ell_1$} 
            to[out=100, in=-90] (v_prime)
            to[out=90, in=-80] (v)
            to[out=100, in=-90] (0.5, 2.7) 
            to[out=90, in=-100] (0.7, 3.5) 
            to[out=80, in=-90] (0.8, 3.8);

        % Red curve ell' (Left) 
        % R1 stays the same: ell' must touch ell at v
        \draw[red curve] (-0.5, 0) node[below, black] {\textcolor{red}{$\ell_1'$}}
            to[out=75, in=200] (v);

        % Red curve ell'' (Right)
        % Comes from top right, touches at v', then goes South-East
        \draw[red curve] (3.0, 3.5) node[right, black] {\textcolor{red}{$\ell''$}}
            to[out=-120, in=70] (v_prime) % curve down
            to[out=-90, in=150] (1.5, 0.6) % turn south-east
            to[out=-30, in=170] (3.5, 0.2); % travel south-east

        % Intersection dots
        \node[dot] at (v) {};
        \node[left] at (v) {$v$};
        \node at (0.9,0.15) {$w$};

        \node[dot] at (v_prime) {};
        \node[right] at (v_prime) {$v'$};

        % Vertical Ellipsis
        \node[red, scale=1.5] at (0.2, 2.3) {$\vdots$};

        % Region labels
        \node at (0.4, 0.8) {$R_1$}; % Left of ell/ell'
        \node at (1.7, 0.2) {$R_2$}; % Right of ell, left of ell''
        \node at (2.5, 1) {$R_4$}; % Right of ell''
        \node at (1.3, 2.0) {$R_3$}; % Small region between v and v'?

        % Label
        \node at (1.5, -1) { $\bar{\Pi}$};
    \end{scope}

\end{tikzpicture}\caption{The construction of the auxiliary strand $\ell''$.}
        \label{fig:q_sw_picture}
\end{figure}
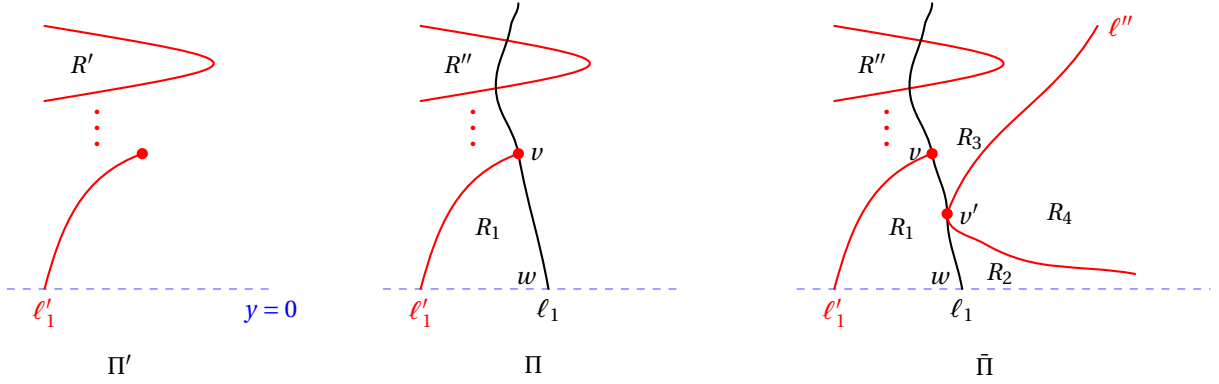
\end{proof}

\begin{proof}[Proof of \eqref{eq: qsw}]
Let $A_1$ be the set of pairs $(\pi^{\bullet},\tau^{\bullet}) \in \bar{\CH}_{a}\times \hat{\CH}_{b}$ such that $\WD(\pi^{\bullet},\tau^{\bullet}) \in B(a,b)$, and let $B_1$ be the set of pairs $(\pi^{\bullet},\tau^{\bullet}) \in \bar{\CH}_{b+1}\times \hat{\CH}_{a-1}$ such that $\WD(\pi^{\bullet},\tau^{\bullet}) \in B(b+1,a-1)$. Similarly, let $C_1$ be the set of pairs $(\pi^{\bullet},\tau^{\bullet}) \in \hat{\CH}_{b}\times \bar{\CH}_{a}$ such that $\WD(\pi^{\bullet},\tau^{\bullet}) \in R(a,b)$,
and let $D_1$ be the set of pairs $(\pi^{\bullet},\tau^{\bullet}) \in \hat{\CH}_{a-1}\times \bar{\CH}_{b+1}$ such that $\WD(\pi^{\bullet},\tau^{\bullet}) \in R(b+1,a-1)$.

Then there exists a bijection $f \colon A_1 \to C_1$ defined by $f = \WD^{-1} \circ \SW' \circ \WD$.
By Proposition~\ref{prop: qsw wiggle}, this bijection satisfies 
\[
z(v) = z(f(v)) \quad \text{ and } \quad \stat(v) = \stat(f(v)) -1 \quad \text{ for all } \quad v \in A_1.
\]
Similarly, there exists a bijection $g \colon B_1 \to D_1$ satisfying the same conditions: $z(v) = z(g(v))$ and $\stat(v) = \stat(g(v)) - 1$.

Finally, we define a bijection 
\[
f' \colon \bar{\CH}_{a}\times \hat{\CH}_{b}\setminus A_1 \to \bar{\CH}_{b+1}\times \hat{\CH}_{a-1}\setminus B_1
\]
as follows.
For $(\pi^{\bullet},\tau^{\bullet})$ in the domain, let $\nu^{\bullet} := (\pi^{(2)},\dots,\pi^{(a)}) \in \hat{\CH}_{a-1}$. Considering  $(\rho^{\bullet},\gamma^{\bullet}) := \WD^{-1}\circ\SW\circ\WD(\nu^{\bullet},\tau^{\bullet})$, we let
\[
    f'(\pi^{\bullet},\tau^{\bullet}) := \bigl((\pi^{(1)},\rho^{\bullet}),\gamma^{\bullet}\bigr) \in \bar{\CH}_{b+1}\times \hat{\CH}_{a-1}.
\]
It is immediate that $f'$ preserves both $z$ and $\stat$.
In the same manner, one can define a bijection $g' \colon \hat{\CH}_{b}\times \bar{\CH}_{a}\setminus C_1 \to \hat{\CH}_{a-1}\times \bar{\CH}_{b+1}\setminus D_1$ satisfying analogous properties.

Together, the bijections $f$, $g$, $f'$, and $g'$ complete the proof of \eqref{eq: qsw}.
\end{proof}

\end{subsection}

\section{Proof of Theorem~\ref{thm: main}}\label{Sec: wil}
In this section, we derive an expression for \( f[-MX^{m,n}]\cdot 1 \) for any symmetric function \( f \), in terms of the operators \( \h_i \), \( \hb_i \), and \( \hhat_i \) defined in the previous section (Proposition~\ref{prop: uniform description} and Corollary~\ref{cor: final}). In particular, we present a Jacobi--Trudi type formula for \( s_{\lambda}[-MX^{m,n}]\cdot 1 \), whose direct implication is the Loehr--Warrington conjecture for \( \nabla^{m}s_{\lambda} \). In addition, Corollary~\ref{cor: final} proves Theorem~\ref{thm: main} for free.

\subsection{Description of $f[-MX^{m,n}]\cdot 1$}
We consider square matrices whose entries are of the form $\h_i$, $\hb_i$, or $\hhat_i$. For such an $\ell$ by $\ell$ matrix $A$, entries of $A$ do not commute, so we need some extra care to define $\det(A)$ properly. We define the operator $\det(A)$ by
\begin{equation*}
    \det(A):=\sum_{w\in \mathfrak{S}_\ell}(-1)^{w}A_{w_1,1}A_{w_2,2}\cdots A_{w_\ell,\ell}.
\end{equation*}
 Throughout this section, we only consider matrices satisfying the following two conditions:
\begin{enumerate}
    \item Each column consists of a single type of entry (i.e., entirely $\h_i$'s, $\hb_i$'s, or $\hhat_i$'s).
    \item The indices of the entries in each column decrease strictly by one from top to bottom.
\end{enumerate}

For example, the following matrix satisfies conditions (1) and (2):
\begin{equation}\label{p}
    \begin{pmatrix}
        \h_2 & \h_3  & \hb_2 & \hhat_3\\
        \h_1 & \h_2  & \hb_1 & \hhat_2\\
        \h_0 & \h_1  & \hb_0 & \hhat_1\\
        \h_{-1} & \h_0  & \hb_{-1} & \hhat_0
    \end{pmatrix}
    =
    \begin{pmatrix}
        \h_2 & \h_3  & \hb_2 & \hhat_3\\
        \h_1 & \h_2  & \hb_1 & \hhat_2\\
        1 & \h_1  & 0 & \hhat_1\\
        0 & 1  & 0 & 1
    \end{pmatrix}.
\end{equation}
Since such a matrix is uniquely determined by its first row (by lowering indices for subsequent rows), we adopt a compact notation where the bracket $[\dots]$ lists the elements of the first row. For instance, the notation $[\h_2, \hhat_3, \hb_4]$ represents the square matrix
\begin{equation*}
    \begin{pmatrix}
        \h_2 & \hhat_3 & \hb_4\\
        \h_1 & \hhat_2 & \hb_3\\
        \h_0 & \hhat_1 & \hb_2
    \end{pmatrix}.
\end{equation*}
Similarly, the matrix in \eqref{p} is denoted by $[\h_2, \h_3, \hb_2, \hhat_3]$.

Lastly, we introduce the following notation. For $f\in \mathbb{C}(q,t)[\mathbf{z}]$, we denote $\Omega(f)$ to be the one obtained by specializing each $z_{i,j,k}$ to $t^jx_k$. For example, given $\pi^{\bullet}\in \cPF^{k}_{m,n}$, we have
\begin{equation*}
    \Omega(z(\pi^{\bullet}))=t^{\area(P_{\pi^{\bullet}})}x^{f_{\pi^{\bullet}}}
\end{equation*}

\begin{lem}\label{lem: C alpha operator expression}
    For a composition $\alpha=(\alpha_1,\dots,\alpha_{\ell}) \models k$, consider a sequence $\{v_i\}_{1 \leq i\leq k}$ defined by
    \begin{equation*}
        v_i=\begin{cases*}
            \hb_i \quad \text{if } i=\sum_{j=0}^{r} \alpha_{\ell-j} \text{ for some } r \ge 0, \\
            \hhat_i \quad \text{otherwise}.
        \end{cases*}
    \end{equation*}
    In other words, $v_i$ is $\hb_i$ if $i$ is a partial sum of the parts of $\alpha$ starting from the end, and $\hhat_i$ otherwise.
    Let $H(\alpha)=[v_1,v_2,\dots,v_k]$. Then we have
    \begin{equation}\label{eq: C alpha operator}
        C_{\alpha}[-MX^{m,n}] \cdot 1 = \Omega\left(\det([v_1,v_2,\dots,v_k]) \cdot 1\right).
    \end{equation}
\end{lem}

\begin{proof}
    By an inclusion-exclusion argument, we have
    \begin{equation*}
        \det([v_1,v_2,\dots,v_k]) \cdot 1=\sum_{\pi^{\bullet}\in\PTab_{m,n}(k;\alpha)}q^{\stat(\pi^{\bullet})}z(\pi^{\bullet}).
    \end{equation*}
    The proof now follows from Corollary~\ref{cor: compositional shuffle final}.
\end{proof}

\begin{example}
    For $\alpha=(2,3)$, according to Lemma~\ref{lem: C alpha operator expression}, we have
    \begin{equation*}
        C_{(2,3)}[-MX^{m,n}] \cdot 1 = \Omega\left(\det([\hhat_1,\hhat_2,\hb_3,\hhat_4,\hb_5]) \cdot 1.\right)
    \end{equation*}
    Note that the matrix $[\hhat_1,\hhat_2,\hb_3,\hhat_4,\hb_5]$ is
    \begin{equation*}
        \begin{pmatrix}
            \hhat_1 & \hhat_2 & \hb_3 & \hhat_4 &\hb_5\\
            1 & \hhat_1 & \hb_2 & \hhat_3 &\hb_4\\
            0 & 1 & \hb_1 & \hhat_2 &\hb_3\\
            0 & 0 & 0 & \hhat_1 &\hb_2\\
            0 & 0 & 0 & 1 &\hb_1
        \end{pmatrix}.
    \end{equation*}
    Due to the block triangular structure, its determinant factors into a product
    \begin{equation*}
        \det \begin{pmatrix}
            \hhat_1 & \hhat_2 & \hb_3 \\
            1 & \hhat_1 & \hb_2 \\
            0 & 1 & \hb_1
        \end{pmatrix}
        \det \begin{pmatrix}
             \hhat_1 &\hb_2\\
             1 &\hb_1
        \end{pmatrix}
        = \det([\hhat_1, \hhat_2 , \hb_3]) \det([\hhat_1 , \hb_2 ]).
    \end{equation*}
    In general, for $\alpha=(\alpha_1, \dots, \alpha_{\ell}) \models k$, we have
    \begin{align}\label{eq: C alph product}
        C_{\alpha}[-MX^{m,n}] \cdot 1 = \Omega\left(\det([\hhat_1, \dots,\hhat_{\alpha_{\ell}-1}, \hb_{\alpha_{\ell}} ]) \cdots \det([\hhat_1, \dots,\hhat_{\alpha_{1}-1}, \hb_{\alpha_{1}} ]) \cdot 1\right).
    \end{align}
\end{example}

\begin{definition}
    Given a symmetric function $f$, let $Y$ and $Z$ be formal alphabets of variables. We expand the symmetric function $f[Y-Z]$ in terms of the basis $h_{\lambda}[Y]h_{\mu}[Z]$, i.e.,
    \begin{equation*}
        f[Y-Z] = \sum_{\lambda,\mu} c_{\lambda,\mu} h_{\lambda}[Y] h_{\mu}[Z].
    \end{equation*}
    We then define an operator $\Phi(f)$ by replacing each $h_{\lambda}[Y] h_{\mu}[Z]$ with $\h_{\lambda}\hhat_{\mu}$:
    \begin{equation*}
        \Phi(f) = \sum_{\lambda,\mu} c_{\lambda,\mu} \h_{\lambda} \hhat_{\mu}.
    \end{equation*}
\end{definition}

For example, considering the Schur function $s_{(2,2)}$, the expansion is
\begin{align*}
    s_{(2,2)}[Y-Z] &= h_{(2,2)}[Y] - h_{(3,1)}[Y] - h_{(2,1)}[Y]h_{1}[Z] + h_{3}[Y]h_{1}[Z] + h_{2}[Y]h_{(1,1)}[Z] \\
    &\quad + h_{(1,1)}[Y]h_{2}[Z] - 2 h_{2}[Y]h_{2}[Z] - h_{1}[Y]h_{(2,1)}[Z] \\
    &\quad + h_{(2,2)}[Z] + h_{1}[Y]h_{3}[Z] - h_{(3,1)}[Z].
\end{align*}
Therefore, the operator $\Phi(s_{(2,2)})$ is given by
\begin{align*}
    \Phi(s_{(2,2)}) &= \h_{(2,2)} - \h_{(3,1)} - \h_{(2,1)} \hhat_{1} + \h_{3}\hhat_{1} + \h_{2}\hhat_{(1,1)} + \h_{(1,1)}\hhat_{2} 
    - 2 \h_{2}\hhat_{2} - \h_{1} \hhat_{(2,1)} + \hhat_{(2,2)} + \h_{1} \hhat_{3} - \hhat_{(3,1)}.
\end{align*}
\begin{proposition}\label{prop: uniform description}
    For a symmetric function $f$, we have
    \begin{equation*}
        f[-MX^{m,n}] \cdot 1 =\Omega\left( \Phi(f) \cdot 1\right).
    \end{equation*}
\end{proposition}

A direct approach to proving Proposition~\ref{prop: uniform description} would be to verify the claim for $f=C_{\alpha}$ by exploiting Lemma~\ref{lem: C alpha operator expression}. However, this approach faces two obstacles: first, the description of $\Phi(C_{\alpha})$ is complicated; and second, the right-hand side of \eqref{eq: C alpha operator} contains monomials where $\hhat_i$ operators may appear to the left of $\hb_i$ operators. In contrast, every monomial in the definition of $\Phi(f)$ has the $\hhat_i$ operators positioned to the right of the $\h_i$ operators.

Instead, we prove Proposition~\ref{prop: uniform description} by establishing the claim for $f=s_{\lambda}$. We begin with an elegant description of $\Phi(s_{\lambda})$, which also yields a fruitful consequence: an elementary proof of the Loehr--Warrington conjecture.

For a partition $\lambda=(\lambda_1,\lambda_2,\dots,\lambda_{\ell})$, we define a matrix of operators $J(\lambda)$ as follows:
\begin{equation*}
    J(\lambda)=[\h_{\lambda_{\ell}},\h_{\lambda_{\ell-1}+1},\dots,\h_{\lambda_{1}+\ell-1},\hhat_{\ell},\hhat_{\ell+1},\dots,\hhat_{\lambda_{1}+\ell-1}].
\end{equation*}
For example, for $\lambda=(2,2)$, we have $J((2,2))=[\h_2,\h_3,\hhat_2,\hhat_3]$.

\begin{lem}\label{lem: Schur Jacobi Trudi}
    For a partition $\lambda$, we have
    \begin{equation*}
        \Phi(s_{\lambda})=\det(J(\lambda)).
    \end{equation*}
\end{lem}

\begin{proof}
    Let $K(\lambda)$ be the matrix obtained by replacing each $\h_i$ with $h_i[Y]$ and each $\hhat_i$ with $h_i[Z]$ in $J(\lambda)$. It suffices to show that $s_{\lambda}[Y-Z]=\det(K(\lambda))$.

    Let $\lambda=(\lambda_1,\dots,\lambda_{\ell})$. Define a rectangle $R=(\underbrace{\lambda_1,\dots,\lambda_1}_{\ell})$ and a partition $\bar{\lambda}=(\lambda_1-\lambda_{\ell}, \lambda_1-\lambda_{\ell-1}, \dots, \lambda_{1}-\lambda_{1})$. Since $s_{\lambda}=s_{R/\overline{\lambda}}$ we have
    \begin{equation*}
        s_{\lambda}[Y-Z] = s_{R/\bar{\lambda}}[Y-Z] = \sum_{\mu\subseteq R} s_{\mu/\bar{\lambda}}[Y] s_{R/\mu}[-Z] = \sum_{\mu\subseteq R} (-1)^{|R/\mu|} s_{\mu/\bar{\lambda}}[Y] s_{(R/\mu)^{t}}[Z].
    \end{equation*}

    On the other hand, by a cofactor expansion along the first $\ell$ columns, we have
    \begin{equation*}
        \det(K(\lambda)) = \sum_{\substack{A \subseteq [\lambda_1+\ell] \\ |A|=\ell}} (-1)^{\sum_{a\in A} a - \frac{\ell(\ell+1)}{2}} \det(L_A) \det(L'_{A^{c}}),
    \end{equation*}
    where $L_{A}$ is the submatrix formed by the first $\ell$ columns and the rows indexed by $A$, and $L'_{A^c}$ is the submatrix formed by the last $\lambda_1$ columns and the rows indexed by $A^{c}=[\lambda_1+\ell] \setminus A$.
    
    There is a bijection between the set of partitions $\mu\subseteq R$ and the set of subsets $A\subseteq[\lambda_1+\ell]$ with $|A|=\ell$, given by:
    \begin{equation*}
        \mu=(\mu_1,\dots,\mu_{\ell}) \longmapsto A=\{\lambda_1-\mu_1+1, \lambda_1-\mu_2+2, \dots, \lambda_1-\mu_\ell+\ell\}.
    \end{equation*}
    Under this correspondence, the Jacobi-Trudi formula implies
    \begin{equation*}
        s_{\mu/\bar{\lambda}}[Y]=\det(L_A) \quad \text{and} \quad s_{(R/\mu)^{t}}[Z]=s_{(\bar{\mu})^{t}}[Z]=\det(L'_{A^{c}})
    \end{equation*}
    where $\bar{\mu}=(\lambda_1-\mu_{\ell},\lambda_1-\mu_{\ell-1},\dots,\lambda_1-\mu_{1})$.
    Furthermore, the sign matches as $(-1)^{|R/\mu|}=(-1)^{\sum_{a\in A}a - \frac{\ell(\ell+1)}{2}}$.
\end{proof}

\begin{example}
    Let $\lambda=(2,2)$. Then the matrix $J(\lambda)$ is given by
    \begin{equation*}
        J((2,2)) = 
        \begin{pmatrix} 
            \h_2 & \h_3 & \hhat_2 & \hhat_3 \\
            \h_1 & \h_2 & \hhat_1 & \hhat_2 \\
            1    & \h_1 & 1       & \hhat_1 \\
            0    & 1    & 0       & 1
        \end{pmatrix}.
    \end{equation*}
    By performing a cofactor expansion along the first two columns, $\det(J(\lambda))$ is given by:
    \begin{align}
       \det(J(\lambda)) &= 
        \det\begin{pmatrix} \h_2 & \h_3 \\ \h_1 & \h_2 \end{pmatrix} \det\begin{pmatrix} 1 & \hhat_1 \\ 0 & 1 \end{pmatrix} 
        - \det\begin{pmatrix} \h_2 & \h_3 \\ 1 & \h_1 \end{pmatrix} \det\begin{pmatrix} \hhat_1 & \hhat_2 \\ 0 & 1 \end{pmatrix} \nonumber\\
        &\quad + \det\begin{pmatrix} \h_2 & \h_3 \\ 0 & 1 \end{pmatrix} \det\begin{pmatrix} \hhat_1 & \hhat_2 \\ 1 & \hhat_1 \end{pmatrix}
        + \det\begin{pmatrix} \h_1 & \h_2 \\ 1 & \h_1 \end{pmatrix} \det\begin{pmatrix} \hhat_2 & \hhat_3 \\ 0 & 1 \end{pmatrix} \nonumber\\
        &\quad - \det\begin{pmatrix} \h_1 & \h_2 \\ 0 & 1 \end{pmatrix} \det\begin{pmatrix} \hhat_2 & \hhat_3 \\ 1 & \hhat_1 \end{pmatrix} 
        + \det\begin{pmatrix} 1 & \h_1 \\ 0 & 1 \end{pmatrix} \det\begin{pmatrix} \hhat_2 & \hhat_3 \\ \hhat_1 & \hhat_2 \end{pmatrix}.\label{llll1}
    \end{align}

    We also have
    \begin{align*}
        s_{(2,2)}[Y-Z] &= \sum_{\mu\subseteq (2,2)} s_{\mu}[Y] s_{(2,2)/\mu}[-Z] 
        = \sum_{\mu\subseteq (2,2)} (-1)^{|(2,2)/\mu|} s_{\mu}[Y] s_{\left((2,2)/\mu\right)^{t}}[Z] \\
        &= s_{(2,2)}[Y]s_{\emptyset}[Z] - s_{(2,1)}[Y]s_{(1)}[Z] + s_{(2)}[Y]s_{(1,1)}[Z] s_{(1,1)}[Y]s_{(2)}[Z] - s_{(1)}[Y]s_{(2,1)}[Z] + s_{\emptyset}[Y]s_{(2,2)}[Z].
    \end{align*}
    Applying the Jacobi-Trudi formula to each term, we conclude:
    \begin{align}
        s_{(2,2)}[Y-Z] &= 
        \det\begin{pmatrix} h_2[Y] & h_3[Y] \\ h_1[Y] & h_2[Y] \end{pmatrix} \det\begin{pmatrix} 1 & h_1[Z] \\ 0 & 1 \end{pmatrix} 
        - \det\begin{pmatrix} h_2[Y] & h_3[Y] \\ 1 & h_1[Y] \end{pmatrix} \det\begin{pmatrix} h_1[Z] & h_2[Z] \\ 0 & 1 \end{pmatrix} \nonumber\\
        &\quad + \det\begin{pmatrix} h_2[Y] & h_3[Y] \\ 0 & 1 \end{pmatrix} \det\begin{pmatrix} h_1[Z] & h_2[Z] \\ 1 & h_1[Z] \end{pmatrix}
        + \det\begin{pmatrix} h_1[Y] & h_2[Y] \\ 1 & h_1[Y] \end{pmatrix} \det\begin{pmatrix} h_2[Z] & h_3[Z] \\ 0 & 1 \end{pmatrix} \nonumber\\
        &\quad - \det\begin{pmatrix} h_1[Y] & h_2[Y] \\ 0 & 1 \end{pmatrix} \det\begin{pmatrix} h_2[Z] & h_3[Z] \\ 1 & h_1[Z] \end{pmatrix}
        + \det\begin{pmatrix} 1 & h_1[Y] \\ 0 & 1 \end{pmatrix} \det\begin{pmatrix} h_2[Z] & h_3[Z] \\ h_1[Z] & h_2[Z] \end{pmatrix}.\label{llll2}
    \end{align}
    Comparing \eqref{llll1} and \eqref{llll2}, we see that Lemma~\ref{lem: Schur Jacobi Trudi} holds for $\lambda=(2,2).$
\end{example}

To prove Proposition~\ref{prop: uniform description}, we establish several technical lemmas.

\begin{lem}\label{lem: tech1}
    Let $v$ and $w$ be any sequences of operators chosen from the set $\{\h_i, \hb_i, \hhat_i\}_{i\in \mathbb{Z}}$. For positive integers $a$ and $\ell$, we have
    \begin{equation*}
        -q\det([v,\hhat_a,\hb_{a+\ell}, w]) = \det([v,\hb_{a+\ell},\hhat_a, w]) + (1-q)\sum_{i=1}^{\ell-1}\det([v,\hb_{a+\ell-i},\hhat_{a+i}, w]).
    \end{equation*}
\end{lem}

\begin{proof}
    By linearity, it suffices to show that for any row indices $j < k$, the following $2 \times 2$ determinant identity holds:
    \begin{equation}\label{eq: induction target}
        -q\det\begin{pmatrix}
            \hhat_{a-j} & \hb_{a+\ell-j}\\
            \hhat_{a-k} & \hb_{a+\ell-k}
        \end{pmatrix}
        = \det\begin{pmatrix}
            \hb_{a+\ell-j} & \hhat_{a-j} \\
            \hb_{a+\ell-k} & \hhat_{a-k}
        \end{pmatrix} 
        + (1-q)\sum_{i=1}^{\ell-1}\det\begin{pmatrix}
            \hb_{a+\ell-i-j} & \hhat_{a+i-j}\\
            \hb_{a+\ell-i-k} & \hhat_{a+i-k}
        \end{pmatrix}.
    \end{equation}
    We proceed by induction on $\ell$. The base case $\ell=1$ corresponds to \eqref{eq: qsw}.

    Let $\ell > 1$ and assume that \eqref{eq: induction target} holds for $\ell-1$. Specifically, we apply the inductive hypothesis by replacing the parameters $a \to a+1$, $\ell \to \ell-1$, and $k \to k+1$. This yields
    \begin{equation}\label{eq: induction step}
        -q\det\begin{pmatrix}
            \hhat_{a+1-j} & \hb_{a+\ell-j}\\
            \hhat_{a-k} & \hb_{a+\ell-k-1}
        \end{pmatrix}
        = \det\begin{pmatrix}
            \hb_{a+\ell-j} & \hhat_{a+1-j} \\
            \hb_{a+\ell-k-1} & \hhat_{a-k}
        \end{pmatrix} \\
        + (1-q)\sum_{i=1}^{\ell-2}\det\begin{pmatrix}
            \hb_{a+\ell-i-j} & \hhat_{a+1+i-j}\\
            \hb_{a+\ell-i-k-1} & \hhat_{a+i-k}
        \end{pmatrix}.
    \end{equation}
    Subtracting \eqref{eq: induction step} from \eqref{eq: induction target} gives
    \begin{equation}\label{klo}
        -q\hhat_A\hb_B+q\hhat_{A+1}\hb_{B-1}=\hb_{B-1}\hhat_{A+1}-\hb_{B}\hhat_{A}+(1-q)(\hb_{A+1}\hhat_{B-1}-\hb_{B-1}\hhat_{A+1})
    \end{equation}
    where we denote $A = a-j$ and $B = a+\ell-k$. From \eqref{eq: qsw} we have
    \begin{equation}\label{eq: known comm}
        -q(\hhat_{A}\hb_{B} - \hhat_{B-1}\hb_{A+1}) = \hb_{A+1}\hhat_{B-1} - \hb_{B}\hhat_{A},
    \end{equation}
    and we subtract \eqref{eq: known comm} from \eqref{klo} giving
    \begin{equation*}
        q\hhat_{A+1}\hb_{B-1} - q\hhat_{B-1}\hb_{A+1} = q\hb_{B-1}\hhat_{A+1} - q\hb_{A+1}\hhat_{B-1}.
    \end{equation*}
      This identity holds since
    \begin{align*}
        &\hhat_{A+1}\hb_{B-1} - \hhat_{B-1}\hb_{A+1} 
        = (\h_{A+1} - \hb_{A+1})\hb_{B-1} - (\h_{B-1} - \hb_{B-1})\hb_{A+1} =
         \h_{A+1}\hb_{B-1} - \h_{B-1}\hb_{A+1} \\
        &=  \h_{A+1}(\h_{B-1}-\hhat_{B-1}) - \h_{B-1}(\h_{A+1}-\hhat_{A+1})= \hb_{B-1}\hhat_{A+1} - \hb_{A+1}\hhat_{B-1},
    \end{align*}
     completing the inductive step.
\end{proof}

\begin{lem}\label{lem: tech2}
    Let $v$ be a sequence of operators chosen from $\{\h_i, \hb_i, \hhat_i\}_{i\in \mathbb{Z}}$. For positive integers $a$, $\ell$, and $k$, we have
    \begin{multline*}
        (-q)^{\ell}\det([v,\hhat_a,\dots,\hhat_{a+\ell-1},\hb_{a+\ell+k},\hhat_{a+\ell+1},\dots,\hhat_{a+\ell+k}]) \\
        = \det([v,\hb_{a+\ell+k},w^{(0)}]) +\sum_{i=1}^{\ell}(-q)^{i-1}(1-q)\det([v,\hb_{a+\ell+k-i},w^{(i)}]),
    \end{multline*}
    where $w^{(i)}$ is the sequence obtained by removing $\hhat_{a+\ell-i}$ from the sequence $(\hhat_a,\dots,\hhat_{a+\ell+k})$.
\end{lem}

\begin{proof}
    We proceed by induction on $\ell$. For the base case $\ell=1$, let $v' = (\hhat_{a+2},\dots,\hhat_{a+k+1})$ then the LHS is $-q\det([v,\hhat_a,\hb_{a+k+1},v'])$.
    By Lemma~\ref{lem: tech1}, we have
    \begin{equation*}
        -q\det([v,\hhat_a,\hb_{a+k+1},v']) = \det([v,\hb_{a+k+1},\hhat_a,v']) + (1-q)\sum_{j=1}^{k}\det([v,\hb_{a+k+1-j},\hhat_{a+j},v']).
    \end{equation*}
    Observe that for $2\leq j\leq k$, the sequence $v'$ already contains the operator $\hhat_{a+j}$. Consequently, the determinant $\det([v,\hb_{a+k+1-j},\hhat_{a+j},v'])$ vanishes for these values of $j$. The only non-zero term in the sum corresponds to $j=1$. Thus, we conclude
    \begin{equation*}
        -q\det([v,\hhat_a,\hb_{a+k+1},v']) = \det([v,\hb_{a+k+1},\hhat_a,v']) + (1-q)\det([v,\hb_{a+k},\hhat_{a+1},v']).
    \end{equation*}
    Noting that $(\hhat_a, v') = w^{(0)}$ and $(\hhat_{a+1}, v') = w^{(1)}$, this matches the claim for $\ell=1$.

    Let $\ell > 1$ and assume the claim holds for $\ell-1$. Applying the inductive hypothesis to the sequence starting with $(v, \hhat_a)$, we have
    \begin{align}
        &(-q)^{\ell}\det([v,\hhat_a,\dots,\hhat_{a+\ell-1},\hb_{a+\ell+k},\hhat_{a+\ell+1},\dots,\hhat_{a+\ell+k}]) \nonumber\\
        &\quad = -q\det([v,\hhat_a,\hb_{a+\ell+k},u^{(0)}]) + \sum_{i=1}^{\ell-1}(-q)^{i}(1-q)\det([v,\hhat_{a},\hb_{a+\ell+k-i},u^{(i)}]),\label{ui}
    \end{align}
    where $u^{(i)}$ is the sequence obtained by removing $\hhat_{a+\ell-i}$ from $(\hhat_{a+1},\dots,\hhat_{a+\ell+k})$.

    Next, we apply Lemma~\ref{lem: tech1} to each term in the expression above.
    By Lemma~\ref{lem: tech1}, we have
    \begin{align*}
        -q \det([v,\hhat_{a},\hb_{a+\ell+k-i},u^{(i)}]) 
        = \det([v,\hb_{a+\ell+k-i},\hhat_{a},u^{(i)}]) 
        + (1-q)\sum_{j=1}^{\ell+k-i-1}\det([v,\hb_{a+\ell+k-i-j},\hhat_{a+j},u^{(i)}]).
    \end{align*}
    Similar to the base case, the term $\det([v,\hb_{a+\ell+k-i-j},\hhat_{a+j},u^{(i)}])$ vanishes unless $\hhat_{a+j}$ is the specific element missing from $u^{(i)}$, which is $\hhat_{a+\ell-i}$. This occurs when $j = \ell-i$. We conclude
    \begin{align*}
          -q \det([v,\hhat_{a},\hb_{a+\ell+k-i},u^{(i)}])=\det([v,\hb_{a+\ell+k-i},\hhat_{a},u^{(i)}]) + (1-q)\det([v,\hb_{a+k},\hhat_{a+\ell-i},u^{(i)}]).
    \end{align*}
    Identifying the sequences, we have $(\hhat_a, u^{(i)}) = w^{(i)}$ and $(\hhat_{a+\ell-i}, u^{(i)}) \cong w^{(\ell)}$ (up to a sign change of $(-1)^{\ell-1-i}$ required to reorder the columns).
    Thus
    \begin{equation*}
        -q \det([v,\hhat_{a},\hb_{a+\ell+k-i},u^{(i)}]) = \det([v,\hb_{a+\ell+k-i},w^{(i)}]) + (-1)^{\ell-1-i}(1-q)\det([v,\hb_{a+k},w^{(\ell)}]).
    \end{equation*}
    Substituting these expansions back into \eqref{ui} establishes the claim for $\ell$.
\end{proof}

\begin{lem}\label{lem: tech3}
    For a partition $\lambda=(\lambda_1,\lambda_2,\dots,\lambda_{\ell})$ and a positive integer $k\geq \lambda_1$, we have
    \begin{equation*}
        \det(J(\lambda))\det([\hhat_1,\dots,\hhat_{k-1},\hb_{k}]) = (-1)^{k-1}\sum_{\mu} \frac{\det(J(\mu))}{q^{\mu_1-1}},
    \end{equation*}
    where the sum runs over all partitions $\mu$ that can be obtained by adding a horizontal strip of size $k$ to $\lambda$.
\end{lem}

\begin{proof}
    First, observe that 
    \[
    \det([\hhat_1,\dots,\hhat_{k-1},\hb_{k}]) = \det([\hhat_1,\dots,\hhat_{k-1},\hb_{k},\hhat_{k}]),
    \]
    since the last row of the matrix $[\hhat_1,\dots,\hhat_{k-1},\hb_{k},\hhat_{k}]$ is $(0,\dots,0,1)$.
    By recursively applying Lemma~\ref{lem: tech1} to move $\hb_k$ to the left, we obtain
    \begin{equation*}
        (-q)^{k-1}\det([\hhat_1,\dots,\hhat_{k-1},\hb_{k},\hhat_{k}]) = \det([\hb_k,\hhat_1,\dots,\hhat_{k}]).
    \end{equation*}
   Therefore we have
    \begin{align*}
        &(-q)^{\lambda_1+k-1}\det(J(\lambda))\det([\hhat_1,\dots,\hhat_{k-1},\hb_{k}]) \\
        &\quad = (-q)^{\lambda_1}\det(J(\lambda))\det([\hb_k,\hhat_1,\dots,\hhat_{k}]) \\
        &\quad = (-q)^{\lambda_1}\det([\h_{\lambda_{\ell}},\dots,\h_{\lambda_{1}+\ell-1},\hhat_{\ell},\hhat_{\ell+1},\dots,\hhat_{\lambda_{1}+\ell-1},\hb_{\lambda_1+\ell+k},\hhat_{\lambda_{1}+\ell+1},\dots, \hhat_{\lambda_{1}+\ell+k}]).
    \end{align*}
    Applying Lemma~\ref{lem: tech2} to this expression yields
    \begin{equation}\label{eq: expansion sum}
        \det([\h_{\lambda_{\ell}},\dots,\h_{\lambda_{1}+\ell-1},\hb_{\lambda_1+\ell+k},w^{(0)}])
        + \sum_{j=1}^{\lambda_1}(-q)^{j-1}(1-q)\det([\h_{\lambda_{\ell}},\dots,\h_{\lambda_{1}+\ell-1},\hb_{\lambda_1+\ell+k-j},w^{(j)}]),
    \end{equation}
    where $w^{(j)}$ is the sequence obtained by removing $\hhat_{\lambda_1+\ell-j}$ from the sequence $(\hhat_{\ell},\dots,\hhat_{\lambda_{1}+\ell+k})$.
    
    For each $0 \leq j \leq \lambda_1$, the sequence $w^{(j)}$ contains the operator $\hhat_{\lambda_1+\ell+k-j}$. Since $\hb_i = \h_i - \hhat_i$, we can perform a column operation adding the column $\hhat_{\lambda_1+\ell+k-j}$ to the column $\hb_{\lambda_1+\ell+k-j}$. This simplifies the determinant as
    \begin{equation}\label{eq: hb to h}
        \det([\h_{\lambda_{\ell}},\dots,\h_{\lambda_{1}+\ell-1},\hb_{\lambda_1+\ell+k-j},w^{(j)}]) = \det([\h_{\lambda_{\ell}},\dots,\h_{\lambda_{1}+\ell-1},\h_{\lambda_1+\ell+k-j},w^{(j)}]).
    \end{equation}
    Analogous to the proof of Lemma~\ref{lem: Schur Jacobi Trudi}, the right-hand side of \eqref{eq: hb to h} equals $(-1)^{\lambda_1-j}\Phi(s_{R(j)})$, where $R(j)$ is a skew shape $R/\nu$ for
    \begin{align*}
        R=(\underbrace{\lambda_1+k-j,\dots,\lambda_1+k-j}_{\ell},k), \quad \text{and} \quad
        \nu=(\lambda_{1}-\lambda_{\ell}+k-j,\lambda_1-\lambda_{\ell-1}+k-j,\dots,\lambda_{1}-\lambda_{1}+k-j).
    \end{align*}
   Rotating the skew shape $R(j)$ by $180^\circ$ yields the skew shape $(\lambda_1+k-j,\lambda_1,\dots,\lambda_{\ell})/(\lambda_1-j)$. Therefore we have 
    \begin{equation*}
        s_{R(j)}=s_{(\lambda_1+k-j,\lambda_1,\dots,\lambda_{\ell})/(\lambda_1-j)}=h_{\lambda_1-j}^{\perp}s_{(\lambda_1+k-j,\lambda_1,\dots,\lambda_{\ell})}.
    \end{equation*}
    We conclude that $ s_{R(j)}= \sum_{\mu} s_{\mu}$, where the sum runs over all partitions $\mu$ obtained by adding a horizontal strip of size $k$ to $\lambda$ such that $\mu_1 \leq \lambda_1+k-j$.
    
    Substituting this back into \eqref{eq: expansion sum}, the expression becomes
    \begin{align*}
    (-1)^{\lambda_1} \Phi \left( s_{R(0)} + \sum_{j=1}^{\lambda_1}(q-1)q^{j-1}s_{R(j)} \right).
    \end{align*}

    We now compute the coefficient of $s_{\mu}$ in $( s_{R(0)} + \sum_{j=1}^{\lambda_1}(q-1)q^{j-1}s_{R(j)})$. For a fixed $\mu$ obtained by adding a horizontal strip of size $k$ to $\lambda$, the term $s_{\mu}$ appears in $s_{R(j)}$ if and only if $0 \leq j \leq \lambda_1+k-\mu_1$.
    Therefore, the total coefficient of $s_{\mu}$ becomes
    \[
    1 + \sum_{j=1}^{\lambda_1+k-\mu_1}(q-1)q^{j-1} = q^{\lambda_1+k-\mu_1}.
    \]
    Thus, we have derived
    \[
    (-q)^{\lambda_1+k-1} \cdot \text{LHS} = (-q)^{\lambda_1+k-1}\det(J(\lambda))\det([\hhat_1,\dots,\hhat_{k-1},\hb_{k}])= (-1)^{\lambda_1} \sum_{\mu} q^{\lambda_1+k-\mu_1} \Phi(s_{\mu})
    \]
     where the sum runs over all partitions $\mu$ that can be obtained by adding a horizontal strip of size $k$ to $\lambda$. Dividing both sides by $(-q)^{\lambda_1+k-1}$ completes the proof together with Lemma~\ref{lem: Schur Jacobi Trudi}.
\end{proof}

\begin{proof}[Proof of Proposition~\ref{prop: uniform description}]
    For a composition $\alpha$, let $H(\alpha)$ be the matrix defined in Lemma~\ref{lem: C alpha operator expression}. We prove the following stronger claim: for any partition $\lambda$, there exists a linear combination $s_{\lambda}=\sum_{\alpha}d_{\lambda,\alpha}C_{\alpha}$ such that
    \begin{equation*}
        \det(J(\lambda))=\sum_{\alpha}d_{\lambda,\alpha}\det(H(\alpha)).
    \end{equation*}
    
    We proceed by induction on the length of $\lambda$. Let $\mu=(\lambda_2,\dots,\lambda_{\ell})$ (note that $\mu$ may be empty). Then we have \cite[Proposition 3.6]{HMZ12}
    \begin{equation*}
        s_{\lambda}=(-q)^{\lambda_1-1}\sum_{i\geq 0} \mathbf{C}_{\lambda_1+i}\left(e^{\perp}_i s_{\mu}\right).
    \end{equation*}
    Let $V_i$ denote the set of partitions obtained by removing a vertical strip of size $i$ from $\mu$. For each $\nu \in V_i$, by the induction hypothesis, there exist constants $d_{\nu,\beta}$ such that
    \begin{equation*}
        s_{\nu}=\sum_{\beta}d_{\nu,\beta}C_{\beta} \quad \text{and} \quad \det(J(\nu))=\sum_{\beta}d_{\nu,\beta}\det(H(\beta)).
    \end{equation*}
    Applying $\mathbf{C}_{\lambda_1+i}$ corresponds to prepending the part $\lambda_1+i$ to the composition. Therefore we have
    \begin{equation*}
        \mathbf{C}_{\lambda_1+i}s_{\nu}=\sum_{\beta}d_{\nu,\beta}C_{(\lambda_1+i,\beta)}.
    \end{equation*}
    We now compute the corresponding sum of determinants. Using the factorization property of $\det(H(\alpha))$ and Lemma~\ref{lem: tech3}, we have
    \begin{align*}
        \sum_{\beta}d_{\nu,\beta}\det(H((\lambda_1+i,\beta)))
        &= \left(\sum_{\beta}d_{\nu,\beta}\det(H(\beta))\right) \det([\hhat_{1},\dots,\hhat_{\lambda_1+i-1},\hb_{\lambda_1+i}]) \\
        &= \det(J(\nu))\det([\hhat_{1},\dots,\hhat_{\lambda_1+i-1},\hb_{\lambda_1+i}]) \\
        &= (-1)^{\lambda_1+i-1}\sum_{\tau\in W(\nu,\lambda_1+i)}\frac{\det(J(\tau))}{q^{\tau_1-1}},
    \end{align*}
    where $W(\nu, k)$ denotes the set of partitions obtained by adding a horizontal strip of size $k$ to $\nu$.
    
    Recall that the Schur function satisfies the relation
    \begin{equation*}
        \left(\sum_{i\geq 0}(-1)^{i}h_{\lambda_1+i}e_{i}^{\perp}\right)s_{\mu}=s_{\lambda}.
    \end{equation*}
    Therefore we conclude
    \begin{equation*}
        \sum_{i\geq 0}\sum_{\nu\in V_i}(-1)^{\lambda_1+i-1}\sum_{\tau\in W(\nu,\lambda_1+i)}\frac{\det(J(\tau))}{q^{\tau_1-1}} = (-q)^{-\lambda_1+1}\det(J(\lambda)).
    \end{equation*}
    Consequently, by expressing $s_{\lambda}$ as
    \begin{equation*}
        s_{\lambda} = (-q)^{\lambda_1-1}\sum_{i \ge 0}\sum_{\nu\in V_i}\sum_{\beta} d_{\nu,\beta}C_{(\lambda_1+i,\beta)},
    \end{equation*}
    we have
    \begin{equation*}
       (-q)^{\lambda_1-1} \sum_{i \ge 0}\sum_{\nu\in V_i}\sum_{\beta}d_{\nu,\beta}\det(H(\lambda_1+i,\beta)) = \det(J(\lambda))
    \end{equation*}
    which proves the claim.
\end{proof}

\subsection{Loehr--Warrington formula}\label{subsec: alternative proof of LW}
In the previous subsection, we showed $s_{\lambda}[-MX^{m,n}]\cdot 1 =\Omega\left(\det(J(\lambda))\cdot 1\right)$. In this subsection, we introduce a modified matrix $J'(\lambda)$ such that:
\begin{enumerate}
    \item $\det(J(\lambda))=(-q)^{\adj(\lambda)}\det(J'(\lambda))$
    \item  $\det(J'(\lambda))\cdot 1$ exhibits a direct connection to the Loehr--Warrington formula when $n=1$.
\end{enumerate}
For technical details of the proof of (1) and (2), we refer to \cite{KO24}. Combining (1) and (2), we yield a new proof of the Loehr--Warrington conjecture regarding $\nabla^{m}s_{\lambda}$.  

Fix a partition $\lambda=(\lambda_1,\lambda_2,\dots,\lambda_{\ell})$ and let $s$ be the size of the \emph{Durfee square}, which is the maximal number such that $\lambda_s\geq s$ (if it does not exist, $s=0$). We define an \textit{adjustment} of $\lambda$ denoted by $\adj(\lambda)$ as
\begin{equation*}
    \adj(\lambda)=\sum_{i=1}^{s}(\lambda_i-i).
\end{equation*}
Consider a vector $(\lambda_1+\ell-1,\lambda_2+\ell-2,\dots,\lambda_{\ell},\ell,\ell+1,\dots,\lambda_1+\ell-1)$
 and define $v(\lambda)$ to be the vector obtained by sorting entries in the above vector in weakly increasing order.  We also define a \textit{pivot} of $\lambda$, denoted by $\piv(\lambda)$, to be a vector $(a_1,a_2,\dots,a_s)$ of length $s$ where $a_i$ is a number satisfying
\begin{equation*}
    v(\lambda)_{a_i}=v(\lambda)_{a_i+1}=\lambda_{s+1-i}+\ell-(s+1-i),
\end{equation*}
i.e., the indices where $v(\lambda)$ is non increasing. Finally, we define a matrix $J'(\lambda)=[w_1,w_2,\dots,w_{\lambda_1+\ell}]$ given by
\begin{align*}
    w_i=\begin{cases*}
        \h_{v(\lambda)_i} \quad \text{if $i\leq \ell-s$}\\
        \hb_{v(\lambda)_i} \quad \text{if $i> \ell-s$ and $i$ is an entry of $\piv(\lambda)$}\\
        \hhat_{v(\lambda)_i} \quad \text{if $i> \ell-s$ and $i$ is not an entry of $\piv(\lambda)$}.
    \end{cases*}
\end{align*}
In \cite[Equation (4.3)]{KO24}, the authors showed that $  \det(J(\lambda))=(-q)^{\adj(\lambda)}\det(J'(\lambda))$. Therefore we conclude
\begin{equation}\label{eq: Schur JT}
    s_{\lambda}[-MX^{m,n}]\cdot 1 =(-q)^{\adj(\lambda)}\Omega\left(\det(J'(\lambda))\cdot 1\right).
\end{equation}
In \cite{BHMPS25LW}, a (signed) positive expansion for $s_{\lambda}[-MX^{m,n}]\cdot 1$ was given. We do not know in general how to recover their result from our right-hand side of \eqref{eq: Schur JT}, however when $n=1$ it is directly connected to the Loehr--Warrington formula first stated in \cite{LW08}.

We state the Loehr--Warrington formula in our formulation. Fix a positive integer $m$ and let $n=1$. We identify an element $\pi \in \cPF_{m,1}$ as a pair $(i,j)\in \mathbb{Z}_{\geq0}\times\mathbb{Z}_{\geq1}$ where $i$ is the single entry of an area sequence of $P_{\pi}$ and $j$ is a label for a single north step. We write $(i,j) \succ (i',j')$ if and only if their corresponding elements  $\pi, \pi' \in \cPF_{m,1}$ satisfy  $\pi \succ \pi'$. 

Then we define the \textit{bottom} of $\lambda$, denoted by $\bo(\lambda)$, as a vector given by
\begin{equation*}
\bo(\lambda) = (s+1, s+2, \ldots, s+\lambda_1+\ell) - v(\lambda),
\end{equation*}
where the subtraction is performed element-wise. Let $N=\lambda_1+\ell$ and lastly we associate $D(\lambda)$ to the partition $\lambda$ as a diagram whose $j$-th column consists of cells $(i,j)$'s for $\bo(\lambda)_{N-j+1}\leq i\leq s$. Finally, define a set $\mathcal{T}(\lambda)$ consisting of fillings $T$ of $D(\lambda)$ with elements in $\mathbb{Z}_{\geq0}\times\mathbb{Z}_{\geq1}$ satisfying the following conditions: 
\begin{itemize}
    \item $T(i+1,j)\succ T(i,j)$ and $T(i,j)\nsucc  T(i,j-1)$
    \item for each $j> \ell-s$ if $j\in \piv(\lambda)$ we have $T(\bo(\lambda)_j,N-j+1)_1=0$
    \item for each $j> \ell-s$ if $j\notin \piv(\lambda)$ we have $T(\bo(\lambda)_j,N-j+1)_1>0$
\end{itemize}
here $T(i,j)\in \mathbb{Z}_{\geq0}\times\mathbb{Z}_{\geq1}$ is the filling for a cell $(i,j)$ and $T_1(i,j)$ is its first entry. We naturally regard $T\in \mathcal{T}(\lambda)$ as an element of
\begin{equation*}
    \CH_{s-\bo(\lambda)_N+1}\times \dots \times \CH_{s-\bo(\lambda)_1+1}
\end{equation*}
by regarding fillings in each $j$-th column of $T$ as an element of $\CH_{s-\bo(\lambda)_{N-j+1}+1}$. Under this identification we may define $\stat(T)$ and $z(T)$. The Loehr--Warrington conjecture in \cite{LW08} is the following identity:
\begin{equation}\label{eq: LW conjecture}
    \nabla^{m}s_{\lambda}=(-q)^{\adj(\lambda)}\sum_{T\in \mathcal{T}(\lambda)}q^{\stat(T)}\Omega(z(T)).
\end{equation}
It was first proved by \cite{BHMPS25LW} in a bigger generality for $s_\lambda[-MX^{m,n}]\cdot1$. Later for the case $m=1$, an elementary proof was provided in \cite{KO24}. In particular, the authors showed \cite[Proposition 4.6]{KO24}
\begin{equation}\label{yyyy}
    \det(J'(\lambda))\cdot1=\sum_{T\in \mathcal{T}(\lambda)}q^{\stat(T)}z(T)
\end{equation}
in the case $m=1$, using a combinatorial argument similar to that in \cite[Theorem 6.3]{SW16}.
The proof of \eqref{yyyy} trivially generalizes to arbitrary $m$.
Equations \eqref{eq: Schur JT} and \eqref{yyyy} imply \eqref{eq: LW conjecture}.

\subsection{Proof of Theorem~\ref{thm: main}}
We begin by defining a lifting map $\up$ on cyclic parking functions, which shifts the path's reference point relative to the line $my=nx$.

\begin{definition}
    Given a cyclic $(m,n)$-parking function $\pi \in \cPF_{m,n}$, we view $\pi$ as an infinite periodic labeled path. Let $(a,b)$ be the unique pair of integers satisfying $na - mb = 1$ with $0 \le b < n$. We define the map $\up(\pi) \in \cPF_{m,n}$ to be the cyclic $(m,n)$-parking function obtained by shifting $\pi$ by the vector $(-a,-b)$, such that the lattice point $(a,b)$ on the original path becomes the new origin.
\end{definition}

Geometrically, the point $(a,b)$ is the lattice point strictly below the line $my=nx$ that is closest to the line. It is straightforward to verify that the shifted labeled path $\up(\pi)$ remains weakly above the line $my=nx$, and thus $\up(\pi) \in \cPF_{m,n}$.

\begin{example}
    On the left of Figure~\ref{fig: up operator}, $\pi \in \cPF_{3,2}$ is
represented as an infinite periodic labeled path. The area sequence of
$P_{\pi}$ is $(0,0)$, and the labels of the north steps from the bottom are
$r_{1}$ and $r_{2}$. The closest lattice point strictly below the line
$3y = 2x$ is $(2,1)$, and shifting the entire picture by the vector
$(-2,-1)$ yields $\up(\pi)$, shown on the right.

\end{example}

\begin{figure}[ht]
    \centering
    \begin{tikzpicture}[scale=0.8]
        % --- Definitions ---
        \def\m{3} \def\n{2} \def\k{2}
        \def\M{6} \def\N{4}
    
        \colorlet{path0}{red!70!orange}
        \colorlet{path1}{blue!70!cyan}
        \colorlet{celltext}{gray!70!black}
    
        \tikzset{
            grid line/.style={gray!20, thin},
            axis line/.style={->, >=stealth, thick},
            diag line/.style={dashed, thick, orange!80},
            path line/.style={line width=2pt, cap=round, join=round},
            content node/.style={font=\sffamily\tiny, text=celltext},
            label node/.style={font=\small\bfseries}
        }
    
        % SHEET 0
        \begin{scope}[local bounding box=sheet0]
            \foreach \x in {-1,...,3} {
                \foreach \y in {0,1} {
                    \draw[grid line] (\x,\y) rectangle (\x+1,\y+1);
                    %\pgfmathsetmacro{\cont}{int(0 + \M*\y - \N*(\x+1))}
                    %\node[content node] at (\x+0.5, \y+0.5) {$\cont$};
                }
            }
            \draw[axis line] (-1.2,0) -- (4.2,0) node[right] {};
            \draw[axis line] (0,-0.5) -- (0,2.5) node[above] {};
            \draw[diag line] (-1, -0.66) -- (3.5, 2.33);
            \draw[path line, path0] (-2,-1)--(-2,0)--(0,0) -- (0,1) -- (1,1) -- (1,2) -- (3,2)--(3,3);
            \fill (2,1) circle (2pt);
            \node[label node] at (-0.3,0.5) {$r_1$};
             \node[label node] at (3-0.3,2.5) {$r_1$};
            \node[label node] at (0.7,1.5) {$r_2$};
            \node[label node] at (0.7-3,1.5-2) {$r_2$};
        \end{scope}
    
        % SHEET 1
        \begin{scope}[xshift=7cm, local bounding box=sheet1]
            \foreach \x in {-2,...,2} {
                \foreach \y in {0,1} {
                    \draw[grid line] (\x,\y) rectangle (\x+1,\y+1);
                    %\pgfmathsetmacro{\cont}{int(1 + \M*\y - \N*(\x+1))}
                    %\node[content node] at (\x+0.5, \y+0.5) {$\cont$};
                }
            }
            \draw[axis line] (-2.2,0) -- (3.2,0) node[right] {};
            \draw[axis line] (0,-0.5) -- (0,2.5) node[above] {};
            \draw[diag line] (-2, -1.33) -- (3, 2);
            \draw[path line, path1] (-2,-1)--(-2,0)--(-1,0) -- (-1,1) -- (1,1) -- (1,2) -- (2,2)--(2,3);
             \node[label node] at (-1.3,0.5) {$r_2$};
             \node[label node] at (0.7,1.5) {$r_1$};
            \node[label node] at (-1.3+3,0.5+2) {$r_2$};
             \node[label node] at (0.7-3,1.5-2) {$r_1$};
        \end{scope}
    \end{tikzpicture}
    \caption{$\pi$ and $\up(\pi)$.}
    \label{fig: up operator}
\end{figure}
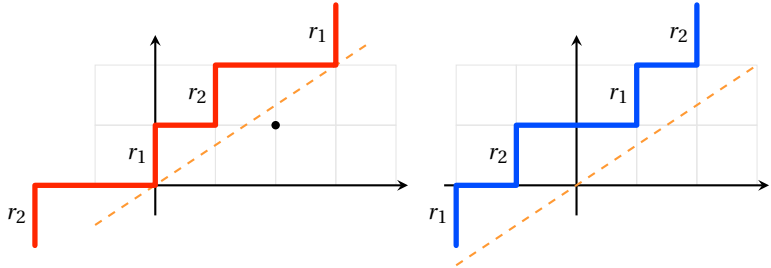

The following lemmas establish the basic properties of this map (proof left to the reader).

\begin{lem}\label{lem: cyc elementary no proof}
    The map $\up$ is a bijection from $\CH_1$ to $\hat{\CH}_1$ such that
    \begin{equation*}
        \area(P_\pi)+1=\area(P_{\up(\pi)}).
    \end{equation*}
\end{lem}

\begin{lem}\label{lem: rational cycling rule}
    For $(\pi^{(1)},\dots,\pi^{(k)}) \in \cPF_{m,n}^{k}$, we have
    \begin{equation*}
        \stat(\pi^{(1)},\dots,\pi^{(k)}) = \stat(\up(\pi^{(k)}), \pi^{(1)}, \dots, \pi^{(k-1)}).
    \end{equation*}
\end{lem}
\begin{proof}
    Pairs contributing to $\pdinv$ and $\ldinv$ are all preserved.
\end{proof}

It is straightforward that $\pi \prec \tau$ if and only if $\up(\pi) \prec \up(\tau)$. Therefore we can construct a bijection between $\CH_a$ and $\hat{\CH}_a$ by sending each $(\pi^{(1)},\cdots,\pi^{(a)})$ to $(\up(\pi^{(1)}),\cdots,\up(\pi^{(a)}))$. Combining with Lemmas~\ref{lem: cyc elementary no proof} and \ref{lem: rational cycling rule} we conclude the following: for any monomial $v$ consisting of $\h_i$, $\hhat_i$ or $\hb_i$ we have
\begin{equation*}
    \Omega\left(v \hhat_a \cdot 1\right)=t^{a} \Omega\left(\h_a v \cdot 1\right).
\end{equation*}
Specifically, we derive $\Omega\left(\h_{\lambda}\hhat_{\mu}\cdot 1\right) = t^{|\mu|}\Omega\left(\h_{(\lambda,\mu)} \cdot 1\right)$. Here, $(\lambda,\mu)$ is an integer vector given by a concatenation. We caution the reader that as operators, $\h_{\lambda}\hhat_{\mu} \neq t^{|\mu|}\h_{(\lambda,\mu)}$, but equality holds only when applied to $1$ and after specializing each $z_{i,j,k}$ to $t^{j}x_k$. Combining this fact with Proposition~\ref{prop: uniform description}, we obtain:

\begin{corollary}\label{cor: final}
    For any symmetric function $f$, we have
    \begin{equation*}
        f[-MX^{m,n}]\cdot 1 = \Omega\left((\sum_{\lambda} c_{\lambda} \h_{\lambda}) \cdot 1\right),
    \end{equation*}
    where $c_{\lambda}$ is the coefficient of $h_{\lambda}[Y]$ in the expansion of $f[(1-t)Y]$.
\end{corollary}

\begin{proof}[Proof of Theorem~\ref{thm: main}]
    We have
    \begin{equation*}
        e_{(1^k)}[(1-t)X] = \left( (1-t)e_{1}[X] \right)^k = (1-t)^k (h_1[X])^k = (1-t)^k h_{(1^k)}[X].
    \end{equation*}
    Applying Corollary~\ref{cor: final}, we immediately obtain
    \begin{equation*}
        e_{(1^k)}[-MX^{m,n}]\cdot 1 = (1-t)^k \Omega\left(\h_{(1^k)} \cdot 1\right)=(1-t)^k\sum_{\pi^\bullet \in\cPF^k_{m,n}} q^{\stat(\pi^\bullet)} t^{\area(P_{\pi^\bullet})} x^{f_{\pi^{\bullet}}}. 
    \end{equation*}
\end{proof}
Since we have
\(    \Omega\left(\h_{(1^k)} \cdot 1\right)=\Omega\left(\h_{(1^{k-1})}(\hhat_1+\hb_1) \cdot 1\right)=t\Omega\left(\h_{(1^{k})}\cdot1\right) +\Omega\left((\h_{(1^{k-1})}\hb_1) \cdot 1\right),\) we also deduce
\begin{align}
 \nonumber   e_{(1^k)}[-MX^{m,n}]\cdot 1&=(1-t)^{k-1}\Omega\left((\h_{(1^{k-1})}\hb_1) \cdot 1\right)\\&=(1-t)^{k-1}\sum_{\substack{\pi^\bullet \in\cPF^k_{m,n}\\ \aseq(P_{\pi^{(1)}})_1=0}} q^{\stat(\pi^\bullet)} t^{\area(P_{\pi^\bullet})} x^{f_{\pi^{\bullet}}}.\label{eq: Wilson compact}
\end{align}

\subsection{Further study}
Proposition~\ref{prop: uniform description} and Corollary~\ref{cor: final} provide a powerful tool for studying $f[-MX^{m,n}]\cdot 1$ for various symmetric functions $f$. In a sequel paper, we will present a formula for $e_{\lambda}[-X^{m,n}]\cdot 1$ that generalizes \cite[Theorem~3.1]{CM25}, which corresponds to the special case $n=1$.

The following problems are natural directions for future research:
\begin{enumerate}
    \item Derive a (signed) positive expression for $\det(J(\lambda))\cdot 1$, providing an alternative formula for the $(m,n)$-version of the Loehr--Warrington formula \cite{BHMPS25LW}. Note that the case $n=1$ was treated in Section~\ref{subsec: alternative proof of LW}.
    \item Find a positive expression for
    \[
        \frac{e_{\lambda}[-MX^{m,n}]\cdot 1}{(1-t)^{\ell(\lambda)-1}}.
    \]
    According to \cite{CGHM24,GL23}, such a formula would yield the shuffle theorem for a link whose components are certain cables of a torus knot.
\end{enumerate}

\bibliographystyle{alpha}  
\bibliography{SFTL.bib} 

\end{document}